\documentclass[a4paper,12pt]{amsart}
\usepackage{amsmath}
\usepackage{amssymb}
\usepackage{amscd}
\usepackage{mathrsfs}
\usepackage{url}
\usepackage[all]{xy}
\usepackage[dvipdfmx]{graphicx}
\usepackage{graphics}
\usepackage{tikz}
\everymath{\displaystyle}
\theoremstyle{plain} 
\newtheorem{theorem}{\indent\sc Theorem}[section]
\newtheorem{lemma}[theorem]{\indent\sc Lemma}
\newtheorem{corollary}[theorem]{\indent\sc Corollary}
\newtheorem{proposition}[theorem]{\indent\sc Proposition}

\theoremstyle{definition} 
\newtheorem{definition}[theorem]{\indent\sc Definition}
\newtheorem{remark}[theorem]{\indent\sc Remark}
\newtheorem{example}[theorem]{\indent\sc Example}

\newtheorem{question}[theorem]{\indent\sc Question}

\begin{document}

\pagestyle{plain}
\thispagestyle{plain}

\title[A gluing construction of K3 surfaces]
{A gluing construction of K3 surfaces}

\author[Takayuki Koike and Takato Uehara]{Takayuki Koike$^{1}$ and Takato Uehara$^{2}$}
\address{ 
$^{1}$ Department of Mathematics \\
Graduate School of Science \\
Osaka Metropolitan University \\
3-3-138 Sugimoto \\
Osaka 558-8585 \\
Japan 
}
\email{tkoike@omu.ac.jp}
\address{
$^{2}$ Department of Mathematics \\
Faculty of Science \\
Okayama University \\
1-1-1, Tsushimanaka \\
Okayama, 700-8530 \\
Japan
}
\email{takaue@okayama-u.ac.jp}
\renewcommand{\thefootnote}{\fnsymbol{footnote}}
\renewcommand{\thefootnote}{\arabic{footnote}}

\begin{abstract}
We develop a new method for constructing K3 surfaces.  
We construct such a K3 surface $X$ by patching two open complex surfaces obtained as the complements of tubular neighborhoods of elliptic curves embedded in blow-ups of the projective planes at general nine points. 
Our construction has $19$ complex dimensional degrees of freedom. 
For general parameters, the K3 surface $X$ is neither Kummer nor projective. 
By the argument based on the concrete computation of the period map, we also investigate which points in the period domain correspond to K3 surfaces obtained by such construction. 
\end{abstract}

\maketitle

\section{Introduction}\label{section:introduction}

The aim of this paper is to develop a new method for constructing K3 surfaces. 
With the method in hand, we state one of main results as follows: 

\begin{theorem}\label{thm:main}
There exists a holomorphic deformation family $\pi \colon\mathcal{X}\to B$ of K3 surfaces over a $19$-dimensional complex manifold $B$ such that the following property holds: \\
$(i)$ The corresponding map from $B$ to the period domain is locally injective. \\
$(ii)$ Each fiber $X$ admits a real $1$ parameter family of compact Levi-flat hypersurfaces $\{H_t\}_{t\in I}$ of $C^\omega$ class such that, for each $t$ of the interval $I$, the real hypersurface $H_t$ is $C^\omega$-diffeomorphic to a real $3$-dimensional torus $S^1\times S^1\times S^1$. \\
$(iii)$ Each leaf of the Levi-flat foliation of $H_t$ is biholomorphic to either $\mathbb{C}$ or $\mathbb{C}^*:=\mathbb{C}\setminus\{0\}$ and is dense in $H_t$. \\
Moreover, general fiber $X$ (i.e. each fiber $X$ of an open dense subset of $B$) is a K3 surface with the Picard number $0$, and hence is neither projective nor Kummer. 
\end{theorem}


Here a real hypersurface $H$ in a complex manifold $X$ is said to be Levi-flat if it admits a foliation of real codimension $1$ whose leaves are complex manifolds holomorphically immersed into $X$. 
We will construct a family $\pi \colon\mathcal{X}\to B$ of K3 surfaces each of whose fibers $X$ has an open complex submanifold $V\subset X$ with the following property: there exists an elliptic curve $C$, a non-torsion flat line bundle $N\to C$ (i.e. a topologically trivial holomorphic line bundle $N\in {\rm Pic}^0(C)$ such that $N^n:=N^{\otimes n}$ is not holomorphically trivial for any positive integer $n$), and two positive numbers $a<b$ such that $V$ is biholomorphic to $\{x\in N\mid a<|x|_h<b\}$, where $h$ is a fiber metric on $N$ with zero curvature. 
The Levi-flat hypersurfaces $\{H_t\}_{t\in I}$ in Theorem \ref{thm:main} are given by the hypersurfaces corresponding to $\{x\in N\mid |x|_h=t\}$ for each $t\in I:=(a, b)\subset \mathbb{R}$. 
Note that Kummer surfaces with such an open subset $V$ can be constructed in a simple manner. 
Actually, the Kummer surface constructed from an abelian surface $A$ has such $V$ as an open complex submanifold if $A$ includes $V$. 
Note also that each leaf of the Levi-flat hypersurface $\{x\in N\mid |x|_h=t\}$ is biholomorphic to either $\mathbb{C}$ or $\mathbb{C}^*$ for each $a<t<b$. 
By considering the universal covering of a leaf, we have the following corollary: 

\begin{corollary}\label{cor:main}
There exists a K3 surface $X$ which is neither projective nor Kummer, and which admits a holomorphic map $f\colon \mathbb{C}\to X$ such that the Euclidean closure of the image $f(\mathbb{C})$ is a compact real hypersurface of $X$. 
In particular, the Zariski closure of $f(\mathbb{C})$ coincides with $X$ (i.e. $f(\mathbb{C})$ cannot be included in any proper analytic subvariety), whereas the Euclidean closure of $f(\mathbb{C})$ is a proper subset of $X$. 
\end{corollary}

We will construct a K3 surface $X$ by patching two open complex surfaces obtained as the complements of tubular neighborhoods of elliptic curves embedded in blow-ups of the projective plane $\mathbb{P}^2$ at general nine points. 
The outline of the construction is as follows: 
Let $C_0^+$ and $C_0^-$ be two smooth elliptic curves in $\mathbb{P}^2$ biholomorphic to each other, and $Z^\pm:=\{p_1^\pm, p_2^\pm, \dots, p_9^\pm\}\subset C_0^{\pm}$ be general nine points.  
Moreover let $S^{\pm}$ be the blow-ups of $\mathbb{P}^2$ at the nine points $Z^\pm$, and $C^\pm\subset S^\pm$ be the strict transforms of $C_0^{\pm}$. 
Denote by $M^\pm$ the complements of tubular neighborhoods of $C^\pm$ in $S^\pm$. 
We construct $X$ by patching $M^+$ and $M^-$. 
In order to patch them holomorphically, namely, $M^{\pm}$ become holomorphically embedded open complex submanifolds of $X$, 
one needs to choose $Z^\pm$ in a suitable manner. 
To this end, we give the following definition: 

\begin{definition}[{\cite[\S 4.1]{U}}] \label{def:Dioph}
A flat line bundle $L \in {\rm Pic}^0(C)$ on an elliptic curve $C$ is said to satisfy the Diophantine condition if $-\log d(\mathbb{I}_C, L^{n})=O(\log n)$ as $n\to \infty$, where $d$ is an invariant distance on ${\rm Pic}^0(C)$ (in the sense of {\cite[\S 4.1]{U}}, which is equivalent to the Euclidean distance on ${\rm Pic}^0(C)$ as a torus, see \S \ref{section:distance_Pic0}) and $\mathbb{I}_C$ is the holomorphically trivial line bundle on $C$. 
The condition is independent of the choice of an invariant distance $d$. 
\end{definition}

\begin{remark}\label{rem:fullmeas}
It is clear from the definition that any flat line bundle is non-torsion if it is Diophantine. 
The set $\mathcal{D}$ of all elements of ${\rm Pic}^0(C)$ which satisfy the Diophantine condition is a subset of ${\rm Pic}^0(C)$ with full Lebesgue measure, since 
\[
{\rm Pic}^0(C)\setminus \mathcal{D} = \bigcap_{m, \ell=1}^\infty\bigcup_{n=1}^\infty\{L\in {\rm Pic}^0(C)\mid d(\mathbb{I}_C, L^n)< n^{-\ell}/m\}
\]
holds and the measure of the set $\bigcup_{n=1}^\infty\{L\in {\rm Pic}^0(C)\mid d(\mathbb{I}_C, L^n)< n^{-\ell}/m\}$ is less than or equal to 
\[
\sum_{n=1}^\infty n^2\cdot \frac{1}{n^2}\cdot \frac{\pi}{m^2\cdot n^{2\ell}}=\frac{\pi}{m^2}\zeta(2\ell), 
\]
which approaches to zero as $m, \ell\to \infty$. 
More precisely, the set can be expressed as the union of a countable number of nowhere dense Euclidean closed subsets of ${\rm Pic}^0(C)$, since 
\[
\mathcal{D} = \bigcup_{m, \ell=1}^\infty\bigcap_{n=1}^\infty\{L\in {\rm Pic}^0(C)\mid d(\mathbb{I}_C, L^n)\geq n^{-\ell}/m\}. 
\]
\end{remark}

Under the setting above, we have the following theorem:

\begin{theorem}\label{thm:constr_of_K3}
Under the above notations, assume that the normal bundles $N_{\pm}:=N_{C^{\pm}/S^{\pm}}$, which are clearly elements of ${\rm Pic}^0(C^\pm)$ by construction, are dual to each other via a biholomorphism $g\colon C^+\to C^-$, 
that is, $g^*N_-\cong N_+^{-1}$, and that $N_\pm$ satisfy the Diophantine condition. 
Then we can patch $M^+$ and $M^-$ holomorphically to yield a K3 surface $X$. 
\end{theorem}
%


Remark \ref{rem:fullmeas} says that for almost all nine point configurations $Z^{\pm} \subset C^{\pm}_0$, the corresponding normal bundles $N_{\pm}$ satisfy the Diophantine condition. 
In such sense, our gluing construction has some degree of freedom. 
Indeed, in the gluing construction of K3 surfaces, there are some parameters coming from, for example, the choice of the elliptic curves $C_0^\pm$, nine points configurations $Z^{\pm}$, the choice of the radii of the tabs of gluing, and some patching parameters, which cause $19$-dimensional deformation as stated in Theorem \ref{thm:main}. 
By concretely computing the period integral, we investigate the relation between these parameters which appear in the construction and the period map (see \S \ref{section:constr_cycles}, \ref{section:deformation_of_X}, and \ref{section:realizability}). 
For this purpose, we also construct a marking, namely, $22$ generators of the second homology group $H_2(X, \mathbb{Z})$ of our K3 surface $X$, which will be denoted by 
\[
A_{\alpha\beta}, A_{\beta\gamma}, A_{\gamma\alpha}, B_\alpha, B_\beta, B_\gamma, C_{12}^+, C_{23}^+, \dots, C_{78}^+, C_{678}^+, C_{12}^-, C_{23}^-, \dots, C_{78}^-, C_{678}^-.
\] 
The homology group $H_2(X, \mathbb{Z})$ equipped with the intersection form can  be expressed as 
\[
H_2(X, \mathbb{Z})=\Pi_{3,19} \cong \langle A_{\alpha\beta}, B_\gamma\rangle\oplus \langle A_{\beta\gamma}, B_\alpha\rangle\oplus \langle A_{\gamma\alpha}, B_\beta\rangle\oplus \langle C^+_\bullet\rangle\oplus \langle C_\bullet^-\rangle, 
\]
where 
$\langle A_{\alpha\beta}, B_\gamma\rangle\cong \langle A_{\beta\gamma}, B_\alpha\rangle\cong \langle A_{\gamma\alpha}, B_\beta\rangle \cong U$ are 
the even unimodular lattice of rank $2$ with $(A_{\bullet}. A_{\bullet})=0, (A_{\bullet}. B_{\bullet})=1, (B_{\bullet}. B_{\bullet})=-2$, and $\langle C_\bullet^\pm\rangle \cong E_8(-1)$,  where $E_8$ is the lattice corresponding to the Dynkin diagram given in Figure \ref{figure:dynkin_1}. 


\begin{center}
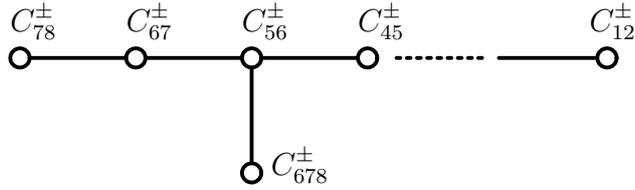
\begin{figure}[h]
\unitlength 0.1in
\begin{picture}( 31.4000,  9.2000)(  3.5000,-10.5000)
%
\special{pn 20}%
\special{pa 450 400}%
\special{pa 950 400}%
\special{fp}%
%
\special{pn 20}%
\special{pa 1050 400}%
\special{pa 1550 400}%
\special{fp}%
%
\special{pn 20}%
\special{pa 1650 400}%
\special{pa 2150 400}%
\special{fp}%
%
\special{pn 20}%
\special{pa 2350 400}%
\special{pa 2800 400}%
\special{dt 0.054}%
%
\special{pn 20}%
\special{pa 2880 400}%
\special{pa 3380 400}%
\special{fp}%
%
\special{pn 20}%
\special{pa 1600 450}%
\special{pa 1600 950}%
\special{fp}%
\put(17.0000,-10.5000){\makebox(0,0)[lb]{$C_{678}^\pm$}}%
\put(3.5000,-3.0000){\makebox(0,0)[lb]{$C_{78}^\pm$}}%
\put(9.5000,-3.0000){\makebox(0,0)[lb]{$C_{67}^\pm$}}%
\put(15.5000,-3.0000){\makebox(0,0)[lb]{$C_{56}^\pm$}}%
\put(21.5000,-3.0000){\makebox(0,0)[lb]{$C_{45}^\pm$}}%
\put(33.5000,-3.0000){\makebox(0,0)[lb]{$C_{12}^\pm$}}%
%
\special{pn 20}%
\special{ar 400 400 50 50  0.0000000 6.2831853}%
%
\special{pn 20}%
\special{ar 1000 400 50 50  0.0000000 6.2831853}%
%
\special{pn 20}%
\special{ar 1600 400 50 50  0.0000000 6.2831853}%
%
\special{pn 20}%
\special{ar 1600 1000 50 50  0.0000000 6.2831853}%
%
\special{pn 20}%
\special{ar 2200 400 50 50  0.0000000 6.2831853}%
%
\special{pn 20}%
\special{ar 3440 400 50 50  0.0000000 6.2831853}%
\end{picture}%
\caption{Dynkin diagram}
\label{figure:dynkin_1}
\end{figure}
\end{center}

A pair $(p, q)$ of real numbers is said to satisfy the Diophantine condition if there exist positive numbers $A$ and $\alpha$ such that $\textstyle\min_{(a, b)\in\mathbb{Z}^2}(|np-a|+|nq-b|)\geq A\cdot n^{-\alpha}$ holds for each positive integer $n$. 
By using the notations above, we have the following: 

\begin{theorem}\label{thm:over_B_pq}
For any pair $(p, q)$ of real numbers satisfying the Diophantine condition, put $v=v_{(p, q)}:=A_{\alpha\beta}+p\cdot A_{\beta\gamma}-q\cdot A_{\gamma\alpha}$. 
Then there exists an open subset $\Xi_{(p, q)}$ of $v^{\perp}:=\{ \xi \in \mathcal{D}_{\rm Period} \mid (\xi. v)=0 \}$ and a proper holomorphic submersion $\pi \colon \mathcal{X}\to \Xi_{(p, q)}$ such that 
each fiber is a K3 surface constructed in our gluing method whose period map yields the identity map $\Xi_{(p, q)}\to \Xi_{(p, q)}$, where $\mathcal{D}_{\rm Period}:=\{\xi\in\mathbb{P}(\Pi_{3, 19}\otimes \mathbb{C})\mid (\xi. \xi)=0, (\xi. \overline{\xi})>0 \}$ is the period domain. 
\end{theorem}

More precise expression of the open subset $\Xi_{(p, q)}$ will be explained in Corollary \ref{cor:realizability_main}. 
Note that the normal vector $v=v_{(p, q)}$ of the subset satisfies $(v. v)=0$. 
Note also that one of the statements of Theorem \ref{thm:main} is a simple consequence of Theorem \ref{thm:over_B_pq}. 

Diophantine condition is needed since we use the following Arnol'd's theorem in our gluing method. 

\begin{theorem}[{\cite[Theorem 4.3.1]{A}}]\label{thm:Arnol'd_original}
Let $S$ be a non-singular complex surface and $C\subset S$ be a holomorphically embedded elliptic curve. 
Assume that $N_{C/S}$ is flat and satisfies the Diophantine condition. 
Then $C$ admits a holomorphic tubular neighborhood $W$ in $S$, 
i.e. there exists a tubular neighborhood $W$ of $C$ in $S$ which is biholomorphic to a neighborhood of the zero section in $N_{C/S}$. 
\end{theorem}

When the normal bundle $N_{C/S}$ does not satisfy the Diophantine condition, $C$ does not necessarily admit a holomorphic tubular neighborhood in general. 
Indeed, it follows from Ueda's classification \cite[\S 5]{U} that Ueda's obstruction class $u_n(C, S)\in H^1(C, N_{C/S}^{-n})$ needs to vanish for all $n\geq 1$ if there exists such a neighborhood $W$ as in Theorem \ref{thm:Arnol'd_original}. 
Ueda also showed the existence of an example of $(C, S)$ with no such neighborhood $W$ for which all the obstruction classes $u_n(C, S)$'s vanish \cite[\S 5.4]{U}. 
In this context, the following is one of the most interesting questions. 

\begin{question}\label{ques:nbhd}
Fix a smooth elliptic curve $C_0\subset\mathbb{P}^2$ and take general nine points $Z:=\{p_1, p_2, \dots, p_9\}\subset C_0$. 
Let $S:={\rm Bl}_Z\mathbb{P}^2$ be the blow-up of $\mathbb{P}^2$ at $Z$ and $C\subset S$ be the strict transform of $C_0$. 
Is there a nine points configuration $Z$ such that $N_{C/S}$ is non-torsion and $C$ does not admit a holomorphic tubular neighborhood?
\end{question}

Note that, in Question \ref{ques:nbhd}, all the obstruction classes $u_n(C, S)$'s vanish for any nine points configuration $Z$ (see \cite[\S 6]{Ne}). 
See  also \cite{B}, in which the existence condition of a suitable neighborhood of $C$ is reworded into the existence condition of a K\"ahler metric on $S$ with semi-positive Ricci curvature in a generic configuration on the choice of nine points, which is one of the biggest motivations of the present paper. 
As is mentioned above, we apply Arnol'd's theorem (Theorem \ref{thm:Arnol'd_original}) to show Theorem \ref{thm:constr_of_K3}. 
Such a gluing construction technique based on Arnol'd-type theorem is also used in \cite{T} to study complex structures on $S^3\times S^3$. 
Note that our construction may be similar to that studied in \cite{D} (see \cite[Example 5.1]{D}, see also \cite{Kov}). 
A main difference between our construction and Doi's one is the point that one need {\it not} deform the complex structures of $M^\pm$ to patch together in our construction any more (i.e. in our construction, the complex structures are explicitly described). 
In particular, one can regard $M^\pm$ as holomorphically embedded open complex submanifolds of the resulting K3 surface $X$ in our construction, which enables us to reduce the study of some submanifolds of $X$ to that of the blow-up $S^\pm$ of $\mathbb{P}^2$. 
In the last part of this paper, we will see some concrete examples. 
Especially in \S \ref{section:type_II_degeneration}, we observe that suitable subfamilies of $\pi \colon \mathcal{X}\to \Xi_{(p, q)}$ as in Theorem \ref{thm:over_B_pq} can be regarded as the family with type II degeneration whose central fiber is the singular K3 surface with normal crossing singularity obtained by patching $S^+$ and $S^-$ along $C^\pm$. 
In this sense, we may explain that our gluing construction gives a concrete expression of some direction of the smoothing of such singular K3 surfaces (cf.  \cite{F}, \cite{KN}). 

The organization of the paper is as follows. 
In \S 2, we show Theorem \ref{thm:constr_of_K3} and explain the details of the construction of $X$. 
In \S 3, we explain the construction of the $2$-cycles $A_\bullet$, $B_\bullet$, and $C_\bullet^\pm$, and compute the integrals of the holomorphic $2$-form along them. 
In \S 4, we consider the deformation of K3 surfaces which are constructed in our gluing method. Here we construct a deformation family $\pi\colon \mathcal{X}\to B$ of such K3 surfaces with injective Kodaira-Spencer map over a $19$-dimensional polydisc $B$. 
In \S 5, we show Theorem \ref{thm:main} and Corollary \ref{cor:main}. 
In \S 6, we investigate the subset of the Period domain which corresponds to the set of all K3 surfaces which are constructed in our method. Here we show Theorem \ref{thm:over_B_pq}. 
In \S 7, we give some examples. 
In \S 8, we show a relative variant of Arnol'd's Theorem, which is needed in \S 4. 
\vskip3mm
{\bf Acknowledgment. } 
The authors would like to give heartful thanks to Prof. Tetsuo Ueda, Prof. Ken-ichi Yoshikawa and Prof. Yuji Odaka whose enormous supports and insightful comments were invaluable during the course of their study. 
Especially, a simple proof of \cite[Lemma 4]{U} due to Prof. Tetsuo Ueda plays an important role in \S \ref{section:more_gen_rel_arnol'd_thm}. 
They also thanks Prof. Noboru Ogawa for helpful comments and warm encouragements. 
The first author is supported by the Grant-in-Aid for Scientific Research (KAKENHI No.28-4196) and the Grant-in-Aid for JSPS fellows, JSPS KAKENHI Grant Number JP18H05834, and by a program: Leading Initiative for Excellent Young Researchers (LEADER, No. J171000201). 
The second author is supported by JSPS KAKENHI Grant Number JP16K17617. 

\section{Construction of $X$ and Proof of Theorem \ref{thm:constr_of_K3}}\label{section:construction}

First we explain our notation on the blow-ups $S^\pm$ of $\mathbb{P}^2$ and related objects. 
Fix two smooth elliptic curves $C_0^\pm\subset\mathbb{P}^2$ such that there exists an isomorphism $g\colon C_0^+\cong C_0^-$. 
Take general nine points $Z^\pm:=\{p_1^\pm, p_2^\pm, \dots, p_9^\pm\}\subset C_0^\pm$. 
In what follows, we always assume 
that the line bundle 
$N_0:=\mathcal{O}_{\mathbb{P}^2}(3)|_{C_0^+}\otimes\mathcal{O}_{C_0^+}(-p_1^+-p_2^+-\cdots-p_9^+)\in {\rm Pic}^0(C_0^+)$ 
satisfies the Diophantine condition. 
Note that such a nine points configuration actually exists, because almost every $Z^+\in (C^+)^9\subset(\mathbb{P}^2)^9$ satisfies these conditions in the sense of Lebesgue measure. 
Assume also that $\mathcal{O}_{\mathbb{P}^2}(3)|_{C_0^-}\otimes\mathcal{O}_{C_0^-}(-p_1^--p_2^--\cdots-p_9^-)$ is isomorphic to $N_0^{-1}$ via $g$. 
Let $S^\pm:={\rm Bl}_{Z^\pm}\mathbb{P}^2$ be the blow-up of $\mathbb{P}^2$ at $Z^\pm$ and $C^\pm\subset S^\pm$ be the strict transform of $C_0^\pm$. 

\subsubsection*{Proof of Theorem \ref{thm:constr_of_K3}}

As $N_\pm:=N_{C^\pm/S^\pm} (\cong N_0^{\pm 1})$ satisfies the Diophantine condition, it follows from Theorem \ref{thm:Arnol'd_original} that there exists a tubular neighborhood $W^\pm$ of $C^\pm$ such that $W^\pm$ is biholomorphic to a neighborhood of the zero section in $N_\pm:=N_{C^\pm/S^\pm}$. 
Therefore, by shrinking $W^\pm$ and considering the pull-back of an open covering $\{U_j^\pm\}$ of $C^\pm$ by the projection $W^\pm\to C^\pm$, one can take an open covering $\{W_j^\pm\}$ of $W^\pm$ and a coordinate system $(z_j^\pm, w_j^\pm)$ of each $W_j^\pm$ which satisfies the following five conditions: 
$(i)$ $W_j^\pm$ is biholomorphic to $U_j^\pm\times\Delta_{R^\pm}$, where $\Delta_{R^\pm}:=\{w\in\mathbb{C}\mid |w|<R^\pm\}$, 
$(ii)$ $W_{jk}^\pm$ is biholomorphic to $U_{jk}^\pm\times\Delta_{R^\pm}$, where $W_{jk}^\pm:=W_j^\pm\cap W_k^\pm$ and $U_{jk}^\pm:=U_j^\pm\cap U_k^\pm$, 
$(iii)$ $z_j^\pm$ can be regarded as a coordinate of $U_j^\pm$ and $w_j^\pm$ can be regarded as a coordinate of $\Delta_{R^\pm}$, and 
$(iv)$ $(z_k^+, w_k^+)=(z_j^++A_{kj}, t_{kj}^+\cdot w_j^+)$ holds on $W_{jk}^+$, where $A_{kj}\in\mathbb{C}$ and $t_{kj}^+\in U(1):=\{t\in\mathbb{C}\mid |t|=1\}$. 
$(v)$ $(z_k^-, w_k^-)=(z_j^-+A_{kj}, t_{kj}^-\cdot w_j^-)$ holds on $W_{jk}^-$, where $A_{kj}$ is the same one as in the previous condition and $t_{kj}^-\in U(1):=\{t\in\mathbb{C}\mid |t|=1\}$. 
Note that one can use any positive number for the constant $R^\pm>0$ by rescaling $w_j^\pm$'s. 
In what follows, we always assume that $R^\pm>1$. 
Here we used the fact that any topologically trivial line bundle on a compact K\"ahler manifold is flat to have that $t_{jk}^\pm\in {\rm U}(1)$ (see \cite[\S 1]{U}). 
Note also that we use $g(U_j^+)$ as $U_j^-$ and the pull-back $(g^{-1})^*z_j^+:=z_j^+\circ g^{-1}$ as $z_j^-$. 
By using these, one can easily deduce from $g^*N_-=N_+^{-1}$ that $t_{jk}^+t_{jk}^-=1$ holds. 
In what follows, we denote $t_{jk}^+$ by $t_{jk}$ (and thus $t_{kj}^-=t_{kj}^{-1}$) simply. 


Set $V^\pm:=\bigcup_{j}V^\pm_j$, where $V^+_j:=\{(z_j^+, w_j^+)\in W_j^+\mid 1/R^-<|w_j^+|<R^+\}$ and $V^-_j:=\{(z_j^-, w_j^-)\in W_j^-\mid 1/R^+<|w_j^-|<R^-\}$. 
In other words, $V^\pm$ is defined as the preimage of the interval $(1/R^-, R^+)$ and $(1/R^+, R^-)$, respectively, by the map $\Phi_\pm\colon W^\pm\to \mathbb{R}_{\geq 0}$ which is defined by $\Phi_\pm(z_j^\pm, w_j^\pm):=|w_j^\pm|$ on each $W_j^\pm$ (Note that it is well-defined since $|t_{jk}|\equiv 1$). 
Denote by $M^\pm$ the complement $S^\pm\setminus \Phi_\pm([0, 1/R^\mp])$. 
Define a holomorphic map $f\colon V^+\to V^-$ by 
\[
f|_{V^+_j}\colon(z_j^+, w_j^+)\mapsto (z_j^-(z_j^+, w_j^+), w_j^-(z_j^+, w_j^+)):=\left(g^{-1}(z_j^+), (w_j^+)^{-1}\right)\in V^-_j
\] 
on each $V_j^+$. 
By identifying $V^+$ and $V^-$ via this map $f$, we can patch $M^+$ and $M^-$ to define a compact complex surface $X$. 

We denote by $V$ the open subset of $X$ which comes from $V^\pm$. 
In what follows, we regard $M^\pm$ as open subsets of $X$. 
Note that $V=M^+\cap M^-$ holds as a subset of $X$, and this $V$ satisfies the conditions as in \S 1. 
Note also that, for each $t\in (1/R^-, R^+)$, the inverse image $H_t=\Phi_+^{-1}(t)$ is a compact Levi-flat of $V$ (and thus of $X$). 
As the foliation structure of $H_t$ is the one induced from the Chern connection of the flat metric of $N_\pm$ and $N_\pm$ is non-torsion element of ${\rm Pic}^0(C^\pm)$ by the Diophantine assumption (recall the explicit description of $H_t$ we made just before Corollary \ref{cor:main} in \S 1), it follows that any leaf of $H_t$ is dense in $H_t$. 

In what follows, we denote by $V_j$ the set $V_j^+$ when we regard it as a subset of $V$, and by $\Phi$ the map $\Phi^+$ when we regard it as the one defined on $V$. 

\begin{proposition}\label{prop:K3}
$X$ is a K3 surface with a global holomorphic $2$-form $\sigma$ with 
\[
\sigma|_{V_j}=dz_j^+\wedge \frac{dw_j^+}{w_j^+}
\]
on each $V_j\subset X$. 
\end{proposition}

\begin{proof}
It easily follows from Mayer--Vietoris sequence associated to the open covering $\{M^\pm, W^\pm\}$ of $S^\pm$ that $H_1(M^\pm, \mathbb{Z})=0$. 
Again by Mayer--Vietoris sequence associated to the open covering $\{M^+, M^-\}$ of $X$, we have that $H_1(X, \mathbb{Z})=0$. 
Therefore it is sufficient to show the existence of a nowhere vanishing global holomorphic $2$-form $\sigma$ with $\sigma|_{V_j}=dz^+_j\wedge \frac{dw_j^+}{w_j^+}$ on each $V_j\subset X$. 
As $-C^\pm$ is the canonical divisor $K_{S^\pm}$, $S^\pm$ admits a global meromorphic $2$-form $\eta^\pm$ with no zeros and with poles only along $C^\pm$. 
Define a nowhere vanishing holomorphic function $F_j^\pm$ on $V_j^\pm$ by
\[
F_j^\pm:=\frac{\eta^\pm}{dz_j^\pm\wedge \frac{dw_j^\pm}{w_j^\pm}}. 
\]
Then the functions $\{(V_j^\pm, F_j^\pm)\}$ glue up to define a holomorphic function $F^\pm\colon V^\pm\to \mathbb{C}$. 
By the following Lemma \ref{lem:W*}, we have that $F^\pm$ is a constant function $F^\pm\equiv A^\pm\in \mathbb{C}^*$. 
Therefore it follows that $\eta^\pm|_{V_j^\pm}=A^\pm\cdot dz_j^\pm\wedge \frac{dw_j^\pm}{w_j^\pm}$. 
As $f^*\frac{dw_j^-}{w_j^-}=-\frac{dw_j^+}{w_j^+}$, we have that $(A^+)^{-1}\cdot \eta^+|_{M^+}$ and $-(A^-)^{-1}\cdot \eta^-|_{M^-}$ glue up to define a nowhere vanishing $2$-form $\sigma$ on $X$, which shows the proposition. 
\end{proof}

Through this paper, we use the notation $\eta^\pm$ in the proof above. 
In what follows we assume that $A^\pm=\pm 1$ by scaling. 
Therefore, $\sigma$ is obtained by patching $\eta^+|_{M^+}$ and $\eta^-|_{M^-}$. 

\begin{lemma}\label{lem:W*}
$H^0(V, \mathcal{O}_{V})=\mathbb{C}$. 
\end{lemma}

\begin{proof}
Let $F\colon V\to \mathbb{C}$ be a holomorphic function. 
Take a real number $t$ with $1/R^-<t<R^+$ and a point $x_t\in H_t=\Phi^{-1}(t)$ which attains the maximum value $\max_{x\in H_t}|F(x)|$. 
Denote by $L$ the leaf of the Levi-flat $H_t$ with $x_t\in L$. 
By the maximum modulus principle for $F|_L$ and the density of $L\subset H_t$, it follows that $F|_{H_t}\equiv A$ for some constant $A\in\mathbb{C}$. As $\{x\in V\mid F(x)=A\}$ is an analytic subvariety of $V$ which includes a real three dimensional submanifold $H_t$, we have that $\{x\in V\mid F(x)=A\}=V$. 
\end{proof}

\begin{remark}
There exist nine points configurations $Z^\pm:=\{p_1^\pm, p_2^\pm, \dots, p_9^\pm\}$ such that $N_\pm$ is {\it not} Diophantine, however $C^\pm$ admit holomorphic tubular neighborhoods. 
For example, Ogus constructed such configuration in \cite[4.18]{O}.  
One can carry out just the same construction of a compact complex surface starting from this Ogus's example, whereas Lemma \ref{lem:W*} does not hold in this case. 
\end{remark}

Theorem \ref{thm:constr_of_K3} follows from Proposition \ref{prop:K3}. \qed

\begin{remark}
As it is easily seen from the construction above, we can replace the condition ``$N_{\pm}$ are dual to each other via a biholomorphism $g\colon C^+\to C^-$'' in Theorem \ref{thm:constr_of_K3} with the following looser condition: 
there exists a biholomorphism $N_+^*\cong N_-^*$ which maps the connected component of the boundary of $N_+^*$ corresponding to 
the zero section of $N_+$ 
to the connected component of the boundary of $N_-^*$ corresponding to the boundary of $N_-$, 
where $N_\pm^*$ is the complement of the zero section of $N_\pm$, considered as a domain of the ruled surface obtained by adding the $\infty$-section to $N_\pm$. 
\end{remark}

For the open subset $V\subset X$, we have the following: 

\begin{proposition}\label{prop:finiteness_of_W^*}
Denote by $\widehat{V}_{r^+, r^-}$ the subset $\{x\in N_0\mid 1/r^-<|x|_h<r^+\}$ for positive numbers $r^\pm>1$, where $h$ is a fiber metric on $N_0$ with zero curvature. 
Denote by $i_0$ the isomorphism $\widehat{V}_{R^+, R^-}\to V (\subset X)$ which is obtained naturally by the construction. 
Then it holds that 
{\rm $\sup\{r\geq R^+\mid$ there exists a holomorphic embedding $i_{r, R^-}\colon \widehat{V}_{r, R^-}\to X$ with $i_{r, R^-}|_{\widehat{V}_{R^+, R^-}}=i_0\}<\infty$} and 
{\rm $\sup\{r\geq R^-\mid $ there exists a holomorphic embedding $i_{R^+, r}\colon \widehat{V}_{R^+, r}\to X$ with $i_{R^+, r}|_{\widehat{V}_{R^+, R^-}}=i_0\}<\infty$}. 
\end{proposition}

\begin{proof}
Take $r^\pm\geq R^\pm$ such that there exists a holomorphic embedding $i_{r^+, r^-}\colon\widehat{V}_{r^+, r^-}\to X$ with $i_{r^+, r^-}|_{\widehat{V}_{R^+, R^-}}=i_0$. Then we can calculate to obtain that 
\[
\int_X\sigma\wedge \overline{\sigma} 
\geq \int_{i_{r^+, r^-}(\widehat{V}_{r^+, r^-})}\sigma\wedge \overline{\sigma}
= 4\pi \cdot \left(\int_{C^+} \sqrt{-1}\eta_{C^+}\wedge\overline{\eta_{C^+}}\right)\cdot\log(r^+r^-), 
\]
where $\sigma$ is as in Proposition \ref{prop:K3} and $\eta_{C^+}$ is the holomorphic $1$-form on $C^+$ defined by $\eta_{C^+}|_{U_j^+}=dz_j^+$ on each $U_j^+$. 
Therefore we have an inequality 
\[
\log r^++\log r^-\leq \frac{\int_X\sigma\wedge \overline{\sigma}}{4\pi\int_{C^+} \sqrt{-1}\eta_{C^+}\wedge\overline{\eta_{C^+}}}<+\infty, 
\]
which proves the proposition. 
\end{proof}

By the same argument as in the proof of Proposition \ref{prop:finiteness_of_W^*}, we can show the following statement on tubular neighborhoods of $C^\pm$ in $S^\pm$. 

\begin{proposition}\label{prop:finiteness_of_W}
Denoting by $\widehat{W}_r$ the subset $\{x\in N_0\mid |x|_h<r\}$ and by $i_{R^\pm}^\pm$ the isomorphism $\widehat{W}_{R^\pm}^\pm\to W^\pm\subset S^\pm$ which is obtained naturally by the construction, it holds that 
{\rm $\sup\{r\geq R^\pm\mid $ there exists a holomorphic embedding $i_r^\pm\colon \widehat{W}_r\to S^\pm$ with $i^\pm_r|_{\widehat{W}_{R^\pm}}=i_{R^\pm}^\pm\}<\infty$}, respectively. 
\end{proposition}

Let $\mathcal{N}^+$ be the subsets of the set of all the neighborhoods of $C^+$ in $S^+$ defined as below. 
In this definition of $\mathcal{N}^+$, we omit ``$+$'' and, for example, denote $C^+$ just by $C$. 
A neighborhood $W$ of $C$ in $S$ is an element of $\mathcal{N}^+$ if and only if it is included in a holomorphic tubular neighborhood $W'=\bigcup_j W'_j$ as a relatively compact subset and there exists a positive constant $r$ such that  $W\cap W'_j=\{(z_j', w_j')\mid |w_j'|<r\}$ holds for each $j$, where $(z_j', w_j')$ is the coordinates of $W_j'$ which satisfies the conditions $(i)$, $(ii)$, $(iii)$, $(iv)$, and $(v)$ in the proof of Theorem \ref{thm:constr_of_K3}. 

In the rest of this section, we show the existence of the maximum element of $\mathcal{N}^+$, which will play an important role in \S \ref{section:realizability} for considering the set of all points of the period domain whose corresponding K3 surface can be obtained by our gluing construction. 
First we show the following: 

\begin{lemma}\label{lem:W_max}
The union $W_{\rm max}^+:=\bigcup_{W^+\in\mathcal{N}^+}W^+$ is a holomorphic tubular neighborhood of $C^+$ in $S^+$ with a local coordinates system $\{(W_{{\rm max},j}^+, (z_j^+, w_j^+))\}$ which satisfies the conditions $(i)$, $(ii)$, $(iii)$, $(iv)$, and $(v)$ in the proof of Theorem \ref{thm:constr_of_K3}. 
\end{lemma}

\begin{proof}
Let $W^{(1)}=\bigcup_jW^{(1)}_j$ and $W^{(2)}=\bigcup_jW^{(2)}_j$ be elements of $\mathcal{N}^+$, 
and $(z_j^{(\nu)}, w_j^{(\nu)})$ be the coordinates of $W_j^{(\nu)}$ as above for $\nu=1, 2$. 
Denote by $r_\nu$ the positive constant such that $W_j^{(\nu)}$ is defined as $|w_j^{(\nu)}|<r_\nu$. 
Note that we may assume that the index sets $\{j\}$ coincide to each other and that $W_j^{(1)}\cap C^+=W_j^{(2)}\cap C^+$ holds for each $j$ by shrinking if necessary. 
By the uniqueness assertion in \cite[Theorem 1.1 $(i)$]{K}, one has that there exists a constant $a\in\mathbb{C}^*$ which does not depend on $j$ such that $w_j^{(1)}=a\cdot w_j^{(2)}$ holds (see also the arguments in \cite[p. 588--589]{U}). 
By scaling, we may assume that $a=1$. 

Without loss of generality, we may assume that $r_1\leq r_2$. 
In this case, we have that $W_j^{(1)}\subset W_j^{(2)}$. 
Therefore, one has that $(\mathcal{N}^+, \subset)$ is a linearly ordered set. 
More precisely, from the argument above, one has the following fact: 
For any two elements $W_j^{(1)}$ and $W_j^{(2)}$ with $W_j^{(1)}\subset W_j^{(2)}$, 
there exist positive numbers $r_1$ and $r_2$ with $r_1\leq r_2$ and the isomorphisms 
$i_\nu\colon\{\xi\in N_+\mid |\xi|_h<r_\nu\}\to W^{(\nu)}$ such that 
$i_2|_{\{\xi\in N_+\mid |\xi|_h<r_1\}}=i_1$, where $h$ is a flat metric on $N_+$. 
The lemma is a direct conclusion from this fact. 
\end{proof}

Let $W^+$, $\{(W^+_j, (z_j^+, w_j^+))\}$, and $R^+>0$ be as in the proof of Theorem \ref{thm:constr_of_K3}. 
Denote by $i\colon \{\xi\in N_+\mid |\xi|_h<R^+\}\to W^+$ the corresponding isomorphism. 
Then, by the argument in the proof above, there exists a positive number $R_{\rm max}^+$ larger than $R^+$ and an isomorphism $I\colon\{\xi\in N_+\mid |\xi|_h<R_{\rm max}^+\}\to W^+_{\rm max}$ such that $I_{\{\xi\in N_+\mid |\xi|_h<R^+\}}=i$. 
By Proposition \ref{prop:finiteness_of_W}, we have that $R_{\rm max}^+<\infty$. 

Let $W_{\rm max}^-$ the neighborhood of $C^-$ in $S^-$ defined in the same manner, and $\{(W_{{\rm max},j}^-, (z_j^-, w_j^-))\}$ and $R_{\rm max}^-$ be the corresponding notations. 
In \S \ref{section:realizability}, the numbers 
\[
{\rm vol}_{\eta^\pm}(S^\pm\setminus W_{\rm max}^\pm):=\int_{S^\pm\setminus W^\pm_{\rm max}}\eta^\pm\wedge\overline{\eta^\pm}
\]
play an important role. 
It is natural to ask the following: 
\begin{question}\label{q:vol_0_or_posi}
Are $\Lambda_\pm:={\rm vol}_{\eta^\pm}(S^\pm\setminus W_{\rm max}^\pm)$ equal to $0$? How do these values depend on the nine points configurations? 
\end{question}

Note that $W_{\rm max}^\pm$ (resp. ${\rm vol}_{\eta^\pm}(S^\pm\setminus W_{\rm max}^\pm)$) are naturally determined objects which only depend on $S^\pm$ (resp. $(S^\pm, \eta^\pm)$), whereas the values $R_{\rm max}^\pm$ depend on the scaling of $w_j^\pm$'s. 


\section{Construction of a marking of $X$}\label{section:constr_cycles}

Let $X$ be a K3 surface which is constructed from $(S^\pm, C^\pm)$, where we are using the notation as in the previous section. 
In this section, we construct $22$ cycles $A_{\alpha\beta}, A_{\beta\gamma}, A_{\gamma\alpha}$, $B_\alpha, B_\beta, B_\gamma$, $C_{12}^\pm, C_{23}^\pm, \dots, C_{78}^\pm$ and $C_{678}^\pm$ as in \S \ref{section:introduction} which can be regarded as generators of the second homology group $H_2(X, \mathbb{Z})$, or the K3-lattice $\Pi_{3, 19}$. 
Here, we also observe the value of the integration of $\sigma$ along these $2$-cycles, where $\sigma$ is the holomorphic $2$-form as in Proposition \ref{prop:K3}. 

Note that the construction of the generators of $\Pi_{3, 19}$ we will explain is known at least in the topological level (see \cite[Chapter 3]{GS}). 
Our construction is a slightly modified variant of it so that it is suitable with respect to the complex structure of $X$ and that the calculation of $\sigma$ along $20$ of them is executed concretely. 

As the other cases are done in the same manner, we will only treat the case where nine points are different from each other just for simplicity in this section. 

\subsection{Definition of the cycles $A_\bullet$'s and the integration of $\sigma$ along them}

As $V\subset X$ is biholomorphic to an annulus bundle over the elliptic curve $C^+$, it is homotopic to $S^1\times S^1\times S^1$ (Here we used the topological triviality of $N_\pm$). 
Let $\alpha, \beta$, and $\gamma$ be $C^\omega$ loops of $V$ whose classes define generators of the fundamental group $\pi_1(V, *)$. 
We assume that $\alpha$ and $\beta$ come from $C^+$, which means that $\alpha$ and $\beta$ can be identified with loops $\widehat{\alpha}$ and $\widehat{\beta}$ of $C^+$ respectively via a continuous section $C^+\to V\cong V^+$, and that $\gamma$ is a simple loop settled in a fiber of the $V\cong V^+\to C^+$. 
We may assume that the loop $\widehat{\alpha}$ is the image of the line segment $[0, 1]$ and $\widehat{\beta}$ is the image of line segment $[0, \tau]$ by the universal covering $\mathbb{C}\to \mathbb{C}/\langle 1, \tau\rangle\cong C^+$ for some element $\tau$ of the upper half plane $\mathbb{H}:=\{\tau\in\mathbb{C}\mid {\rm Im}\,\tau>0\}$.  
We define $2$-cycles $A_{\alpha\beta}, A_{\beta\gamma}$, and $A_{\gamma\alpha}$ by 
$A_{\alpha\beta}:=\alpha\times \beta$, 
$A_{\beta\gamma}:=\beta\times \gamma$, and 
$A_{\gamma\alpha}:=\gamma\times \alpha$. 
As these are (concretely defined) topological tori included in $V$ and we have the concrete description of $\sigma|_V$ ($\sigma|_{V_j}=(dz_j^+\wedge dw_j^+)/w_j^+$), one can carry out the integration of $\sigma$ along these three $2$-cycles in a concrete manner. 
By the computation, one has that 
\[
\frac{1}{2\pi\sqrt{-1}}\int_{A_{\alpha\beta}}\sigma=a_\beta-\tau\cdot a_\alpha, 
\]
\[
\frac{1}{2\pi\sqrt{-1}}\int_{A_{\beta\gamma}}\sigma=\tau 
\]
and 
\[
\frac{1}{2\pi\sqrt{-1}}\int_{A_{\gamma\alpha}}\sigma=1, 
\]
where $a_\alpha$ and $a_\beta$ are the real numbers such that the monodromies of the $U(1)$-flat line bundle $N_+$ along the loops $\widehat{\alpha}$ and $\widehat{\beta}$ are $\exp(2\pi\sqrt{-1}\cdot a_\alpha)$ and $\exp(2\pi\sqrt{-1}\cdot a_\beta)$, respectively. 

\subsection{Definition of the cycles $C_\bullet^\pm$'s and the integration of $\sigma$ along them}\label{section:def_C_bullet}

Denote by $e_\nu^\pm$ the exceptional curve in $S^\pm$ which is the preimage of $p_\nu^\pm\in Z^\pm$ for $\nu=1, 2, \dots, 9$, 
by $h^\pm$ the preimage of a line in $\mathbb{P}^2$ by the blow-up $\pi^\pm\colon S^\pm\to \mathbb{P}^2$, 
and by $q_\nu^\pm$ the point of $C^\pm$ which corresponds to $p_\nu^\pm$ via 
$\pi^\pm|_{C^\pm}\colon C^\pm\cong C_0^\pm$ for $\nu=0, 1, \dots, 9$, where $p_0^\pm$ is an inflection point of the cubic $C_0^\pm$. 

First we give the definition of $C_{12}^+$. 
In this paragraph, we omit ``$+$'' to denote, for example, $C^+$ simply by $C$, since all the objects are observed in $S^+$. 
Fix a line segment $\Gamma_{12}$ in $C$ which connects $q_1$ and $q_2$. 
Fix also a positive number $\varepsilon$ less than $R$. 
Denote by $\Delta^{(\varepsilon)}_\nu$ the subset $e_\nu\cap \Phi^{-1}([0, \varepsilon))$ of $e_\nu$. 
Let $\widehat{\Gamma}_{12}$ be a line segment of $C$ which is a slight extension of $\Gamma_{12}$ such that the inverse image $\widehat{T}_{12}^{(\varepsilon)}$ of $\widehat{\Gamma}_{12}$ by the natural projection $\Phi^{-1}(\varepsilon)\to C$ satisfies the condition that 
both $\widehat{T}_{12}^{(\varepsilon)}\cap e_1$ and $\widehat{T}_{12}^{(\varepsilon)}\cap e_2$ are homeomorphic to $S^1$. 
Denote by $T_{12}^{(\varepsilon)}$ the connected component of 
$\widehat{T}_{12}^{(\varepsilon)}\setminus(e_1\cup e_2)$ whose boundary is the union of  $\widehat{T}_{12}^{(\varepsilon)}\cap e_1$ and $\widehat{T}_{12}^{(\varepsilon)}\cap e_2$. 
Define a subset $C_{12}^{(\varepsilon)}$ of $S$ by 
\[
C_{12}^{(\varepsilon)}:=(e_1\setminus \Delta^{(\varepsilon)}_1)\cup T_{12}^{(\varepsilon)}\cup(e_2\setminus \Delta^{(\varepsilon)}_2), 
\]
which is homeomorphic to $S^2$ (see Figure \ref{figure:C12}).

\begin{center}
\begin{figure}[h]
\begin{tikzpicture}[x=0.75cm,y=0.75cm]
\draw(0,0) circle (3.5 and 2.5);
\draw (-3.4,1) node [left, above] {$S$};
\draw [dashed] (1.3,2.2) to (1.3,-2.2);
\draw [<->, bend right,distance=0.7cm] (1.3,-2.6) to node [fill=white,inner sep=0.2pt,circle] {$W$} (3.5,-1.5);
\draw (1.5,-0.8) to (1.5,0.8);
\draw (2.5,-1.5) to (2.5,1.5);
\draw (1.5,0.8) to [out=70,in=190]  (2.5,1.5);
\draw[dashed] (1.5,0.8) to [out=20,in=240]  (2.5,1.5);
\draw (1.5,-0.8) to [out=350,in=130] (2.5,-1.5);
\draw[dashed] (2.5,-1.5) to [out=160,in=300]  (1.5,-0.8);
\draw (2.5,1.5) .. controls (2.4,1.6) and (-7.0,1.0)  .. (1.5,0.8); 
\draw (2.5,-1.5) .. controls (2.4,-1.6) and (-7.0,-1.0)  .. (1.5,-0.8); 
\draw (-1.2,0.5) node [right] {$e_{1}$};
\draw (-1.2,-0.5) node [right] {$e_{2}$};
\fill (2,1.15) circle (1.5pt);
\draw (2.4,1.15) node [right] {$q_{1}$};
\fill (2,-1.15) circle (1.5pt);
\draw (2.4,-1.05) node [right] {$q_{2}$};
\draw (2,1.15) to (2,-1.15);
\draw (1.5,0.3) to [out=20,in=160]  (2.5,0.3);
\draw[dashed] (1.5,0.3) to [out=-20,in=-160]  (2.5,0.3);
\draw[->] (4.0,0.1) node [right] {$\Gamma_{12} \subset C$} to [out=180,in=10]  (2.1,0);
\draw[->] (4.0,0.8) node [right] {$T_{12}^{(\varepsilon)}$} to [out=180,in=30]  (2.6,0.3);
\draw (0,1.3) node [above] {$C_{12}^{(\varepsilon)}$};
\end{tikzpicture}
\caption{The cycle $C_{12}^{(\varepsilon)}$}
\label{figure:C12}
\end{figure}
\end{center}

By defining the orientation suitably, one can regard $C_{12}^{(\varepsilon)}$ as a $2$-cycle of $S$ which is homologous to $e_1-e_2$. 

By taking $\varepsilon$ so that $1/R^-<\varepsilon$, we have that $C_{12}^{(\varepsilon)}\subset M^+$. 
Therefore $C_{12}^{(\varepsilon)}$ can also be regarded as a $2$-cycle of $X$, which is the definition of $C_{12}^+$. 
Note that this definition (as a $2$-cycle) does not depend on the choice of $\varepsilon\in (1/R^-, R^+)$. 

Next, we calculate the integration of $\sigma$ along $C_{12}^+$. 

\begin{proposition}
\[
\frac{1}{2\pi\sqrt{-1}}\int_{C_{12}^+}\sigma=\int_{\Gamma_{12}^+}dz^+, 
\]
where $dz^+$ is the global holomorphic $1$-form on $C^+$ such that $dz^+|_{U_j^+}=dz_j^+$ holds on each $U_j^+$. 
\end{proposition}

\begin{proof}
Let $\eta^\pm$ be a meromorphic $2$-form of $S^\pm$ as in the proof of Proposition \ref{prop:K3}. 
As is mentioned just after the proof of Proposition \ref{prop:K3}, we may assume that $\sigma$ is obtained by gluing $\eta^+|_{M^+}$ and $\eta^-|_{M^-}$. 
Especially it holds that $\sigma|_{M^+}=\eta^+|_{M^+}$. 
Thus, instead of calculating $\int_{C_{12}}\sigma$ on $X$, 
here we calculate $\int_{C_{12}^{(\varepsilon)}}\eta^+$ on $S^+$. 

Again we omit the notation $+$. 
As $\eta|_{S\setminus C}$ is closed and $C_{12}^{(\varepsilon)}$ is homologous to $C_{12}^{(\varepsilon')}$ as $2$-cycle for any $\varepsilon'\in (0, R)$, 
we have that
\[
\int_{C_{12}}\sigma=\lim_{\varepsilon\searrow 0}\int_{C_{12}^{(\varepsilon)}}\eta
=\lim_{\varepsilon\searrow 0}\left(\int_{e_1\setminus \Delta^{(\varepsilon)}_1}\eta+\int_{T_{12}^{(\varepsilon)}}\eta+\int_{e_2\setminus \Delta^{(\varepsilon)}_2}\eta\right)
=\lim_{\varepsilon\searrow 0}\int_{T_{12}^{(\varepsilon)}}\eta 
\]
(Note that $\int_{e_\nu}\eta=0$, since $e_\nu$ is analytic and $\eta$ is a $(2, 0)$-form). 
For calculating the right hand side, first we calculate the integral $\int_{\widehat{T}_{12}^{(\varepsilon)}}\eta$ on $W_j$. 
As we may assume that $W_j$ coincides with the preimage of $U_j=W_j\cap C$ by the projection $W\to C$, we have that 
\[
\int_{\widehat{T}_{12}^{(\varepsilon)}\cap W_j}\eta
=\int_{\{(z_j, w_j)\mid z_j\in \widehat{\Gamma}_{1, 2}\cap U_j, |w_j|=\varepsilon\}}\frac{dw_j\wedge dz_j}{w_j}
=2\pi\sqrt{-1} \int_{\widehat{\Gamma}_{12}\cap U_j}dz_j. 
\]
When $\varepsilon$ is sufficiently small, it follows from the construction of $\widehat{\Gamma}_{12}$ that 
\[
\int_{\Gamma_{12}\cap U_j}dz_j=\int_{\widehat{\Gamma}_{12}\cap U_j}dz_j+O(\varepsilon) 
\]
as $\varepsilon \searrow 0$. 
Thus we have the proposition by considering the limit as $\varepsilon \searrow 0$. 
\end{proof}

Define $2$-cycles $C_{\nu, \nu+1}^+$ for $\nu=2, 3, \dots, 7$ 
and $C_{\nu, \nu+1}^-$ for $\nu=1, 2, \dots, 7$ by the same manner 
so that $C_{\nu, \nu+1}^\pm$ is included in $M^\pm$ as a set and is homologous to $e_\nu-e_{\nu+1}$ as a $2$-cycle of $S^\pm$. 
By the same argument as in the proof above, we have that 
\[
\frac{1}{2\pi\sqrt{-1}}\int_{C_{\nu, \nu+1}^\pm}\sigma=\int_{\Gamma_{\nu, \nu+1}^\pm}dz^\pm,\ 
\]
where $\Gamma_{\nu, \nu+1}^\pm$ is a line segment in $C^\pm$ which  connects  $q_\nu^\pm$ and $q_{\nu+1}^\pm$ and $dz^-$ is the global holomorphic $1$-form on $C^-$ defined just in the same manner as the definition of $dz^+$. 

The cycle $C_{678}^\pm$ is also defined as a cycle supported in $M^\pm$ which is homeomorphic to $S^2$ and homologous to  $ -h^\pm+e_6^\pm+e_7^\pm+e_8^\pm$ as a $2$-cycle in $S^\pm$. 
The integration of $\sigma$ is computed as 
\[
\frac{1}{2\pi\sqrt{-1}}\int_{C_{678}^\pm}\sigma=\int_{\Gamma_{06}^\pm+\Gamma_{07}^\pm+\Gamma_{08}^\pm}dz^\pm, 
\]
where $\Gamma_{0\nu}^\pm$ is a line segment in $C^\pm$ connecting $q_{0}^\pm$ and $q_{\nu}^\pm$. 

\subsection{Definition of the cycles $B_\bullet$'s and the integration of $\sigma$ along them}

First we give the definition of $B_\alpha$. 
Regard $\alpha\subset V$ as a $1$-cycle of $M^\pm$. 
Then, as $M^\pm$ is simply connected (\cite[Chapter 3]{GS}),  there exists a topological disc $D_\alpha^\pm\subset M^\pm$ such that $\partial D_\alpha^\pm=\pm\alpha$ holds. 
$B_\alpha$ is defined by patching $D_\alpha^+$ and $D_\alpha^-$. 
The $2$-cycle $B_\beta$ is defined in the same manner. 
At this moment, we have not succeeded in giving the concrete values of the integrals of $\sigma$ along these two cycles. 
These values are investigated in \S \ref{section:realizability}. 

The $2$-cycle $B_\gamma$ is also defined in the same manner: $B_\gamma=D_\gamma^+\cup_\gamma D_\gamma^-$, where $D_\gamma^\pm$ is a topological disc in $M^\pm$ such that $\partial D_\gamma^\pm=\pm\gamma$ holds. 
In this case, $D_\gamma^\pm$ can be constructed more concretely so that the integral $\int_{B_\gamma}\sigma$ is computable. 
Take a real number $r\in (1/R^-, R^+)$. 
We may assume that $\gamma$ is the preimage of a point $z_0\in C^+$ by the projection $\Phi_+^{-1}(r)\to C^+$. 
Take a line segment $\Gamma_9$ in $C^+$ connecting $p_9^+$ and $g^{-1}(p_9^-)$. 
Then, by taking a suitable topological tube $T_9^{(r)}$ included in the preimage of the line segment which is a slight extension of $\Gamma_9$ by the natural map $\Phi_+^{-1}(r)\to C^+$, it follows from the same argument as in the previous subsection that one may assume 
\[
B_\gamma:=(e_9^+\setminus \Phi_+^{-1}([0, r)))\cup T_9^{(r)}\cup  (e_9^-\setminus \Phi_-^{-1}([0, 1/r))). 
\]
By the same computation as before, we have that 
\[
\frac{1}{2\pi\sqrt{-1}}\int_{B_\gamma}\sigma=\int_{\Gamma_9}dz^+. 
\]

\subsection{Summary}

By the previous subsection, we defined $22$ $2$-cycles and computed the concrete value of the integral of $\sigma$ along $20$ of them other than $B_\alpha$ and $B_\beta$. 
Note that, as is explained before the construction of them, they actually forms the generators of $H_2(X, \mathbb{Z})$, since the construction itself coincides with the classically known construction of the generators in the topological level (See \cite[Chapter 3]{GS}). 

Here we summarize the conclusion and investigate the relation between these values and the parameters that appear in the construction of $X$. 

First, the value $\frac{1}{2\pi\sqrt{-1}}\int_{A_{\beta\gamma}}\sigma$ is equal to $\tau$, which is the modulus of the elliptic curve $C_0^\pm$. 
Next, it follows from the computation in \S \ref{section:def_C_bullet} that the data
\[
\left(\int_{C_{12}^\pm}\sigma, \int_{C_{23}^\pm}\sigma, \dots, \int_{C_{78}^\pm}\sigma, \int_{C_{678}^\pm}\sigma
\right)\in\mathbb{C}^8
\]
is completely determined only by the choice of $p_1^\pm, p_2^\pm, \dots, p_8^\pm\in C_0^\pm$ after fixing an inflection point $p_0^\pm$. 
After fixing $C_0^\pm$ and the eight points configurations $p_1^\pm, p_2^\pm, \dots, p_8^\pm\in C_0^\pm$, to determine the complex structure (or equivalently, the flat structure) of  $N_+\cong\mathcal{O}_{\mathbb{P}^2}(3)|_{C_0^+}\otimes \mathcal{O}_{C_0^+}(-p_1^+-p_2^+-\cdots-p_9^+)$ is equivalent to determine the choice of the ninth point $p_9^+$, whose information is reflected by the integral of $\sigma$ along $A_{\alpha\beta}$. 
Next, by determining the isomorphism $g\colon C^+\to C^-$, the relative position of $g^{-1}(p_9^-)$ from $p_9^+$ is determined (Note that here we used the condition $N_+=g^*N_-$). 
This information is reflected by the integral of $\sigma$ along $B_\gamma$. 

Finally, after fixing the data $C_0^\pm$, $(p_1^\pm, p_2^\pm, \dots, p_9^\pm)$, and $g$, 
only the remaining parameter on the construction of $X$ is the parameter on the scaling of $w_j^\pm$'s and the ``size'' $R^\pm$ of $V$. 
For the relation between this degree of freedom and the value $\int_{B_\alpha}\sigma$ and $\int_{B_\beta}\sigma$, see \S \ref{section:realizability}. 

\begin{table}
\scalebox{1}{
  \begin{tabular}{|c|c|c|c|}\hline
    & $2$-cycle & $\frac{1}{2\pi\sqrt{-1}}\int\sigma$ & The corresponding parameter \\\hline\hline
    U& $A_{\beta\gamma}$ & $\tau$ & The choice of $C_0^\pm$ \\ \cline{2-4}
    & $B_{\alpha}$ & unknown & The scaling of $w_j$'s and the choice of $R^\pm$  \\ \hline
    U& $A_{\gamma\alpha}$ & $1$ & - (``the normalization of $\sigma$'') \\ \cline{2-4}
    & $B_{\beta}$ & unknown & The scaling of $w_j$'s and the choice of $R^\pm$ \\ \hline
    & $C_{12}^+$ & $\int_{\Gamma_{12}^+}dz^+$ & The relative position of $p_1^+$ viewed from $p_2^+$ \\ \cline{2-4}
    & $C_{23}^+$ & $\int_{\Gamma_{23}^+}dz^+$ & The relative position of $p_2^+$ viewed from $p_3^+$ \\ \cline{2-4}
    $E_8(-1)$& $\vdots$ & $\vdots$ &  $\vdots$ \\ \cline{2-4}
    & $C_{78}^+$ & $\int_{\Gamma_{78}^+}dz^+$ & The relative position of $p_7^+$ viewed from $p_8^+$ \\ \cline{2-4}
    & $C_{678}^+$ &$\int_{\Gamma_{06}^++\Gamma_{07}^++\Gamma_{08}^+}dz^+$ & The position of ``$p_6^++p_7^++p_8^+$'' \\ \hline
    & $C_{12}^-$ & $\int_{\Gamma_{12}^-}dz^-$ & The relative position of $p_1^-$ viewed from $p_2^-$ \\ \cline{2-4}
    & $C_{23}^-$ & $\int_{\Gamma_{23}^-}dz^-$ & The relative position of $p_2^-$ viewed from $p_3^-$  \\ \cline{2-4}
    $E_8(-1)$& $\vdots$ & $\vdots$ &  $\vdots$ \\ \cline{2-4}
    & $C_{78}^-$ & $\int_{\Gamma_{78}^-}dz^-$ & The relative position of $p_7^-$ viewed from $p_8^-$  \\ \cline{2-4}
    & $C_{678}^-$ &$\int_{\Gamma_{06}^-+\Gamma_{07}^-+\Gamma_{08}^-}dz^-$ & The position of ``$p_6^-+p_7^-+p_8^-$'' \\ \hline
    U& $A_{\alpha\beta}$ & $a_\beta-\tau\cdot a_\alpha$ & The choice of $p_9^+$ (or $N_+$) \\ \cline{2-4}
    & $B_{\gamma}$ & $\int_{\Gamma_9}dz^+$ & The choice of $g\colon C^+\to C^-$ \\ \hline
  \end{tabular}
}
\caption{The integration of $\sigma$ along $2$-cycles and the corresponding parameters.}
\end{table}

\section{Deformation of $X$}\label{section:deformation_of_X}

Let $C_0^{\pm}, Z^{\pm}, N_0$, and $X$ be those in \S \ref{section:construction}. 
As we observed in the previous section, our construction of $X$ has some degrees of freedom: 
on the complex structure of $C_0^{\pm}$, on the points configurations $Z^{\pm}$, and on the patching functions, 
even after one fixes $N_0$. 
In this section, we investigate some of the deformation families constructed by considering such degrees of freedom, 
by using notations mentioned in \S \ref{section:construction}. 

\subsection{Deformation families constructed by changing the patching functions}\label{section:deform_patching}

\subsubsection{A deformation family corresponding to the change of the patching function $w^-(z^+, w^+)$}\label{section:deform_patching_w}

Here we fix $C_0^{\pm}$, $Z^{\pm}$, and the isomorphism $g$. 
Put $\Delta=\Delta_{w}:=\{ t \in \mathbb{C} \mid 0 <|t|< R_+R_-\}$, 
$\mathcal{V}^{\pm} :=\{ (x,t) \in W^{\pm} \times \Delta \mid |t|/R^{\mp} < \Phi_{\pm}(x) < R^{\pm} \}$ 
and $\mathcal{M}^{\pm}:=\{ (x,t) \in S^{\pm} \times \Delta \mid x \notin \Phi_{\pm}^{-1}([0,|t|/R^{\mp}]) \}$. 
In what follows, we sometimes omit the index $j$ to, for example, denote $(z_j^\pm, w_j^\pm)$ just by $(z^\pm, w^\pm)$. 
Define a holomorphic function $F \colon \mathcal{V}^+ \to \mathcal{V}^-$ by 
\[
F \colon (z^+, w^+, t) \mapsto (z^-(z^+, w^+, t), w^-(z^+, w^+, t), t):=\left(g(z^+), \frac{t}{w^+}, t\right)
\]
Then we can patch $\mathcal{M}^+$ and $\mathcal{M}^-$ via the map $F$ to define a complex manifold $\mathcal{X}$. 
In what follows, we may regard $\mathcal{M}^{\pm}$ and $\mathcal{V}:=\mathcal{V}^{\pm}$ as open subsets of $\mathcal{X}$. 
The second projections $\mathcal{M}^{\pm} \to \Delta$ glue up to define a proper holomorphic submersion $\pi \colon \mathcal{X} \to \Delta$. 
It follows from Proposition \ref{prop:K3} that each fiber $X_t:=\pi^{-1}(t)$ is a K3 surface. 
In what follows, we fix a base point $t_0 \in \Delta$. 

In this subsection, we assume that $Z^\bullet$ is sufficiently general so that it includes four points in which no three points are collinear for each $\bullet\in \{+, -\}$. 
This assumption is just for simplicity (see Remark \ref{rmk:colinear_iranai}).

\begin{proposition}\label{prop:deform_R}
The Kodaira--Spencer map $\rho_{{\rm KS}, \pi}\colon T_{\Delta, t_0}\to H^1(X_{t_0}, T_{X_{t_0}})$ of the deformation family 
$\pi \colon \mathcal{X}\to \Delta$ is injective. 
\end{proposition}

We define holomorphic vector fields $\theta_1, \theta_2$ on $V_{t_0}:=X_{t_0} \cap \mathcal{V}$ by 
\[
\theta_1:= w^+ \frac{\partial}{\partial w^+}, \quad \theta_2:=\frac{\partial}{\partial z^+}, 
\]
and put $M_{t_0}^{\pm}:=X_{t_0}\cap \mathcal{M}^{\pm}$. 
Proposition \ref{prop:deform_R} follows from the following two lemmata. 

\begin{lemma}\label{lem:KS_map_for_w_deform}
As a \v{C}ech cohomology class, we have 
\[
\rho_{{\rm KS}, \pi} \left( \left.\frac{\partial}{\partial t}\right|_{t_0} \right)=[\{(M_{t_0}^+ \cap M_{t_0}^-, - t_0^{-1} \cdot \theta_1)\}] \in \check{H}^1(\{M_{t_0}^+, M_{t_0}^-\}, T_{X_{t_0}}). 
\]
\end{lemma}

\begin{proof}
The lemma directly follows from the computation 
\[
\frac{\partial w^-(z^+, w^+, t)}{\partial t}\frac{\partial}{\partial w^-}
=\frac{\partial}{\partial t} \left(\frac{t}{w^+} \right) \cdot \frac{\partial}{\partial w^-}
=t^{-1} w^-\frac{\partial}{\partial w^-}=-t^{-1}\cdot w^+ \frac{\partial}{\partial w^+}=-t^{-1} \cdot \theta_1. 
\]
\end{proof}

\begin{lemma}\label{lem:inj_MV_seq}
The coboundary map $H^0(V_{t_0}, T_{V_{t_0}})\to H^1(X_{t_0}, T_{X_{t_0}})$ which appears in the Mayer--Vietoris sequence corresponding to the covering $\{M_{t_0}^+, M_{t_0}^-\}$ of $X_{t_0}$ is injective. 
\end{lemma}

See \cite[II 5.10]{I} for example of the Mayer--Vietoris sequence for an open covering. 

\begin{proof}
From the Mayer--Vietoris sequence, it is sufficient to show that 
$H^0(M_{t_0}^{\pm}, T_{M_{t_0}^{\pm}})=0$. 
Take an element $\xi \in H^0(M_{t_0}^+, T_{M_{t_0}^+})$. 
As $\xi|_{V_{t_0}}\in H^0(V_{t_0}, T_{V_{t_0}})$ and 
$T_{V_{t_0}}=\theta_1 \cdot \mathbb{I}_{V_{t_0}}\oplus \theta_2\cdot\mathbb{I}_{V_{t_0}}$, 
it follows from Lemma \ref{lem:W*} that there exists an element $(a_1, a_2)\in\mathbb{C}^2$ such that $\xi|_{V_{t_0}}=a_1\cdot \theta_1+a_2\cdot \theta_2$. 
As it is clear that both $\theta_1$ and $\theta_2$ can be extended to holomorphic vector fields on $W_{t_0}^+:=W^+ \times \{t_0\}$, 
it follows that there exists $\zeta\in H^0(W_{t_0}^+, T_{W_{t_0}^+})$ such that $\zeta|_{V_{t_0}}=\xi|_{V_{t_0}}$. 
By gluing $\zeta$ and $\xi$, we obtain a global vector field $\widetilde{\xi} \in H^0(S^+, T_{S^+})$. 
It follows from Lemma \ref{lem:P^2_bup_cohomoogy} that $\widetilde{\xi}=0$, and thus $\xi=0$, 
which shows that $H^0(M_{t_0}^+, T_{M_{t_0}^+})=0$. 
Similarly we have $H^0(M_{t_0}^-, T_{M_{t_0}^-})=0$. 
\end{proof}

\begin{proof}[Proof of Proposition \ref{prop:deform_R}]
As $\check{H}^1(\{M_{t_0}^+, M_{t_0}^-\}, T_{X_{t_0}})\to
\mathop{\varinjlim}\limits_{\mathcal{U}}\check{H}^1(\mathcal{U}, T_{X_{t_0}})=H^1(X_{t_0}, T_{X_{t_0}})$ is injective, 
Proposition \ref{prop:deform_R} follows from Lemma \ref{lem:KS_map_for_w_deform} and Lemma \ref{lem:inj_MV_seq}. 
\end{proof}

\begin{remark}\label{rmk:colinear_iranai}
In the proof of Proposition \ref{prop:deform_R} (especially in the proof of Lemma \ref{lem:inj_MV_seq}), 
we used the assumption that $Z^\pm$ is sufficiently general so that it includes four points in which no three points are collinear. 
However one can remove this assumption by the following alternative proof of Proposition \ref{prop:deform_R}. 
For each $t\in \Delta_w$, denote by $I_t\colon \mathbb{R}/\mathbb{Z}\times [|t|/R_-, R_+]\to V_t^+$ the map defined by 
\[
I_t([r], \ell) := [(r, \ell\cdot \exp({2\pi\sqrt{-1}a_\alpha}r))], 
\]
where we are identifying $V_t^+$ with the quotient $\mathbb{C}\times \Delta/\sim$ by the relation $\sim$ generated by $(z, w)\sim (z+1, \exp({2\pi\sqrt{-1}a_\alpha})\cdot w)\sim (z+\tau, \exp({2\pi\sqrt{-1}a_\beta})\cdot w)$. 
As both the components of $X_t\setminus V_t$ are simply connected, there exist continuously embedded discs $D_\alpha^+$ and $D_\alpha^-$ in $X_t\setminus V_t$ such that $\partial D_\alpha^+ = I_t(\mathbb{R}/\mathbb{Z}\times \{R_+\})$ and $\partial D_\alpha^- = I_t(\mathbb{R}/\mathbb{Z}\times \{|t|/R_-\})$. 
Note that one may regard each $D_\alpha^\pm$ as a subset of $S^\pm$ which does not depend on the parameter $t$. 
Then one has that 
\begin{eqnarray}
\frac{\partial}{\partial t}\int_{B_\alpha}\sigma &=& \frac{\partial}{\partial t}\left(\int_{D_\alpha^+}\sigma +\int_{I_t(\mathbb{R}/\mathbb{Z}\times [|t|/R_-, R_+])}\sigma+\int_{D_\alpha^-}\sigma \right)\nonumber \\
 &=& \frac{\partial}{\partial t} \int_{I_t(\mathbb{R}/\mathbb{Z}\times [|t|/R_-, R_+])}\sigma =\frac{\partial}{\partial t} \int_{\mathbb{R}/\mathbb{Z}\times [|t|/R_-, R_+]}I_t^*\sigma\nonumber \\
 &=& \frac{\partial}{\partial t} \left(\int_0^1dr\cdot \int_{|t|/R_-}^{R_+}\frac{d\ell}{\ell} \right) 
=\frac{\partial}{\partial t} (-\log |t|)
= \frac{-1}{2t}.  \nonumber
\end{eqnarray}
Therefore, by considering the Griffiths transversality on the relation between the Kodaira-Spencer map and the derivative of the period map, one can prove Proposition \ref{prop:deform_w_g_points}. 
Here we remark that the value $\frac{\partial}{\partial t}\int_{B_\beta}\sigma$ can also be concretely calculated in the same manner. 
\end{remark}

\subsubsection{A deformation family corresponding to the change of the patching function $z^-(z^+, w^+)$}

Let $\Delta=\Delta_z:=\{t\in\mathbb{C}\mid |t|<1\}$, and
$\widetilde{g}\colon \mathbb{C}\to \mathbb{C}$ be the isomorphism such that 
$p_{C^-} \circ \widetilde{g}=g \circ p_{C^+}$ holds, where $p_{C^{\pm}} \colon \mathbb{C} \to C^{\pm}$ are the universal covers. 
Without loss of generality, we may assume that $\frac{\partial}{\partial z}\widetilde{g}(z) \equiv 1$. 
Denote by $\widetilde{g}_t$ the automorphism of $\mathbb{C}$ defined by $\widetilde{g}_t(z):=\widetilde{g}(z)+t$ and by $g_t\colon C^+ \to C^-$ the isomorphism induced by $\widetilde{g}_t$ for each $t\in \Delta$. 
Let $\mathcal{M}^{\pm}:=M^{\pm} \times \Delta$, $\mathcal{W}^{\pm}:=W^{\pm} \times \Delta$ and $\mathcal{V}^{\pm}:=V^{\pm} \times \Delta$. 
Let $\mathcal{V}^{\pm}_j$ be the subsets of $\mathcal{V}^{\pm}$ defined by 
$\mathcal{V}^+_j:=U_j^+ \times \Phi_{+}^{-1}((1/R^-,R^+)) \times \Delta$ and 
$\mathcal{V}^-_j:=\{(z_j^-, w_j^-, t)\mid z_j^- \in g_t(U_j^+), w_j^- \in \Phi_{-}^{-1}((1/R^+,R^-)), t\in\Delta\}$. 
Define a holomorphic function $F\colon \mathcal{V}^+ \to \mathcal{V}^-$ by 
\[
F|_{\mathcal{V}^+_j}\colon(z_j^+, w_j^+, t)\mapsto (z_j^-(z_j^+, w_j^+, t), w_j^-(z_j^+, w_j^+, t), t):=\left(g_t(z_j^+), (w_j^+)^{-1}, t\right)\in \mathcal{V}^-_j
\]
on each $\mathcal{V}^+_j$. 
Then we can patch $\mathcal{M}^+$ and $\mathcal{M}^-$ by identifying $\mathcal{V}^{\pm}$, denoted by $\mathcal{V}$, 
via the map $F$ to define a complex manifold $\mathcal{X}$. 
We regard $\mathcal{M}^{\pm}$ and $\mathcal{V}$ as open subsets of $\mathcal{X}$. 
The second projections $\mathcal{M}^{\pm} \to \Delta$ glue up to define a proper holomorphic submersion $\pi\colon \mathcal{X}\to \Delta$. 
By Proposition \ref{prop:K3}, we have that each fiber $X_t:=\pi^{-1}(t)$ is a K3 surface. 

\begin{proposition}\label{prop:deform_g}
The Kodaira--Spencer map $\rho_{{\rm KS}, \pi}\colon T_{\Delta, 0}\to H^1(X_{0}, T_{X_{0}})$ of the deformation family $\pi\colon \mathcal{X}\to \Delta$ is injective. 
\end{proposition}

As in the previous subsection, we use notations in the proof of Proposition \ref{prop:deform_R}. 

\begin{lemma}\label{lem:KS_map_for_g_deform}
As a \v{C}ech cohomology class, we have 
\[
\rho_{{\rm KS}, \pi} \left( \left.\frac{\partial}{\partial t}\right|_{t=0} \right)=[\{(M_{0}^+ \cap M_{0}^-, \theta_2)\}] \in \check{H}^1(\{M_{0}^+, M_{0}^-\}, T_{X_{0}}). 
\]
\end{lemma}

\begin{proof}
The lemma directly follows from the computation 
\[
\frac{\partial z^-(z^+, w^+, t)}{\partial t}\frac{\partial}{\partial z^-}
=\left(\frac{\partial}{\partial t}(g(z^+)+t)\right)\cdot \frac{\partial}{\partial z^-}
=\frac{\partial}{\partial z^-}=\frac{\partial}{\partial z^-}g_t^{-1}(z^-)\cdot \frac{\partial}{\partial z^+}=\frac{\partial}{\partial z^+}=\theta_2. 
\]
\end{proof}

\begin{proof}[Proof of Proposition \ref{prop:deform_g}]
As in the proof of Proposition \ref{prop:deform_R}, Proposition \ref{prop:deform_g} follows from Lemma \ref{lem:inj_MV_seq} and Lemma \ref{lem:KS_map_for_g_deform}. 
\end{proof}

\subsection{A deformation family constructed by changing the nine points configurations}\label{section:deform_points}

Let $C_0^{\pm}, N_0, Z^{\pm}=\{p_1^{\pm}, \dots, p_9^{\pm} \}, S^{\pm}, C^{\pm}, M^{\pm}, W^{\pm}, R^{\pm}, g$, and $V$ be those in \S \ref{section:construction}. 
In this subsection, we construct a deformation family corresponding to the change of the nine points configurations $Z^{\pm}$, leaving $C_0^{\pm}, N_0$ and the patching functions fixed. 
For simplicity, we also fix $Z^-$ and consider only the change of $Z^+$ here, and put $C_0:=C_0^{\pm}$ and $Z=\{p_1,\dots,p_9 \}:=Z^+$. 

Fix a sufficiently small open neighborhood $U_\nu$ of $p_\nu$ in $C_0$ for each $\nu=1, 2, \dots, 8$ and 
denote by $T$ the product $U_1\times U_2\times\cdots \times U_8$. 
In what follows, we regard $t_0:=(p_1, p_2, \dots, p_8)$ as a base point of $T$. 
For each $t=(q_1, q_2, \dots, q_8)\in T$, we define $q(t)\in C_0$ by $\mathcal{O}_{\mathbb{P}^2}(3)|_{C_0}\otimes\mathcal{O}_{C_0}(-q_1-q_2-\cdots-q_8-q(t))\cong N_0$. 
Let $\pi\colon \mathcal{S}\to T$ be a proper holomorphic submersion from a $10$-dimensional complex manifold $\mathcal{S}$ to $T$ such that each fiber $S_t:=\pi^{-1}(t)$ is isomorphic to the blow-up of $\mathbb{P}^2$ at nine points $q_1, q_2, \dots, q_8$, and $q(t)$ for each $t=(q_1, q_2, \dots, q_8)\in T$. 
Such $\mathcal{S}$ can be constructed as the blow-up of $\mathbb{P}^2\times T$ along some submanifolds. 
Let $\mathcal{C}\subset \mathcal{S}$ be the strict transform of $C_0\times T$. 
Note that $\pi|_{\mathcal{C}}={\rm Pr}_2\colon C_0\times T\to T$ holds via the natural isomorphism between $\mathcal{C}$ and $C_0\times T$. 
Denote by $C_t$ the intersection $S_t\cap \mathcal{C}$ for each $t\in T$. 
Then it follows from the construction that $N_{\mathcal{C}/\mathcal{S}}|_{C_t}=N_{C_t/S_t}\cong N_0$ for each $t$. 
Therefore, by regarding  $N_{\mathcal{C}/\mathcal{S}}$ as a holomorphic line bundle on $C_0\times T$, it follows from Proprosition \ref{prop:coarse_moduli} that there exists a holomorphic map $i\colon T\to {\rm Pic}^0(C_0)$ such that $({\rm id}_{C_0}\times i)^*\mathcal{L}\cong N_{\mathcal{C}/\mathcal{S}}$ holds, where $\mathcal{L}$ is the universal line bundle on $C_0\times {\rm Pic}^0(C_0)$. 
As it clearly holds that $i(t)\equiv N_0\in {\rm Pic}^0(C_0)$, we obtain that $N_{\mathcal{C}/\mathcal{S}}\cong {\rm Pr}_1^*N_0$. 
Then, we can apply the following relative variant of Arnol'd's theorem to our $(\mathcal{C}, \mathcal{S})$. 

\begin{theorem}\label{thm:rel_Arnol'd}
Let $\pi\colon \mathcal{S}\to T$ be a deformation family of complex surfaces over a ball in $\mathbb{C}^n$, 
and $\mathcal{C}\subset \mathcal{S}$ be a submanifold which is biholomorphic to $C_0\times T$, where $C_0$ is an elliptic curve, and satisfies $\pi|_{\mathcal{C}}={\rm Pr}_2$ via this biholomorphism. 
Assume that $N_{\mathcal{C}/\mathcal{S}}\cong {\rm Pr}_1^*N_0$ holds for some line bundle $N_0$ on $C_0$ which satisfies the Diophantine condition. 
Then, by shrinking $T$ if necessary, there exists a tubular neighborhood $\mathcal{W}$ of $\mathcal{C}$ in $\mathcal{S}$ 
which is isomorphic to a neighborhood of the zero section in $N_{\mathcal{C}/\mathcal{S}}$. 
\end{theorem}

See \S \ref{section_thm:rel_Arnol'd} for the proof of Theorem \ref{thm:rel_Arnol'd}. 
Take $\mathcal{W}\subset \mathcal{S}$ as in Theorem \ref{thm:rel_Arnol'd}. 
By shrinking $\mathcal{W}$ and considering the pull-back of an open covering $\{U_j\times T\}$ of $\mathcal{C}$ by the projection $\mathcal{W}\to \mathcal{C}$, we can take an open covering $\{\mathcal{W}_j\}$ of $\mathcal{W}$ and a coordinate system $(z_j, w_j, t)$ of each $\mathcal{W}_j$ which satisfies the following four conditions: 
$(I)$ $\mathcal{W}_j$ is biholomorphic to $U_j\times\Delta_R\times T$ ($R>1$), 
$(II)$ $\mathcal{W}_{jk}$ is biholomorphic to $U_{jk}\times\Delta_R\times T$, 
$(III)$ 
$\pi(z_j, w_j, t)=t$ holds on each $\mathcal{W}_j$, and 
$(IV)$ $(z_k, w_k, t)=(z_j+A_{kj}, t_{kj}\cdot w_j, t)$ holds on $\mathcal{W}_{jk}$, where $A_{kj}$ and $t_{kj}$ are as in the condition $(iv)$ in \S \ref{section:construction}. 
Denote by $\mathcal{V}$ the union of $\mathcal{V}_j$'s, where $\mathcal{V}_j:=\{(z_j, w_j, t)\in \mathcal{W}_j\mid 1/R^-<|w_j|<R^+\}$. 
Let $\mathcal{M}$ be the complement $\mathcal{S}\setminus \Phi^{-1}([0,1/R^-])$, where $\Phi : \mathcal{W} \to \mathbb{R}_{\ge 0}$ is given by 
$\Phi(z_j, w_j, t)=|w_j|$ on $\mathcal{W}_{j}$. 
Define a holomorphic function $F\colon \mathcal{V} \to \mathcal{S}^-:=S^-\times T$ by $F(z_j, w_j, t):=(g(z_j), w_j^{-1}, t)\in \mathcal{W}_j^-:=W_j^-\times T$ on each $\mathcal{V}_j$. 
By regarding $\mathcal{V}$ as a subset of $\mathcal{M}^-:=M^-\times T$ via $F$, we patch $\mathcal{M}$ and $\mathcal{M}^-$ to construct a manifold $\mathcal{X}$ just as in the previous section. 
The maps $\pi$ and $\pi^-:={\rm Pr}_2\colon \mathcal{S}^- \to T$ glue up to define a proper holomorphic submersion $P\colon \mathcal{X}\to T$. 
By the same argument as in the proof of Proposition \ref{prop:K3}, we have that each fiber $X_t:=P^{-1}(t)$ is a K3 surface. 
As before, by considering the Griffiths transversality and the calculation of the period integral in the previous section, one has that Kodaira-Spencer map $\rho_{{\rm KS}, P}$ is injective. 

\subsection{A deformation family fixing the Diophantine condition and the patching functions}\label{section:deform_fix_dioph_patch}
In this subsection, we consider a broader deformation family than that constructed in the previous subsection, 
that is, a family corresponding to the change of $C_0^{\pm}$, $N_0$ and $Z^{\pm}$ by fixing the Diophantine condition and the patching functions. 
As in the previous subsection, we also fix $Z^-$ and put $C_0:=C_0^{\pm}$ and $Z=\{p_1,\dots,p_9 \}:=Z^+$. 

Let $(p,q)$ be a pair of real numbers that satisfies the Diophantine condition, and fix $\tau_0 \in \mathbb{H}$ with $C_0 \cong \mathbb{C}/\langle 1, \tau_0 \rangle$, 
and a sufficiently small open neighborhood $U=U_{\tau}$ of $\tau_0$ in $\mathbb{H}$. 
For each $\tau \in U$, let $C_0(\tau)$ be a smooth elliptic curve given by $C_0(\tau) \cong \mathbb{C}/\langle 1, \tau \rangle$, and 
let $N_0(\tau) \in {\rm Pic}^0(C_0)$ be a line bundle given by $N_0(\tau) =p-q\tau$ via the identification ${\rm Pic}^0(C_0(\tau))\cong \mathbb{C}/\langle 1, \tau \rangle$. 
For a fixed sufficiently small open neighborhood $U_\nu$ of $p_\nu$ in $\mathbb{C}/\langle 1, \tau \rangle$ for each $\nu=1, 2, \dots, 8$,
we use a $9$-dimensional complex manifold $T:= U \times U_1\times U_2\times \cdots\times U_8$ as a parameter space in this subsection (More rigorously, each $U_\nu$ should be regarded not as a subset of $\mathbb{C}/\langle 1, \tau \rangle$ but as that of the universal cover $\mathbb{C}$ since $\tau$ also varies in this subsection). 
In what follows, we regard $t_0:=(\tau_0,p_1, p_2, \dots, p_8)$ as a base point of $T$. 
For each $t=(\tau, q_1, q_2, \dots, q_8)\in T$, we define $q(t)\in \mathbb{C}/\langle 1, \tau \rangle$ by 
$\mathcal{O}_{\mathbb{P}^2}(3)|_{C_0(\tau)}\otimes\mathcal{O}_{C_0(\tau)}(-q_1-q_2-\cdots-q_8-q(t))\cong N_0(\tau)$. 
Let $\pi\colon \mathcal{S}\to T$ be a proper holomorphic submersion from an $11$-dimensional complex manifold $\mathcal{S}$ to $T$ such that each fiber $S_t:=\pi^{-1}(t)$ is isomorphic to the blow-up of $\mathbb{P}^2$ at nine points $q_1, q_2, \dots, q_8$, and $q(t)$ for each $t=(\tau, q_1, q_2, \dots, q_8)\in T$. 
Let $\mathcal{C}\subset \mathcal{S}$ be the strict transform of $\{(x,(\tau,q_1,\dots,q_8)) \in \mathbb{P}^2 \times T \mid x \in C_0(\tau)\}$. 

An important issue here is the existence of a holomorphic tubular neighborhood $\mathcal{W}$ of $\mathcal{C}$ is $\mathcal{S}$, 
which is guaranteed by Theorem \ref{thm:rel_Arnol'd_gen}, a generalized version of Theorem \ref{thm:rel_Arnol'd}. 
See subsection \ref{section:more_gen_rel_arnol'd_thm} for more details. 
With such a holomorphic tubular neighborhood $\mathcal{W}$ of $\mathcal{C}$ is $\mathcal{S}$ in  hand, 
one can construct a deformation family $P\colon\mathcal{X}\to T$ such that each fiber $X_t:=P^{-1}(t)$ is a K3 surface, 
in a similar manner mentioned in the previous subsection. 

\subsection{A deformation family : summary}\label{section:construction_19_dim_family_summary}

In this subsection, we use a $19$-dimensional complex manifold $T:=\Delta_z \times \Delta_w \times U_{\tau} \times U_1^+\times \cdots\times U_8^+\times U_1^- \times \cdots \times U_8^-$ as a parameter space, 
where $\Delta_z$, $\Delta_w$ and $U_{\tau}$ are those in \S \ref{section:deform_patching} and \S \ref{section:deform_fix_dioph_patch}, and $U_\nu^{\pm}$ is a neighborhood of $p_\nu^{\pm}$ in $\mathbb{C}/\langle 1, \tau \rangle$. 
Let $(p,q)$ be a pair of real numbers that satisfies the Diophantine condition. 
By combining the constructions of the deformation families as in \S \ref{section:deform_patching}, \S \ref{section:deform_points} and \S \ref{section:deform_fix_dioph_patch}, one can naturally construct a deformation family $\pi\colon\mathcal{X}\to T$ and the subsets $\mathcal{M}^{\pm} \subset \mathcal{X}$ such that the following conditions hold: 
for each $t=(t_1, t_2, \tau, q_1^+, \dots, q_8^+, q_1^-,  \cdots, q_8^-)\in T$, 
$M_t^{\pm}:=\mathcal{M}^{\pm} \cap \pi^{-1}(t)$ is a subset of the blow-up of $\mathbb{P}^2$ at $\{q_1^{\pm}, \dots, q_8^{\pm}, q_9^{\pm} \}$, 
and $X_t:=\pi^{-1}(t)$ is a K3 surface obtained by patching $M_t^+$ and $M_t^-$ by identifying the point $(z_j, w_j)\in M_t^+$ with the point 
$(z_j^-(z_j^+, w_j^+), w_j^-(z_j^+, w_j^+)):=\left( g_{t_1}(z_j^+), t_2(w_j^+)^{-1}\right) \in M_t^-$, 
where $q_9^{\pm} \in C_0^{\pm}$ is the point such that $\mathcal{O}_{\mathbb{P}^2}(3)|_{C_0(\tau)}\otimes\mathcal{O}_{C_0(\tau)}(-q_1^{\pm}-\cdots-q_8^{\pm}-q_9^{\pm}) \cong (p-q \tau)^{\pm 1}$ with the identification ${\rm Pic}^0(C_0(\tau))\cong \mathbb{C}/\langle 1, \tau \rangle$, and $(z_j^{\pm}, w_j^{\pm})$ are the coordinates near the boundary of $M_t^{\pm}$ with $|t_2|/R^{\mp} <|w_j^{\pm}|<R^{\pm}$ and $(z_k^{\pm}, w_k^{\pm})=(z_j^{\pm}+A_{kj}, t_{kj}^{\pm 1} \cdot w_j^{\pm})$. 
Then, by shrinking the parameter space $T$ if necessary, we have the following proposition. 

\begin{proposition}\label{prop:deform_w_g_points}
The Kodaira--Spencer map $\rho_{{\rm KS}, \pi}\colon T_{T, t}\to H^1(X_{t}, T_{X_{t}})$ of the deformation family $\pi\colon \mathcal{X}\to T$ is injective for all $t\in T$. 
\end{proposition}

The injectivity of the Kodaira--Spencer map corresponding to $t_1$ and $t_2$ has already been shown in \S \ref{section:deform_patching}, 
while that corresponding to the rest parameters will be shown in \S \ref{section:realizability}. 
It is known that the moduli space of K3 surfaces has $20$ complex dimension. 
However, it is difficult to construct a deformation family $\pi\colon \mathcal{X}\to T$ with $\mathrm{dim}_{\mathbb{C}} T=20$, 
since the assumption that $(p,q)$ satisfies the Diophantine condition is essential in our construction. 

\begin{question}
What is the maximum dimension of a complex manifold $T$ over which there exists a deformation family $\pi\colon \mathcal{X}\to T$ whose fibers are K3 surfaces obtained by the construction as in \S \ref{section:construction} such that the Kodaira--Spencer map $\rho_{{\rm KS}, \pi}\colon T_{T, t}\to H^1(X_{t}, T_{X_{t}})$ is injective for all $t\in T$? 
\end{question}


\section{Proof of Theorem \ref{thm:main} and Corollary \ref{cor:main}}

\subsection{Proof of Theorem \ref{thm:main}}

First we notice that the K3 surfaces constructed in \S \ref{section:construction} admit Levi-flat hypersurfaces $\{H_t\}$ satisfying the condition in Theorem \ref{thm:main}. 
So in order to prove Theorem \ref{thm:main}, 
it is sufficient to show that a general K3 surface $X$ we constructed has $\rho(X)=0$, 
since $\rho(X) \ge 1$ if $X$ is projective, and $\rho(X) \ge 16$ if $X$ is Kummer. 
Here $\rho(X):=\mathrm{rank} \{ x \in H_2(X,\mathbb{Z}) \, | \, \int_{x} \sigma=0 \}$ is the Picard number of $X$. 
In the calculation of the integrals given in \S \ref{section:constr_cycles}, one can choose $(a_{\alpha},a_{\beta})$ and $\tau$ generally so that 
$\int_x \sigma \neq 0$ for any $x \in \mathbb{Z} A_{\alpha \beta} \oplus \mathbb{Z} A_{\beta \gamma} \oplus \mathbb{Z} A_{\gamma \alpha}$, 
and then choose $(p_1^{\pm},\dots,p_8^{\pm})$ generally so that  
$\int_x \sigma \neq 0$ for any 
$x \in (\mathbb{Z} A_{\alpha \beta} \oplus \mathbb{Z} A_{\beta \gamma} \oplus \mathbb{Z} A_{\gamma \alpha}) \oplus 
\bigoplus_{\nu=\pm} (\mathbb{Z} C_{12}^\nu \oplus \cdots \mathbb{Z} C_{78}^\nu \oplus \mathbb{Z} C_{678}^\nu)$. 
Finally, from the calculation of $\int_{B_{\gamma}} \sigma$ and Proposition \ref{prop:deform_w_g_points}, 
one can choose the patching parameters generally so that 
$\int_x \sigma \neq 0$ for any $2$-cycle $x \in H_2(X,\mathbb{Z})$ (here we used the Griffiths transversality on the relation between the Kodaira-Spencer map and the derivative of the period map, see \cite[p. 107, Proposition 2.4]{H}), which establishes Theorem \ref{thm:main}. 

\subsection{Proof of Corollary \ref{cor:main}}
One can easily deduce the corollary by considering the universal covering of a leaf of a Levi-flat hypersurface $H_t\subset X$ given in \S \ref{section:construction}. 
\qed


\section{Realizability in the period domain}\label{section:realizability}

In this section,  we discuss the realizability of our K3 surfaces in the period domain 
\[
\mathcal{D}_{\rm Period}:=\{ \xi \in \mathbb{P}(H_2(X,\mathbb{C})) \mid (\xi.\xi)=0, (\xi.\overline{\xi})>0 \}.
\] 
Let $X$ be a K3 surface constructed in \S \ref{section:construction},
and let $\sigma$ be the holomorphic $2$-form on X that is normalized as in \S \ref{section:constr_cycles}. 
Through the Poincar\'e duality, we identify $\widehat{\sigma}:=\sigma/(2\pi\sqrt{-1})$ with an element in $H_2(X,\mathbb{C})$, expressed as 
\[
\widehat{\sigma}=
a_{\alpha\beta}A_{\alpha\beta}+a_{\beta\gamma}A_{\beta\gamma}+a_{\gamma\alpha}A_{\gamma\alpha}
+b_\alpha B_\alpha + b_{\beta}B_\beta + b_\gamma B_{\gamma}
+\sum_{\bullet \in \{\pm\}}\Bigl( \sum_{j=1}^7 c_{j, j+1}^{\bullet} C_{j, j+1}^{\bullet} + c_{678}^{\bullet}C_{678}^{\bullet} \Bigl)
\]
with $a_\bullet, b_\bullet, c_\bullet^\pm \in \mathbb{C}$. 
From the arguments in \S \ref{section:constr_cycles}, we have the following proposition. 

\begin{proposition} \label{propositon:realizability}
With the notations in \S \ref{section:construction} and \S \ref{section:constr_cycles}, 
put $\mu:=\mu(N_{C/S}):=a_\beta-\tau\cdot a_\alpha$, 
$\Lambda_\pm=\Lambda_\pm(S^\pm, C^\pm):={\rm vol}_{\eta^\pm}(S^\pm\setminus W_{\rm max}^\pm)$ and  
$\Lambda:=\Lambda_++\Lambda_-$. 
Then we have the following: 
\begin{enumerate}
\item[$(i)$] $b_\alpha = \tau$, $b_{\beta} = 1$,  $b_\gamma=\mu$. \\
\item[$(ii)$] $\hat{\sigma}$ is orthogonal to $v:=A_{\alpha\beta}+a_\alpha\cdot A_{\beta\gamma}-a_\beta\cdot A_{\gamma\alpha}$ in $H_2(X, \mathbb{C})$. \\
\item[$(iii)$] $c_\bullet^{\pm}$ are determined uniquely by the positions of $(p_1^{\pm}, p_2^{\pm}, \dots, p_8^{\pm})$ in $(C_0^{\pm})^8$. \\
\item[$(iv)$] $a_{\alpha\beta} = 2\mu+\int_{\Gamma_9}dz$. \\
\item[$(v)$] By putting $x:=a_{\beta\gamma}, y:=a_{\gamma\alpha}$, we have  
\[
2\tau x+2y
+2\mu^2+2\mu\cdot\int_{\Gamma_9}dz-2\tau^2 -2+ \Theta(c^+,c^-)
=0
\]
and 
\[
2{\rm Re}(\overline{\tau} x)+2{\rm Re}(y)-2(|\tau|^2+1)
+4|\mu|^2+2{\rm Re}\left(\overline{\mu}\cdot\int_{\Gamma_9}dz\right)+\overline{\Theta}(c^+,c^-)
>\Lambda,
\]
where $\Theta(c^+,c^-)$ and $\overline{\Theta}(c^+,c^-)$ are constants depending only on the values $c^{\pm}:=(c_{12}^{\pm},\dots,c_{78}^{\pm}, c_{678}^{\pm})$. 
\end{enumerate}
\end{proposition}

Conversely, we have the following theorem. 

\begin{theorem} \label{theorem:realizability}
Let $(p, q)$ be a pair of real numbers that satisfies the Diophantine condition. 
For $v=v_{(p, q)}:=A_{\alpha\beta}+p\cdot A_{\beta\gamma}-q\cdot A_{\gamma\alpha}$, we put 
\[
v^{\perp}:=\{ \xi \in \mathcal{D}_{\rm Period} \mid (\xi. v)=0 \}. 
\]
Then there exists a mapping 
\[
\Lambda_{(p, q)}\colon \mathbb{H}_\tau\times \mathbb{C}^8_{c^+}
\times \mathbb{C}^8_{c^-}\times \mathbb{C}_g
\to \mathbb{R}_{\geq 0}
\]
such that the following holds: 
if $\xi \in v^{\perp}$ with an expression
\[
\xi =a_{\alpha\beta}A_{\alpha\beta}+a_{\beta\gamma}A_{\beta\gamma}+a_{\gamma\alpha}A_{\gamma\alpha}
+b_\alpha B_\alpha + b_{\beta}B_\beta + b_\gamma B_{\gamma}
+\sum_{\bullet \in \{\pm\}}\Bigl( \sum_{j=1}^7 c_{j, j+1}^{\bullet} C_{j, j+1}^{\bullet} + c_{678}^{\bullet}C_{678}^{\bullet} \Bigl)
\]
satisfies the following conditions, then $\xi$ is realized as the image of the period map of a K3 surface constructed in \S \ref{section:construction}. 
\begin{enumerate}
\item[$(a)$] $b_\beta\not=0$. In what follows, we normalize $b_{\beta}$ so that $b_{\beta}=1$. \\
\item[$(b)$] $b_\alpha \in \mathbb{H}$. \\
\item[$(c)$] For $x:=a_{\beta\gamma}, y:=a_{\gamma\alpha}, \tau:=b_\alpha, \mu:=b_\gamma$, 
\begin{eqnarray}
&&2{\rm Re}(\overline{\tau} x)+2{\rm Re}(y)-2(|\tau|^2+1)
+4|\mu|^2+2{\rm Re}\left(\overline{\mu}\cdot\int_{\Gamma_9}dz\right)+\overline{\Theta}(c^+,c^-) \nonumber \\
&>&\Lambda_{(p, q)}(b_\alpha, c^+, c^-, a_{\alpha\beta}), \nonumber 
\end{eqnarray}
where $c^{\pm}$ and $\overline{\Theta}(c^+,c^-)$ are given in Proposition \ref{propositon:realizability}. 
\end{enumerate}
\end{theorem}

We notice that the function $\Lambda_{(p, q)}$ is independent of the values $(x, y)=(a_{\beta\gamma}, a_{\gamma\alpha})$, 
where $(a_{\beta\gamma}, a_{\gamma\alpha})$ correspond to the integrals $\Bigl( \int_{B_\alpha}\sigma, \int_{B_\beta}\sigma \Bigr)$ 
that we still have not clarified. 

\begin{proof}
For a given $\xi \in v^{\perp}$, Proposition \ref{propositon:realizability} determines some of parameters of our K3 surfaces as follows:  
\begin{itemize}
\item The coefficients $(b_{\alpha},b_{\beta})$ determine $\tau$ and hence $C_0^{\pm}$. Moreover $b_{\gamma}$ is also determined from the condition 
$\xi \in v^{\perp}$. \\
\item The coefficients $c^{\pm}$ determine $(p_1^{\pm}, p_2^{\pm}, \dots, p_8^{\pm})$. Moreover $p_9^{\pm}$ and thus $Z^{\pm}$ are also determined 
from the condition $N_{C^{\pm}/S^{\pm}} \cong N_0^{\pm 1}$ with $N_0 \cong p - q \tau$. 
\item The coefficient $a_{\alpha \beta}$ determines $\int_{\Gamma_9}dz$ and hence $g$. 
\end{itemize}
Therefore the coefficients of $\xi$ other than $(x,y)=(a_{\beta\gamma}, a_{\gamma\alpha})$ uniquely determine the parameters other than the scaling of $w_j$'s and the choice of $R^{\pm}$. 
In particular, these coefficients determine $\Lambda_\pm={\rm vol}_{\eta^\pm}(S^\pm\setminus W_{\rm max}^\pm)$, 
which is independent of $(x,y)$. 
Hence by putting $\Lambda_{(p, q)}(b_\alpha, c^+, c^-, a_{\alpha\beta}-2b_\gamma):=\Lambda_++\Lambda_-$, 
we will show that if $(x,y)$ satisfies condition (c), then $\xi \in \mathcal{D}_{\rm Period}$ is realized as the period of a K3 surface we constructed. 
More precisely, we will show that the equality 
\[
\Delta_{C_0^{\pm},Z^{\pm},g}=D_{C_0^{\pm},Z^{\pm},g}
\]
holds, where $\Delta_{C_0^{\pm},Z^{\pm},g}$ the set of $(x,y)$ such that $\xi \in \mathcal{D}_{\rm Period}$ is realized 
as the period of a K3 surface we constructed, 
and $D_{C_0^{\pm},Z^{\pm},g}$ is the set of $(x,y)$ satisfying condition (c). 
Proposition \ref{propositon:realizability} says that the inclusion $\Delta_{C_0^{\pm},Z^{\pm},g} \subset D_{C_0^{\pm},Z^{\pm},g}$ holds. 
Since $D_{C_0^{\pm},Z^{\pm},g}$ is connected and $\Delta_{C_0^{\pm},Z^{\pm},g}$ is an open subset of $D_{C_0^{\pm},Z^{\pm},g}$ 
from the argument in \S \ref{section:construction_19_dim_family_summary} (see also Remark \ref{rmk:openness} below), it is enough to show that $\Delta_{C_0^{\pm},Z^{\pm},g}$ is a closed subset of $D_{C_0^{\pm},Z^{\pm},g}$, 
To this end, for any sequence $\{\xi_\nu=(x_\nu, y_\nu)\}\subset \Delta_{C_0^{\pm},Z^{\pm},g}$ with 
$\xi_\nu\to \xi_\infty=(x_\infty, y_\infty)\in D_{C_0^{\pm},Z^{\pm},g}$ as $\nu\to \infty$, 
we will show that $\xi_\infty\in \Delta_{C_0^{\pm},Z^{\pm},g}$. 

Let $W_{\rm max}^{\pm} \subset S^{\pm}$ and $R_{\rm max}^{\pm} >0$ be given is \S \ref{section:construction}. 
We can construct a K3 surface $X_{\nu}$ corresponding to $\xi_{\nu}$ by gluing 
$M_\nu^{\pm}:=S^{\pm} \setminus \{w \in W_{{\rm max}}^{\pm}\mid \Phi_{\pm}(w) \leq r_\nu^{\pm} \}$
via an identification 
\[
f_{\nu} : \{w \in W_{\rm max }^{+} \mid r_\nu^{+} <\Phi_{+}(w)< \ell_\nu^{+} \}
 \to \{ w \in W_{\rm max }^{-} \mid r_\nu^{-} <\Phi_{-}(w)< \ell_\nu^{-} \}
 \]
for some $0 < r_\nu^{\pm}<\ell_\nu^{\pm} < R_{\rm max}^{\pm}$. 
In particular, the Levi-flat hypersurfaces 
\[
H_{\ell_\nu^{\pm}}^{\pm}:=\{w\in W_{\rm max}^{\pm} \mid \Phi_{\pm}(w)= \ell_\nu^{\pm} \}
\]
are identified with the boundaries $\partial M_\nu^{\mp}$. 
Therefore the ratios $\ell_\nu^{\pm}/r_\nu^{\pm}$ coincide, which is denoted by $\lambda_\nu$. 
It should be noted that we can choose $\ell_{\nu}^{\pm}$ sufficiently close to $R_{\rm max}^{\pm}$ for each $\nu$, 
since for any $1<b^{\pm} <R_{\rm max}^{\pm} /\ell_\nu^{\pm}$ one may extend a domain of the identification $f$ as 
\[
f_{\nu} : \{ w \in W_{\rm max }^{+} \mid r_\nu^{+}/b^- <\Phi_{+}(w)< b^+ \cdot \ell_\nu^{+} \}
 \to \{ w \in W_{\rm max }^{-} \mid r_\nu^{-}/b^+ <\Phi_{-}(w)< b^- \cdot \ell_\nu^{-} \},
 \]
which gives rise to the same K3 surface $X_{\nu}$. 

As in the proof of Proposition \ref{prop:finiteness_of_W^*}, the volume of $X_\nu$ is given by
\[
\mathrm{vol}(X_{\nu})= \Lambda+4\pi \cdot \left(\int_C \sqrt{-1}\eta_C\wedge\overline{\eta_C}\right)
\cdot\left(\log\frac{R_{\rm max}^+}{\ell_\nu^+}+\log\frac{R_{\rm max}^-}{\ell_\nu^-}+\log \lambda_\nu\right). 
\] 
Since $\xi_\nu\to \xi_\infty$, a sequence of the volumes $\{ \mathrm{vol}(X_\nu) \}$ is convergent to a constant $> \Lambda$. 
Hence one may assume that there is a constant $P>1$ such that the estimates
\[
\frac{1}{P}<\log\frac{R_{\rm max}^+}{\ell_\nu^+}+\log\frac{R_{\rm max}^-}{\ell_\nu^-}+\log \lambda_\nu < P
\]
hold for each $\nu$. 
Moreover as $\ell_{\nu}^{\pm}$ is chosen sufficiently close to $R_{\rm max}^{\pm}$, 
we have the estimates 
\[
0 < \log\frac{R_{\rm max}^\pm}{\ell_\nu^\pm} < \frac{1}{4P}, \qquad 
\frac{1}{2P}< \log \lambda_\nu < P. 
\]
for each $\nu$. 
By passing to a subsequence if necessary, we assume that $\{ \lambda_\nu \}$ converges to a constant $\lambda_\infty\in [e^{1/(2P)}, e^P]$, 
$\{ \ell_\nu^{\pm} \}$ converges to a constant $\ell_\infty^{\pm} \in [e^{-1/(4P)}\cdot R_{\rm max}^{\pm}, R_{\rm max}^{\pm}]$, 
and thus $\{ r_\nu^{\pm} \}$ converges to $r_{\infty}^{\pm} :=\ell_{\infty}^{\pm}/\lambda_{\infty}<\ell_{\infty}^{\pm}$.  

Now put $a^{\pm}:=\sqrt{r_{\infty}^{\pm} \ell_{\infty}^{\pm}} \in (r_{\infty}^{\pm}, \ell_{\infty}^{\pm})$, 
and fix $\varepsilon>0$ so that $r_{\nu}^- < a^- - \varepsilon < a^-+\varepsilon < \ell_{\nu}^-$. 
Then, there is a positive integer $\nu_0>0$ such that $r_\nu^+<a^+<\ell_\nu^+$ for any $\nu \ge \nu_0$, 
which guarantees that $a_{\nu} \in (r_{\nu}^-,\ell_{\nu}^-)$ can be defined by $f_{\nu}(H_{a^+}^+)=H_{a_{\nu}}^-$, 
and $a_{\nu}$ satisfies $|a_{\nu} -a^-| \le \varepsilon$ for any $\nu \ge \nu_0$, 
where we notice that $\lim_{\nu \to \infty} a_{\nu}=a^-$ from our construction. 
Fix $\eta \in H_{a^+}^+$, and put $\eta_{\nu}:=f_{\nu}(\eta) \in H_{a_{\nu}}^-$ for any $\nu \ge \nu_0$. 
Since $\eta_{\nu}$ is contained in the compact set $\bigcup_{|a-a_{\nu}| \le \varepsilon} H_{a}^-$, 
by passing to a subsequence if necessary, $\{\eta_{\nu} \}$ converges to a point $\eta_{\infty}$ in $V_{\infty}^-$, where 
$V_{\infty}^{\pm}:=\{ w \in W_{\rm max}^{\pm} \mid r_\infty^{\pm}< \Phi_{\pm}(w) < \ell_\infty^{\pm}  \}$. 

Hence we can define a K3 surface $X_\infty$ by gluing 
$M_{\infty}^{\pm}:=S^{\pm} \setminus \{w \in W_{{\rm max}}^{\pm}\mid \Phi_{\pm}(w) \leq r_\infty^{\pm} \}$
via an identification $f_{\infty} : V_{\infty}^{+} \to V_{\infty}^{-}$, 
where $f_{\infty} : V_{\infty}^{+} \to V_{\infty}^{-}$ is a unique isomorphism characterized by 
the conditions $f_{\infty}(\eta)=\eta_{\infty}$ and $g \circ p^{+} = p^{-} \circ f_{\infty}$ with the natural projections 
$p^{\pm} \colon W_{\rm max}^{\pm} \to C^{\pm}$. 
Note that the automorphism group of the annulus $\{z \in \mathbb{C} \mid r< |z| <\ell \}$ 
is generated by $U(1)$-rotations and the map $z \mapsto r \ell/z$. 
The period of K3 surface $X_{\infty}$ corresponds to $\xi_\infty \in \Delta_{C_0^{\pm},Z^{\pm},g}$, 
which establishes the theorem. 
\end{proof}

\begin{remark}\label{rmk:openness}
Let $\pi\colon \mathcal{X}\to T$ be the family as in \S \ref{section:construction_19_dim_family_summary}. 
By Torelli Theorem, there exists a holomorphic map $i\colon T\to \mathcal{D}_{\rm Period}$ such that $i^*\mathcal{U}:=\mathcal{U}\times_{\mathcal{D}_{\rm Period}}T\to T$ coincides with $\pi\colon \mathcal{X}\to T$, where $\mathcal{U}\to \mathcal{D}_{\rm Period}$ is the universal family. 
By Proposition \ref{prop:deform_w_g_points}, $i$ is an embedding. 
Therefore, by considering an open set defined by $i(T)\cap D_{C_0^{\pm},Z^{\pm},g}$, one has that $\Delta_{C_0^{\pm},Z^{\pm},g}$ is open in $D_{C_0^{\pm},Z^{\pm},g}$. 
\end{remark}

The following corollary follows from Theorem \ref{theorem:realizability} and the relative variant of Arnol'd's theorem. 

\begin{corollary}\label{cor:realizability_main}
Let $\Xi_{(p, q)}$ be the set of $\xi \in v_{(p, q)}^\perp$ satisfying the conditions $(a), (b), and\ (c)$ in Theorem \ref{theorem:realizability}. 
Then there exists a proper holomorphic submersion $\pi \colon \mathcal{X}\to \Xi_{(p, q)}$ such that 
each fiber is a K3 surface constructed in \S \ref{section:construction} whose period map yields the identity map $\Xi_{(p, q)}\to \Xi_{(p, q)}$. 
\end{corollary}

Theorem \ref{thm:over_B_pq} follows from the above corollary. 

\section{Examples}

\subsection{An example of a K3 surface with an involution which switches $M^+$ and $M^-$}\label{section:involution_eg}

\subsubsection{Preliminary for the construction}

Let $N\to C$ be a flat line bundle over an elliptic curve $C$. 
Assume that $N$ is a non-torsion element of ${\rm Pic}^0(C)$. 
Take a flat metric $h$ of $N$. 
Denote by $\widehat{W}^{(r)}$ the subset $\{\xi\in N\mid |\xi|_h<r\}$ for each $r>0$, 
and by $\widehat{W}$ the set $\widehat{W}^{(1)}$, and by 
$\widehat{C}$ the zero-section. 

Let $\{U_j\}$ be an open covering of $C$, $z_j$ be a coordinate of $U_j$ with 
$z_j=z_k+A_{jk}$ for some $A_{jk}\in\mathbb{C}$ on each $U_{jk}$. 
Denote by $\pi\colon \widehat{W}\to C$ the restriction of the natural projection $N\to C$, 
and by $\widehat{W}_j$ the pull-back $\pi^{-1}(U_j)$. 
Let $w_j$ be a fiber coordinate of $\widehat{W}_j$ with $w_j=t_{jk}w_k$ on each  $\widehat{W}_{jk}$ for some $t_{jk}\in U(1)$. 
Denoting $\pi^*z_j:=z_j\circ \pi$ also by $z_j$, we regard $(z_j, w_j)$ as coordinates of $\widehat{W}_j$. Note that we may assume 
$|(z_j, w_j)|_h=|w_j|$ by scaling. 
For each $B\in \mathbb{C}^*$, we denote by ${\rm Rot}_B\colon \widehat{W}\to \widehat{W}$ the automorphism defined by ${\rm Rot}_B(z_j, w_j):=(z_j, B\cdot w_j)$ on each $\widehat{W}_j$.

\begin{proposition}\label{prop:aut_W_C}
Let $F\colon \widehat{W}\to N$ be an open holomorphic embedding such that 
$F(\widehat{C})=\widehat{C}$ and $F|_{\widehat{C}}={\rm id}_{\widehat{C}}$ hold. 
Then, there exists an element $B\in \mathbb{C}^*$ such that $F={\rm Rot}_B$ holds. 
\end{proposition}

\begin{proof}
Take an open covering $\{U_j^*\}$ of $C$ such that the index set $\{j\}$ coincides with that of $\{U_j\}$ and that $U_j^*\Subset U_j$ holds on each $j$ (i.e. $U_j^*$ is a relatively compact subset of $U_j$). 
Take a sufficiently small positive number $\varepsilon>0$ such that $F^{-1}\left(\widehat{W}^{(\varepsilon)}_j\right)\subset \widehat{W}_j$ holds, where $\widehat{W}^{(\varepsilon)}_j:=\pi^{-1}(U_j^*)\cap \widehat{W}^{(\varepsilon)}$. 
Define functions $p_j(z_j, w_j)$ and $q_j(z_j, w_j)$ on $\widehat{W}^{(\varepsilon)}_j$ by 
$p_j:=\left(F|_{F^{-1}\left(\widehat{W}^{(\varepsilon)}_j\right)}\right)^*z_j$ and $q_j:=\left(F|_{F^{-1}\left(\widehat{W}^{(\varepsilon)}_j\right)}\right)^*w_j$: 
i.e. $F|_{\widehat{W}^{(\varepsilon)}_j}(z_j, w_j)=(p_j(z_j, w_j), q_j(z_j, w_j))$. 
Consider the psh (plurisubharmonic) function $\varphi$ on 
$\widehat{W}$ defined by 
$\varphi|_{\widehat{W}_j}(z_j, w_j)=\log |w_j|$. 
Then it follows from Lemma \ref{lem:uniqueness_psh_varphi_nbhd_of_C_hat} below that there exists a constant $c\in \mathbb{C}$ such that 
$\log |q_j(z_j, w_j)|=c+\log |w_j|$ holds, which means that the holomorphic function 
$B_j$ on each $F^{-1}\left(\widehat{W}^{(\varepsilon)}_j\right)$ defined by
$B_j(z_j, w_j):=q_j(z_j, w_j)/w_j$ is a constant function with $|B_j|\equiv c$. 
Thus we have that $q_j(z_j, w_j)=B_j\cdot w_j$. 

Denote by $\theta_1$ and $\theta_2$ the global holomorphic vector field on
$\widehat{W}$ defined by 
\[
\theta_1:=\frac{\partial}{\partial z},\ \ 
\theta_2:=w\frac{\partial}{\partial w}, 
\]
and consider a global holomorphic function on a neighborhood of $\widehat{C}$ defined by 
$F_*\theta_2(dz_j)=\theta_2(dp_j)=w_j\frac{\partial p_j}{\partial w_j}$ (here we used the fact that $dz_j$'s glue up to define a global holomorphic $1$-form on $\widehat{W}$). 
As this function is a globally defined holomorphic function on $\widehat{W}^{(\delta)}$ for sufficiently small positive number $\delta$ which identically vanishes on $\widehat{C}$, it follows from Lemma \ref{lem:W*} that this function is the zero-map. Especially we have that 
\[
\frac{\partial p_j}{\partial w_j}\equiv 0
\]
holds on each $j$. 
As $p_j(z_j, 0)=z_j$ by the assumption, one has that 
$p_j(z_j, w_j)=z_j$ holds on each $j$. 

Take an element $x\in \widehat{W}_{jk}^{(\varepsilon)}$ and let
$x=(z_j, w_j)$ and $x=(z_k, w_k)$ be the corresponding coordinates on $\widehat{W}_j^{(\varepsilon)}$ and $\widehat{W}_k^{(\varepsilon)}$, respectively. 
Then now we have that 
$F(x)=(z_j, B_jw_j)=(z_k, B_kw_k)$ holds. 
Therefore we have that $w_j=t_{jk}w_k$ and 
$B_jw_j=t_{jk}B_kw_k$. 
Thus we obtain 
$B_j=B_k$, which proves the proposition. 
\end{proof}

\begin{lemma}\label{lem:uniqueness_psh_varphi_nbhd_of_C_hat}
Let $F$ and $\varphi$ be those in the proof of Proposition \ref{prop:aut_W_C}. 
Then there exists a constant $c\in \mathbb{R}$ such that 
$F^*\varphi=\varphi+c$ holds on a neighborhood of $\widehat{C}$, where $F^*\varphi=\varphi\circ F$. 
\end{lemma}

\begin{proof}
As $F^*\varphi$ is pluriharmonic on $\widehat{W}^{(\varepsilon)}\setminus\widehat{C}$, it follows from the same argument as in the proof of Lemma \ref{lem:W*} 
that $F^*\varphi|_{H_t}$ is a constant map for each $t\in (0, \varepsilon)$, where $H_t:=\{\xi\in N\mid |\xi|_h=t\}$, 
or equivalently, the function $F^*\varphi(z_j, w_j)$ depends only on $|w_j|$ 
(Note that here we used the assumption that $N$ is a non-torsion element of ${\rm Pic}^0(C)$). 
Thus there exists a function $\psi\colon (0, \varepsilon)\to \mathbb{R}$ with $\lim_{t\to 0}\psi(t)=-\infty$ such that $F^*\varphi(z_j, w_j)=\psi(|w_j|)$ holds on each $\widehat{W}^{(\varepsilon)}_j$. 
Note that $\psi$ is a $C^\infty$'ly smooth function by the regularity theorem. 
By the pluriharmonicity, one has the differential equation on $\psi$, by solving which we have that $\psi(t)=\alpha+\beta\cdot \log t$ for some constants $\alpha, \beta\in\mathbb{R}$. 
As $dd^c\varphi$ is the current $[\widehat{C}]$ defined by the integration along $\widehat{C}$ and $F$ is a biholomorphism which does not move $\widehat{C}$, one has that $dd^c(F^*\varphi)$ also coincides with $[\widehat{C}]$, which means that $\beta=1$. 
Thus the lemma follows by letting $c:=\alpha$. 
\end{proof}

Let $S$ be the blow-up of $\mathbb{P}^2$ at the nine points on a smooth cubic curve $C_0\subset \mathbb{P}^2$ such that $N:=N_{C/S}$ satisfies the Diophantine condition, where $C$ is the strict transform of $C_0$. 
Take a holomorphic tubular neighborhood $W$ of $C$ such that there is an isomorphism $H\colon\widehat{W}^{(R)}\to W$ from $\widehat{W}^{(R)}=\{\xi\in N\mid |\xi|_h<R\}$ for some $R>0$. 
Let $\tau\in \mathbb{H}$ be the modulus of $C$. 
In what follows, we sometimes identify $C$ with the quotient $\mathbb{C}/\langle 1, \tau \rangle$ of the complex plain $\mathbb{C}$ with the coordinate $z$. 
Denote by the involution $\iota\colon C\to C$ induced by the automorphism $z\mapsto -z$ of $\mathbb{C}$. 

As $N$ is a flat line bundle, one may regard $N$ as the quotient $\mathbb{C}^2/\sim$ of $\mathbb{C}^2$ with coordinates $(z, w)$, where ``$\sim$'' is the relation 
generated by $(z, w)\sim (z+1, t_1w) \sim (z+\tau, t_\tau w)$ for suitable elements $t_1, t_\tau\in {\rm U}(1)$. 
Note that the projection $N\to C$ is the one induced by $(z, w)\mapsto z$. 
Note also that the pull-back $\iota^*N$ can be regarded as $\mathbb{C}/\sim_\iota$, 
where ``$\sim_\iota$'' is the relation generated by 
$(z, w)\sim_\iota (-z-1, t_1w) \sim_\iota (-z-\tau, t_\tau w)$. 
As is observed by this easily, it holds that $\iota^*N\cong N^{-1}$. 
Denote by $I\colon N\to \iota^*N$ the map induced from the automorphism of $\mathbb{C}^2$ defined by $(z, w)\mapsto (-z, w)$. 

\begin{proposition}\label{prop:aut_S_C}
Let $F\colon S\to S$ be an automorphism such that $F(C)=C$ and $F|_{C}=\iota$ hold. 
Then there exists an element $B\in {\rm U}(1)$ such that 
$F|_W={\rm Rot}_B\circ I|_W$. 
Especially, it holds that $F(W)=W$ and that $F|_W$ is an automorphism of $W$. 
\end{proposition}

In Proposition \ref{prop:aut_S_C}, we are regarding ${\rm Rot}_B$ and $I$ as automorphisms of $W$ by using the isomorphism $H\colon\widehat{W}^{(R)}\to W$. 
The assertion in the proposition can be reworded as $H^{-1}\circ F\circ H={\rm Rot}_B\circ I$ on  $\widehat{W}^{(R)}$. 

\begin{proof}
Let $R_{\rm max}$ and $W_{\rm max}$ be those as in \S \ref{section:construction}. 
As is mentioned in the proof of Lemma \ref{lem:W_max}, 
$H$ can be extended to the isomorphism between $\widehat{W}^{(R_{\rm max})}$ and $W_{\rm max}$, which is also denoted by $H$. 
Define an open embedding $J\colon \widehat{W}^{(R_{\rm max})}\to S$ by $J:=F|_{W_{\rm max}}\circ H\circ I$. 
As $J|_{\widehat{C}}=C$, one has that $J({\widehat{W}^{(\varepsilon)}})\subset W_{\rm max}$ holds for a sufficiently small positive number $\varepsilon$. 
Denote by $\widehat{J}\colon \widehat{W}^{(\varepsilon)}\to \widehat{W}^{(R_{\rm max})}$ the map defined by 
$\widehat{J}:=H^{-1}\circ J$. 
As $\widehat{J}|_{\widehat{C}}={\rm id}_{\widehat{C}}$, it follows from Proposition \ref{prop:aut_W_C} that there exists a constant $B\in \mathbb{C}^*$ such that $\widehat{J}={\rm Rot}_B$ holds: i.e. $J=H\circ {\rm Rot}_B$ holds on $W^{(\varepsilon)}$. 
By replacing $F$ with $F^{-1}$ if necessary, we may assume $b:=|B|\geq 1$. 
By identity theorem, this equation also holds for any $\varepsilon$ with $\varepsilon \leq R_{\rm max}/b$. 
Therefore, it is sufficient to show that $|B|=1$. 

\[
\xymatrix{
\widehat{W}^{(R_{\rm max})} \ar@{}[dr]|\circlearrowleft \ar[r]^-J\ar[d]^I & S   &  \widehat{W}^{(\varepsilon)} \ar[r]^{J|_{\widehat{W}^{(\varepsilon)}}}\ar[rd]^{\widehat{J}}\ar@{=}[d] & W_{\rm max}\ar[d]^{H^{-1}}\\
\widehat{W}^{(R_{\rm max})} \ar[r]^-H & S\ar[u]^F  &  \widehat{W}^{(\varepsilon)} \ar[r]^-{{\rm Rot}_B} & \widehat{W}^{(R_{\rm max})}
}
\]

Consider the open embedding $K\colon \widehat{W}^{(b\cdot R_{\rm max})}\to S$ defined by $K:=J\circ {\rm Rot}_{B^{-1}}$. 
\[
\xymatrix{
\widehat{W}^{(b\cdot R_{\rm max})} \ar@{}[drr]|\circlearrowleft \ar[r]^{{\rm Rot}_{B^{-1}}}\ar@{}[d]|{\bigcup}&  \widehat{W}^{(R_{\rm max})}\ar[r]^-J  &  J(\widehat{W}^{(R_{\rm max})}) \\
\widehat{W}^{(\varepsilon)} \ar[rr]^-{\widehat{J}\circ {\rm Rot}_{B^{-1}}={\rm id}} & &\widehat{W}^{(\varepsilon)} \ar[u]_H
}
\]
As is observed by the commutativity of the diagram above, $K$ is an extension of $H$. Therefore, it follows by the property of $W_{\rm max}$ that $b\cdot R_{\rm max}\leq R_{\rm max}$, which means that $b\leq 1$. 
Thus one has that $b=1$, which proves the proposition. 
\end{proof}

\subsubsection{Construction of a K3 surface with an involution which switches $M^+$ and $M^-$}

Let $C_0$ be a smooth elliptic curve $\{[X; Y; Z]\mid Y^2Z=X(X-Z)(X-\lambda Z)\}\in\mathbb{P}^2$ for a fixed constant $\lambda\not=0, 1$. 
Denote by $C_0^\pm$ the copies of $C_0$. 
Fix any eight points $p_1^+, p_2^+, \dots, p_8^+\in C_0^+$. 
Take a point $p_9^+$ from $C_0^+$ such that 
the line bundle $\mathcal{O}_{\mathbb{P}^2}(3)|_{C_0^+}\otimes\mathcal{O}_{C_0^+}(-p_1^+-p_2^+\cdots -p_9^+)$ satisfies the Diophantine condition. 

Set $\iota_{\mathbb{P}^2}([X; Y; Z]):=[X; -Y; Z]$. 
Define $p_\nu^-$'s by $p_\nu^-:=\iota_{\mathbb{P}^2}(p_\nu^+)$ for $\nu=1, 2, \dots, 9$. 
Then, by letting $g\colon C_0^+\to C_0^-$ be the identity map obtained by the trivial identification $C_0^+=C_0=C_0^-$, it holds that 
\[
\mathcal{O}_{\mathbb{P}^2}(3)|_{C_0^-}\otimes\mathcal{O}_{C_0^-}(-p_1^--p_2^-\cdots -p_9^-)
=\iota_{\mathbb{P}^2}^*\left(\mathcal{O}_{\mathbb{P}^2}(3)|_{C_0^+}\otimes\mathcal{O}_{C_0^+}(-p_1^+-p_2^+\cdots -p_9^+)\right)
\]
and
\[
\mathcal{O}_{\mathbb{P}^2}(3)|_{C_0^-}\otimes\mathcal{O}_{C_0^-}(-p_1^--p_2^-\cdots -p_9^-)
=g^*\left(\mathcal{O}_{\mathbb{P}^2}(3)|_{C_0^+}\otimes\mathcal{O}_{C_0^+}(-p_1^+-p_2^+\cdots -p_9^+)\right)^{-1}. 
\]
In what follows, we use the notation as in \S \ref{section:construction}. 
Let $(S^\pm, C^\pm)$ be the models obtained by these nine points configurations. 
Denote by $\widetilde{F}\colon S^+\to S^-$ the isomorphism naturally induced by 
$\iota_{\mathbb{P}^2}$. 
Fix a real number $R>1$. Set $R^\pm:=R$. 
By scaling of the flat metric $h$ on $N_+$, we may assume that there exists a holomorphic tubular neighborhood $W^+$ of $C^+$ which is the image of an open embedding $\{\xi\in N_+\mid |\xi|_h<R\}\to S$. 
Set $W^-:=\widetilde{F}(W^+)$. 
Denote by $M^\pm$ the subset $S^\pm\setminus \Phi_\pm^{-1}([0, 1/R^\mp])$ of $S^\pm$.  Then it follows from Proposition \ref{prop:aut_S_C} that $\widetilde{F}(M^+)=M^-$. 
As we did in \S \ref{section:construction}, define a K3 surface by gluing $M^+$ and $M^-$ by identifying $V^\pm:=\Phi_\pm^{-1}((1/R^\mp, R^\pm))$ via the map $f$ defined by $f(z_j^+, w_j^+):=(g(z_j^+), 1/w_j^+)$. 
Note that this $f$ coincides with the restriction of the automorphism $\widetilde{f} \colon N_+\to N_-$ induced from the automorphism $(z, w)\mapsto (z, 1/w)$ of $\mathbb{C}^2$, where we are regarding $N_+$ as the quotient $\mathbb{C}^2/\sim$ (See the previous section for the relation ``$\sim$''). 
As $R^+=R^-$, it is easily observed that $\widetilde{F}(V^+)=V^-$ holds. 
Again by Proposition \ref{prop:aut_S_C}, it follows that the diagram 
\[
\xymatrix{
V^+ \ar[d]_f\ar[r]_{\widetilde{F}} & V^-\\
V^- \ar[r]_{\widetilde{F}^{-1}} & \ar[u]_f V^+. 
}
\]
is commutative. 
Thus we have that the automorphism $F$ on $X$ defined by 
\[
F(x):= \begin{cases}
    \widetilde{F}(x) & (x\in M^+) \\
    \widetilde{F}^{-1}(x) & (x\in M^-)
  \end{cases}
\]
is well-defined. 
This map $F$ is an involution which switches $M^+$ and $M^-$. 

\subsection{Some cases where $X$ is a Kummer surface}

In this section, we investigate the nine points constructions $Z^\pm$ such that the resulting K3 surface $X$ is the Kummer surface $K(Y)$ corresponding to a complex torus $Y$.

For simplicity, we here consider the case where  $Y$ can be written as $Y=$\linebreak$\mathbb{C}^2/\left\langle (1, 0), (a\sqrt{-1}, 0), (b_1\sqrt{-1}, b_2), (c_1\sqrt{-1}, c_2)\right\rangle$: 
i.e. $Y$ is the quotient of $\mathbb{C}^2$ with coordinates $(x_1, x_2)$ by the relation $\sim$ generated by 
\[
(x_1, x_2)
\sim (x_1+1, x_2)
\sim (x_1+a\sqrt{-1}, x_2)
\sim (x_1+b_1\sqrt{-1}, x_2+b_2)
\sim (x_1+c_1\sqrt{-1}, x_2+c_2), 
\]
where $a, b_1$, and $c_1$ are real numbers, and $b_2$ and $c_2$ are complex numbers which are $\mathbb{R}$-linearly independent. 
By considering a new coordinate $\xi_1:=\exp(2\pi\sqrt{-1}x_1)$, $Y$ can be also written as 
$Y=\mathbb{C}^*_{\xi_1}\times\mathbb{C}_{x_2}/\sim'$, where $\sim'$ is the relation generated by 
\[
(\xi_1, x_2)
\sim' (A\cdot \xi_1, x_2)
\sim' (B_1\cdot \xi_1, x_2+b_2)
\sim' (C_1\cdot \xi_1, x_2+c_2), 
\]
where $A:=e^{-2\pi a}$, 
$B_1:=e^{-2\pi b_1}$, and 
$C_1:=e^{-2\pi c_1}$. 
By construction, $Y$ admits an elliptic fibration $P\colon Y\to E$ onto the elliptic curve 
$E:=\mathbb{C}_{x_2}/\langle b_2, c_2\rangle$, which is the one induced by the second projection $(\xi_1, x_2)\mapsto x_2$. 
Complex tori $Y$ and $E$ admit involutions $i$ and $\overline{i}$ induced by $(\xi_1, x_2)\mapsto (\xi_1^{-1}, -x_2)$ and $x_2\mapsto -x_2$, respectively: 
\[
\xymatrix{
Y \ar[d]_{P} \ar[r]^i \ar@{}[dr]|\circlearrowleft & Y \ar[d]^{P} \\
E \ar[r]^{\overline{i}} & E \\
}\]
Consider Levi-flat hypersurfaces 
\[
H'_t:=\{[(x_1, x_2)]\in Y\mid {\rm Re}\,x_1\in t+\mathbb{Z}\}
=\{[(\xi_1, x_2)]\in Y\mid \xi_1\in \mathbb{R}_{>0}\cdot e^{2\pi\sqrt{-1}t}\}
\]
of $Y$ for each  $0\leq t< 1$. 

Denote by 
$\widetilde{Y}$ the blow-up of $Y$ at $16$ fixed points of $i$. 
The Kummer surface $K(Y)$ is the quotient of $\widetilde{Y}$ by the involution 
$\widetilde{i}$ induced by $i$. 
Denote by $\widetilde{H}_t$ the preimage of $H'_t$ by the blow-up map $\widetilde{Y}\to Y$, 
and by $H_t$ the image of $\widetilde{H}_t$ by the quotient map $\widetilde{Y}\to K(Y)$. 
It is easily observed that $H_t$ is a Levi-flat hypersurface of $K(Y)$ for each 
$t\not\in \mathbb{Z}\cup (1/2+\mathbb{Z})$. 
Note that $H_t=H_s$ if $s+t\in\mathbb{Z}$. 

Consider the subset $V^{(\varepsilon)}$ of $K(Y)$ defined by 
\[
V^{(\varepsilon)}:=\bigcup_{|t-1/4|<\varepsilon}H_t
\]
for each $\varepsilon$ with $0<\varepsilon<1/4$. 
For this set, we have the following: 

\begin{proposition}\label{prop:V_for_Kummer}
For each $\varepsilon$ with $0<\varepsilon<1/4$, $V^{(\varepsilon)}$ can be embedded into a flat line bundle $N$ on $E$ defined by 
$N:=\mathbb{C}_{w}\times\mathbb{C}_{x_2}/\sim_N$, where 
the relation $\sim_N$ is the one generated by 
$(w, x_2)
\sim (\exp(2\pi\sqrt{-1}b_1/a)\cdot w, x_2+b_2)
\sim (\exp(2\pi\sqrt{-1}c_1/a)\cdot w, x_2+c_2)$. 
\end{proposition}

\begin{proof}
As is easily observed, $V^{(\varepsilon)}$ is biholomorphic to 
\[
\{x_1\in \mathbb{C}\mid 1/4-\varepsilon \leq {\rm Re}\,x_1\leq 1/4+\varepsilon\}\times \mathbb{C}/\left\langle (a\sqrt{-1}, 0), (b_1\sqrt{-1}, b_2), (c_1\sqrt{-1}, c_2)\right\rangle.
\]
Thus, it is embedded into the set 
\[
\mathbb{C}^2/\left\langle (a\sqrt{-1}, 0), (b_1\sqrt{-1}, b_2), (c_1\sqrt{-1}, c_2)\right\rangle, 
\]
which is biholomorphic to $N\setminus \widehat{E}$, where $\widehat{E}$ is the zero-section 
of $N\to E$. 
\end{proof}

By Proposition \ref{prop:V_for_Kummer}, one may regard $V^{(\varepsilon)}$ as a counterpart of $V$ in \S \ref{section:construction}. 
In what follows, we fix $\varepsilon$ and denote $V^{(\varepsilon)}$ simply by $V$. 
Let $\Omega^+$ and $\Omega^-$ be two connected components of the complement $K(Y)\setminus V$. 
Set $M^\pm:=K(Y)\setminus \Omega^\mp$. 
By Proposition \ref{prop:V_for_Kummer}, one can embed $V$ into the compactification $\overline{N}:=\mathbb{P}(\mathbb{I}_E\oplus N)$ of $N$. 
Denote by $W^\pm$ two connected components of $\overline{N}\setminus \overline{V}$. 
By switching $W^+$ and $W^-$ if necessary, one can naturally patch $W^\pm$ and $M^\pm$ to construct compact complex surfaces $S^\pm$ by using $V$ as the tab for gluing. 
According to the classification theory, $S^\pm$ is either the blow-up of $\mathbb{P}^2$ at nine points or the blow-up of a Hirzebruch surface $\Sigma_n:=\mathbb{P}(\mathcal{O}_{\mathbb{P}^1}\oplus \mathcal{O}_{\mathbb{P}^1}(-n))$ at eight points. 

\begin{question}
When is $S^\pm$ isomorphic to the blow-up of $\mathbb{P}^2$ at nine points? 
\end{question}

Assume that both $S^+$ and $S^-$ are isomorphic to the blow-up of $\mathbb{P}^2$ at nine points. In this case, one can regard $K(X)$ as the one constructed in the manner as in \S \ref{section:construction} if $(b_1/a, c_1/a)$ is a Diophantine pair. 
Note that ${\rm vol}_{\eta^\pm}(S^\pm\setminus W_{\rm max}^\pm)=0$ holds in this case (cf. Question \ref{q:vol_0_or_posi}).

\subsection{Type II degeneration of K3 surfaces constructed by gluing method}\label{section:type_II_degeneration}

Fix elliptic curves $C_0^\pm$, nine points configurations $Z^\pm$ and an isomorphism $g$. 
Thus the complex structure of $S^\pm$ and the parameters $\tau$, $\mu=a_\beta-\tau\cdot a_\alpha$, and $\int_{\Gamma_9}dz^+$ are fixed. 
Let us denote by $(X_x, \sigma_x)$ the K3 surface corresponding to the class of the period domain represented by 
\[
\left(2\mu+\int_{\Gamma_9}dz^+\right)\cdot A_{\alpha\beta}+\mu\cdot B_\gamma
+x\cdot A_{\beta\gamma}
+\tau\cdot B_\alpha
+y\cdot A_{\gamma\alpha}
+B_{\beta}
+\sum c_\bullet^+ C_\bullet^+
+\sum c_\bullet^- C_\bullet^-, 
\]
where $y$ is the constant defined by the linear equation $y=-\tau\cdot x+N_1$ which comes from $(\sigma, \sigma)=0$. 
Note that $N_1$ is a constant which depends only on the choice of $C_0^\pm$, $Z^\pm$, and $g$, therefore it is a fixed constant in this section. 
As was observed in \S \ref{section:realizability}, 
the volume of $(X_x, \sigma_x)$ is calculated as below: 
\begin{eqnarray}
(\sigma_x, \overline{\sigma_x})
&=& 2({\rm Re}(\overline{\tau}\cdot x)+{\rm Re}(y)+N_2)\nonumber \\
&=& 2({\rm Re}(\overline{\tau}\cdot x)-{\rm Re}(\tau\cdot x)+{\rm Re}(N_1)+N_2)\nonumber \\
&=& 2((au+bv)-(au-bv)+{\rm Re}(N_1)+N_2)\nonumber \\
&=& 4bv + 2({\rm Re}(N_1)+N_2)\nonumber \\
&=& 4{\rm Im}\,\tau\cdot {\rm Im}\,x + N_3, \nonumber 
\end{eqnarray}
where $\tau=a+b\sqrt{-1}$ and $x=u+v\sqrt{-1}$ are the decomposition into the real part and the imaginary part, $N_2=N_2(C_0^\pm, Z^\pm, g)$ is a constant, and 
$N_3:=2({\rm Re}(N_1)+N_2)$ is a real constant. 
As ${\rm Im}\,\tau>0$, it follows from Theorem \ref{theorem:realizability} that there exists a constant $N_4=N_4(C_0^\pm, Z^\pm, g)$ such that $(X_x, \sigma_x)$ can be constructed by the gluing method as in the manner in \S \ref{section:construction} for any $x\in\mathbb{C}$ with ${\rm Im}\,x>N_4$. 

Take $x$ such that ${\rm Im}\,x>N_4$. 
Let $M^\pm_x\subset S^\pm$, $V^\pm_x\subset M^\pm_x$, and 
\[
f_x\colon V_x^+\ni (z_j^+, w_j^+)\mapsto (g(z_j^+), 1/w_j^+)\in V_x^-
\]
be the data by which $X_x$ is constructed as in \S \ref{section:construction}, where $(z_j^\pm, w_j^\pm)$'s are local coordinates of $W_{\rm max}^\pm$ as in Lemma \ref{lem:W_max}. 
Denote by $(X_x^{(a)}, \sigma_x^{(a)})$ the K3 surface constructed by patching $(M^\pm, \eta^\pm|_{M^\pm})$'s by identifying $V^+$ and $V^-$ via the biholomorphism 
\[
f_x^{(a)}\colon V_x^+\ni (z_j^+, w_j^+)\mapsto (g(z_j^+), e^{\sqrt{-1}a}/w_j^+)\in V_x^-
\]
for each $a\in\mathbb{R}$ (Note that $(X_x^{(0)}, \sigma_x^{(0)})=(X_x, \sigma_x)$). 

\begin{proposition}\label{prop:x_and_a}
Let $x\in\mathbb{C}$ be a constant with ${\rm Im}\,x>N_4$. 
Then it holds that $(X_x^{(a)}, \sigma_x^{(a)})=(X_{x-a}, \sigma_{x-a})$ for each $a\in\mathbb{R}$. 
\end{proposition}

\begin{proof}
Denote by $B_\alpha^{(a)}$ the $2$-cycle of $X_x^{(a)}$ which corresponds to $B_\alpha$. 
In what follows, we regard $B_\alpha^{(a)}$ as the sum of $T_a$, $D_\alpha^+$, and $D_\alpha^-$, where 
\[
T_a:=\{(z, w)\in \mathbb{C}\times \Delta_R\mid z\in\mathbb{R}, w=r\cdot e^{2\pi\sqrt{-1}\cdot (a\cdot \rho(r)+a_\alpha\cdot z)}\ \text{for some}\ r\in(1/R^-, R^+)\}/\sim
\]
and $D^\pm_\alpha$ is a topological disc with $\partial T_a=\partial D^+_\alpha\cup \partial D^-_\alpha$. 
Here we let the relation $\sim$ be the one generated by 
\[
(z, w)\sim (z+1, e^{2\pi\sqrt{-1}\cdot a_\alpha}\cdot w)\sim (z+\tau, e^{2\pi\sqrt{-1}\cdot a_\beta}\cdot w)
\]
and are identifying $V$ with $\{(z, w)\in \mathbb{C}\times \Delta_R\mid 1/R^-<|w|<R^+\}/\sim$. 
The function $\rho\colon \mathbb{R}\to [0, 1]$ is a cut-off function which is weakly monotonically decreasing such that ${\rm supp}\,\rho\subset (-\infty, R^+)$, and 
that $\rho$ is equivalently equal to $1$ on a neighborhood of $(-\infty, 1/R^-]$. 

For showing the proposition, it is sufficient to show that
\begin{equation}\label{eq:diff_a}
\frac{d}{da}\int_{B_\alpha^{(a)}}\frac{\sigma_x^{(a)}}{2\pi\sqrt{-1}}\equiv -1. 
\end{equation}
As 
\[
-\sigma_x^{(a)}=\frac{dz\wedge d(e^{\sqrt{-1}a}w)}{e^{\sqrt{-1}a}w}
=\frac{dz\wedge dw}{w}
\]
holds on $V^-$ and 
\[
\int_{B_\alpha^{(a)}}\frac{\sigma_x^{(a)}}{2\pi\sqrt{-1}}-\int_{B_\alpha}\frac{\sigma_x}{2\pi\sqrt{-1}}
=\int_{T_a}\frac{dz\wedge dw}{2\pi\sqrt{-1}w}-\int_{T_0}\frac{dz\wedge dw}{2\pi\sqrt{-1}w}
\]
holds, the equation (\ref{eq:diff_a}) follows from the following calculation: 
\begin{eqnarray}
\int_{T_a}\frac{dz\wedge dw}{w}
&=& \int_{\theta=0}^1\int_{r=1/R^-}^{R^+} \frac{\frac{d}{dr}(r\cdot e^{2\pi\sqrt{-1}\cdot (a\cdot \rho(r)+a_\alpha\cdot z)})\cdot drd\theta}{r\cdot e^{2\pi\sqrt{-1}\cdot (a\cdot \rho(r)+a_\alpha\cdot z)}} \nonumber \\
&=& \int_{\theta=0}^1\int_{r=1/R^-}^{R^+} \frac{(e^{2\pi\sqrt{-1}\cdot (a\cdot \rho(r)+a_\alpha\cdot z)}+2\pi\sqrt{-1}a\rho'(r)\cdot r\cdot e^{2\pi\sqrt{-1}\cdot (a\cdot \rho(r)+a_\alpha\cdot z)})\cdot drd\theta}{r\cdot e^{2\pi\sqrt{-1}\cdot (a\cdot \rho(r)+a_\alpha\cdot z)}} \nonumber \\
&=& \int_{\theta=0}^1d\theta\cdot \int_{r=1/R^-}^{R^+} \left(\frac{1}{r}+2\pi\sqrt{-1}a\rho'(r)\right)\, dr \nonumber \\
&=& \log(R^+R^-)-2\pi\sqrt{-1}\cdot a. \nonumber 
\end{eqnarray}
\end{proof}

As is observed easily, $(X_x^{(2\pi)}, \sigma_x^{(2\pi)})$ coincides with $(X_x, \sigma_x)$ if one omits the information on the marking. 
Therefore it follows from Proposition \ref{prop:x_and_a} that there exists a deformation of K3 surfaces (i.e. smooth holomorphic surjective submersion whose fibers are K3 surfaces) 
\[
\pi^*\colon \mathcal{X}^*\to B^* 
\]
on $B^*$ defined by 
\[
B^*:=\{x\in\mathbb{C}\mid {\rm Im}\,x>N_4\}/\sim_{B^*}, 
\]
where the relation $\sim_{B^*}$ is the one generated by $x\sim_{B^*} x+2\pi$ 
such that the fiber $(\pi^*)^{-1}(b)$ is isomorphic to $X_x$ for a preimage $x$ of $b$ by the quotient map $\{x\in\mathbb{C}\mid {\rm Im}\,x>N_4\}\to B^*$. 

In what follows, we regard as $B^*=\{b\in\mathbb{C}\mid 0<|b|<e^{-N_4}\}$ by using the coordinate $b:=e^{\sqrt{-1}x}$. 
Set $B:=\{b\in\mathbb{C}\mid |b|<e^{-N_4}\}$. 

\begin{proposition}\label{prop:type_II_degeneration}
There exists a proper holomorphic submersion 
\[
\pi\colon \mathcal{X}\to B
\]
from a smooth complex manifold $\mathcal{X}$ such that $\pi|_{\pi^{-1}(B^*)}$ coincides with $\pi^*$ and the central fiber $X_0:=\pi^{-1}(0)$ is a compact complex variety with normal crossing singularity whose irreducible components are $S^+$ and $S^-$ and whose singular part is the one obtained by identifying $C^+$ and $C^-$ by $g$. 
\end{proposition}

\begin{proof}
We may prove the proposition by replacing $N_4$ with a bit larger constant $N_4'$ than $N_4$. 
Take $x\in \mathbb{C}$ with ${\rm Im}\,x=N_4'$. 
Let $(C_0^\pm, Z^\pm, g, R^\pm)$ be the parameters by which the K3 surface $(X_x, \sigma_x)$ is constructed by the gluing methods as in \S \ref{section:construction}. 
As it is observed easily, we may assume that $R^+=R^-$. 
Denote them simply by $R$. 
In what follows, we use the notation as in \S \ref{section:construction} for these parameters: $X_x=M^+\cup M^-$, for example. 

Set 
$\mathcal{V}:=\{(z, w^+, w^-)\in\mathbb{C}^3\mid |w^+|<R, |w^-|<R, |w^+\cdot w^-|<1\}/\sim$, 
where $\sim$ is the relation generated by 
\[
(z, w^+, w^-)
\sim (z+1, e^{2\pi\sqrt{-1}a_\alpha}\cdot w^+, e^{-2\pi\sqrt{-1}a_\alpha}\cdot w^-)
\sim (z+\tau, e^{2\pi\sqrt{-1}a_\beta}\cdot w^+, e^{-2\pi\sqrt{-1}a_\beta}\cdot w^-). 
\]
Define the morphism $\pi_{\mathcal{V}}\colon \mathcal{V}\to B'$ by 
$\pi_{\mathcal{V}}(z, w^+, w^-):=w^+\cdot w^-\cdot e^{-N_4'}$, where 
$B'$ is the disc with radius $e^{-N_4'}$. 
By gluing this map and the second projections 
\[
\pi_{\mathcal{M}^+}\colon \mathcal{M}^+:=M^+\times B'\to B'
\]
and 
\[
\pi_{\mathcal{M}^-}\colon \mathcal{M}^-:=M^-\times B'\to B'
\]
in the following manner, the morphism $\pi\colon \mathcal{X}\to B'$ as the assertion of the proposition is constructed: 
On a point $b\in B'$, glue the fiber $M_b^+:=M^+\times\{b\}$ of $\mathcal{M}^+\to B'$ and 
the fiber $V_b$ of $\mathcal{V}\to B'$ by using the map 
\[
V^+\times\{b\}\ni (z^+, w^+, b)\mapsto
\left(z^+, w^+, \frac{b\cdot e^{N_4'}}{w^+}\right)\in \mathcal{V}, 
\]
and glue the fiber $M_b^-:=M^-\times\{b\}$ of $\mathcal{M}^-\to B'$ and 
the fiber $V_b$ of $\mathcal{V}\to B'$ by using the map 
\[
V^-\times\{b\}\ni (z^-, w^-, b)\mapsto
\left(g^{-1}(z^-), \frac{b\cdot e^{N_4'}}{w^-}, w^-\right)\in \mathcal{V}. 
\]
\end{proof}

Let $\pi\colon \mathcal{X}\to B$ be the one as in Proposition \ref{prop:type_II_degeneration}. 
Then it is easily observed that the monodromy automorphism $m_\ell\colon \Pi_{3, 19}\to \Pi_{3, 19}$ of the K3 lattice along a simple loop $\ell\subset B^*$ is defined by 
$B_\alpha\mapsto B_\alpha+A_{\gamma\alpha}$ and 
$B_\beta\mapsto B_\beta+A_{\beta\gamma}$ ($A_\bullet$'s, $B_\gamma$, and $C_\bullet^\pm$'s  are fixed).  Therefore it follows that $\pi$ is a type II degeneration of K3 surfaces (see \cite[Chapter 6, \S 5.3]{H}). 

\subsection{Gluing construction of one concrete example of K3 surfaces which are neither Kummer nor elliptic}

In this subsection, we construct a K3 surface that is neither Kummer nor elliptic. 
Let $C_0^{\pm} \cong \mathbb{C}/\langle 1, \sqrt{-1} \rangle \subset \mathbb{P}^2$ be an elliptic curve, 
and define $Z^{\pm}:=\{p_1^{\pm}, p_2^{\pm}, \dots, p_9^{\pm}\}$ by 
\[
p_j^{\pm}:=0\ {\rm mod}\, \langle 1, \sqrt{-1} \rangle \quad (j=1,\dots,8), \qquad p_9^{\pm}=\mp \mu \ {\rm mod}\, \langle 1, \sqrt{-1} \rangle,  
\]
and $g : C_0^+ \to C_0^-$ by $g(z)=z+2 \mu$, which satisfies $g(p_9^+)=p_9^-$, 
where $\mu:=-2^{\frac{1}{3}}$. 
It should be noted that the pair $(p, q):=(0, \mu)$ satisfies the Diophantine condition from Liouville's Theorem. 
As in \S \ref{section:construction} and \S \ref{section:constr_cycles}, 
if $\mathrm{Im}(x) > 2-\mu^2+\Lambda /2$, we can construct a marked K3 surface $(X, \sigma, (A_\bullet, B_\bullet, C_\bullet^{\pm}))$ with period 
\[
\sigma/(2\pi \sqrt{-1})=2\cdot \mu \cdot A_{\alpha\beta}+\mu \cdot B_\gamma
+x\cdot A_{\beta\gamma}
+\sqrt{-1}\cdot B_\alpha
+y\cdot A_{\gamma\alpha}
+B_{\beta} \quad 
(y=-\mu^2-x \sqrt{-1}). 
\]
From the configurations $Z^{\pm}$, the K3 surface $X$ admits some $(-2)$-curves. 
Namely, denote by $D_0^{\pm}$ the $(-2)$-curve on $X$ derived from the strict transform of a line in $\mathbb{P}^2$ under the blow-up at $Z^{\pm}$, 
and for $j=1,\dots, 7$, denote by $D_j^{\pm}$ the $(-2)$-curve on $X$ derived from the irreducible component of the exceptional divisor of the blow-up at $Z^{\pm}$ 
with the intersection given by 
\[
(D_i^{\pm}. D_j^{\pm})=\begin{cases}
-2 & (\text{if}\ i=j)\\
1 & (\text{if $i$ is jointed to $j$ in  Figure \ref{figure:dynkin}} ) \\
0 & (\text{otherwise}). 
\end{cases}
\]
\begin{center}
\begin{figure}[t]
\unitlength 0.1in
\begin{picture}( 31.4000,  9.2000)(  3.5000,-10.5000)
%
\special{pn 20}%
\special{pa 450 400}%
\special{pa 950 400}%
\special{fp}%
%
\special{pn 20}%
\special{pa 1050 400}%
\special{pa 1550 400}%
\special{fp}%
%
\special{pn 20}%
\special{pa 1650 400}%
\special{pa 2150 400}%
\special{fp}%
%
\special{pn 20}%
\special{pa 2350 400}%
\special{pa 2800 400}%
\special{dt 0.054}%
%
\special{pn 20}%
\special{pa 2880 400}%
\special{pa 3380 400}%
\special{fp}%
%
\special{pn 20}%
\special{pa 1600 450}%
\special{pa 1600 950}%
\special{fp}%
\put(17.0000,-10.5000){\makebox(0,0)[lb]{$0$}}%
\put(3.5000,-3.0000){\makebox(0,0)[lb]{$1$}}%
\put(9.5000,-3.0000){\makebox(0,0)[lb]{$2$}}%
\put(15.5000,-3.0000){\makebox(0,0)[lb]{$3$}}%
\put(21.5000,-3.0000){\makebox(0,0)[lb]{$4$}}%
\put(33.5000,-3.0000){\makebox(0,0)[lb]{$7$}}%
%
\special{pn 20}%
\special{ar 400 400 50 50  0.0000000 6.2831853}%
%
\special{pn 20}%
\special{ar 1000 400 50 50  0.0000000 6.2831853}%
%
\special{pn 20}%
\special{ar 1600 400 50 50  0.0000000 6.2831853}%
%
\special{pn 20}%
\special{ar 1600 1000 50 50  0.0000000 6.2831853}%
%
\special{pn 20}%
\special{ar 2200 400 50 50  0.0000000 6.2831853}%
%
\special{pn 20}%
\special{ar 3440 400 50 50  0.0000000 6.2831853}%
\end{picture}%
\caption{Dynkin diagram}
\label{figure:dynkin}
\end{figure}
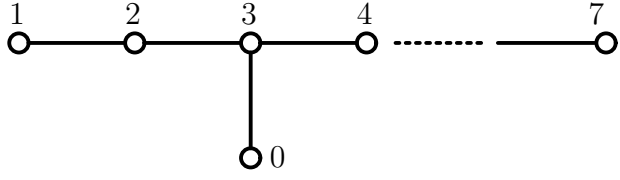
\end{center}
Moreover, by noting that $g(p_9^+)=p_9^-$, we have the Picard group
\[
{\rm Pic}(X)=Q_0 \oplus Q_+ \oplus Q_-, \qquad 
Q_0:=\mathbb{Z}\cdot \{B_\gamma\}, \quad 
Q_{\pm}:= \oplus_{j=0}^7 \mathbb{Z} \cdot \{D_j^{\pm}\}
\]
for a general  $x$ with $\mathrm{Im}(x) > 2-\mu^2+\Lambda /2$. 
In particular, the Picard number of $X$ is given by $\rho(X)=17$. 

\begin{proposition} \label{proposition:effdiv}
For a general $x$ with $\mathrm{Im}(x) > 2-\mu^2+\Lambda /2$, 
let $X$ be the corresponding K3 surface. 
Then, any effective divisor $D$ on $X$ is expressed as 
\[
D=a_\gamma D_\gamma+ \sum_{j=0}^7 a_j^+ D_j^+ + \sum_{j=0}^7 a_j^- D_j^-, 
\]
for some $a_\gamma, a_j^+, a_j^- \in \mathbb{Z}_{\ge 0}$, 
where $D_\gamma$ is a unique $(-2)$-curve with $\{ D_\gamma \}=\{B_{\gamma}\}$. 
In particular, $X$ is neither Kummer nor elliptic. 
The effective cone of $X$ is generated by the $(-2)$-curves $D_\gamma$ and $D_j^{\pm}$.  
\end{proposition}

\begin{remark}
The set of divisors on $X$ is generated by a so called exceptional family (see \cite{Bo}). 
In particular, Boucksom shows that the number of the members of any exceptional family of a complex surface is less than its Picard number. 
The K3 surface $X$ constructed in this subsection gives an example such that the number of the exceptional family, 
all of which are given explicitly, is equal to the Picard number. 

Moreover, since the exceptional family is given explicitly, the K\"ahler class of $X$ is also given explicitly by 
\[
\mathcal{K}_{X}=\{ \xi \in \mathcal{C}_X \mid  (\xi.D)>0 \, (D \in \{D_{\gamma},D_0^+,\dots,D_7^+,D_0^-,\dots,D_7^-\})\},
\] 
where $\mathcal{C}_X \subset H^{1,1}(X,\mathbb{R})$ is the positive cone of $X$ (see \cite[Chapter 8, \S 5]{H}). 
\end{remark}

In order to prove Proposition \ref{proposition:effdiv}, we prepare the following lemmata. 

\begin{lemma} \label{lemma:-2-curve01}
Let $D$ be an irreducible curve on $X$. 
Then $D$ is a $(-2)$-curve on $X$. 
Moreover its cohomology class $\{D\}$ belongs to one of  
$Q_0$, $Q_+$, or $Q_-$. 
\end{lemma}

\begin{proof}
Since $D$ is irreducible, the adjunction formula says that 
\[
p_a(D)=\frac{1}{2} \{  (K_X.D)+(D.D) \} + 1= \frac{1}{2} (D.D)+1, 
\]
where $K_X$ is the canonical class on $X$ and $p_a(D)$ is the arithmetic genus of $D$. 
As the intersection form on $\mathrm{Pic}(X)$ is even negative definite, we have 
$(D.D)=-2$ and $p_a(D)=0$, which means that $D$ is a $(-2)$-curve on $X$. 
Moreover since $(D.D)=-2$ is the maximal even negative integer and 
${\rm Pic}(X)=Q_0 \oplus Q_+ \oplus Q_-$ is an orthogonal decomposition with respect to 
the intersection form, $\{D\}$ must be contained in one of the components. 
\end{proof}

\begin{lemma} \label{lemma:-2-curve02}
Let $D$ be a $(-2)$-curve on $X$ whose class $\{D\}$ is contained in $Q_{\pm}$. 
Then we have $D=D_j^{\pm}$ for some $j=0,1,\dots,7$. 
\end{lemma}

\begin{proof}
As the intersection form on $Q_{\pm}$ is negative definite, 
there are at most finitely many class $\alpha \in Q_{\pm}$ with $(\alpha.\alpha)=-2$. 
In addition, if $\alpha \neq \pm \{D_j^{\pm} \}$ for any $j$, 
then $\alpha$ must satisfy $(\alpha.\{D_j^{\pm}\}) \ge 0$ for any $j$. 
A little calculation shows that such a class $\alpha$ is expressed uniquely as 
\[
\alpha=-3\{ D_0^{\pm}\} -2 \{D_1^{\pm}\} -4 \{D_2^{\pm}\} -6\{ D_3^{\pm}\} -5 \{D_4^{\pm}\} -4 \{D_5^{\pm}\} -3\{D_6^{\pm}\} -2\{D_7^{\pm}\}. 
\]
For a $(-2)$-curve $D$, one has $\{D\} \neq \alpha$ as $-\alpha$ is an effective class (see Remark \ref{remark:effective}). 
Hence we have $\{D\} = \pm \{D_j^{\pm} \}$ for some $j$. 
Since $\{D_j^{\pm} \}$ is effective and $D_j^{\pm}$ is irreducible, $D$ satisfies $\{D\} = \{D_j^{\pm} \}$ and thus $D=D_j^{\pm}$
(see Remark \ref{remark:effective}). 
\end{proof}

\begin{lemma} \label{lemma:-2-curve03}
Let $D$ be a $(-2)$-curve on $X$ whose class $\{D\}$ is contained in $Q_{0}$. 
Then we have $D=D_{\gamma}$,  where $D_\gamma$ is a unique $(-2)$-curve with $\{ D_\gamma \}=\{B_{\gamma}\}$. 
\end{lemma}

\begin{proof}
If there exists a $(-2)$-curve $D_{\gamma}$ with $\{D_\gamma\}=\{B_{\gamma} \}$ (the uniqueness follows from Remark \ref{remark:effective}), then 
it immediately follows that $D=D_{\gamma}$ (see also Remark \ref{remark:effective}). 
Hence we will prove the existence of a $(-2)$-curve $D_{\gamma}$ with $\{D_\gamma\}=\{B_{\gamma} \}$. 
The Hirzebruch-Riemann-Roch theorem says that 
\[
\chi(\mathcal{O}_X(\{B_{\gamma}\}))=\frac{1}{2}(\{B_{\gamma}\}. \{B_{\gamma}\} - K_X)+\frac{1}{12}\{ (K_X.K_X)+c_2(X) \}.
\]
As $K_X$ is trivial and $c_2(X)=24$, we have $\chi(\mathcal{O}_X(\{B_{\gamma}\}))=1$, which means that 
\[
h^0(X, \mathcal{O}_X(\{B_{\gamma}\}))+h^0(X, \mathcal{O}_X(-\{B_{\gamma}\}))=h^1(X, \mathcal{O}_X(\{B_{\gamma}\}))+1. 
\]
Here we note that 
$h^2(X, \mathcal{O}_X(\{B_{\gamma}\}))=h^0(X, K_X \otimes \mathcal{O}_X(-\{B_{\gamma}\}))=h^0(X, \mathcal{O}_X(-\{B_{\gamma}\}))$ 
by the Serre duality theorem. 
Hence either $h^0(X, \mathcal{O}_X(\{B_{\gamma}\}))$ or $h^0(X, \mathcal{O}_X(-\{B_{\gamma}\}))$ is positive, 
and 
there exists an effective divisor $D_{\gamma}$ such that either $\{ D_{\gamma} \}= \{ B_{\gamma}\}$ or $\{ D_{\gamma} \}= -\{ B_{\gamma}\}$ holds. 
Let $D_0$ be an irreducible component of $D_{\gamma}$ such that $(D_0.B_{\gamma}) \neq 0$. 
Then Lemma \ref{lemma:-2-curve01} says that $D_0$ is a $(-2)$-curve and either $\{ D_0 \}= \{ B_{\gamma}\}$ or $\{ D_0 \}= -\{ B_{\gamma}\}$ holds, 
which means that $D_{\gamma}=D_0$ is a $(-2)$-curve with either $\{ D_{\gamma} \}= \{ B_{\gamma}\}$ or $\{ D_{\gamma} \}= -\{ B_{\gamma}\}$. 

Now we will show that $\{ D_{\gamma} \}= \{ B_{\gamma}\}$ holds. To this end, we define a $(1,1)$-current $T$ on $S^+\setminus C^+$ by 
\[
T:=\frac{\sqrt{-1}}{2\pi}\partial\overline{\partial}\phi, 
\]
where
\[
\phi : S^+ \setminus C^+ \to \mathbb{R}_{\ge 0}, \quad 
\phi(p):=\begin{cases}
-\log \Phi_+(p) & (p\in W^+\setminus C^+, \Phi_+(p)<1)\\ 
0 & (\text{otherwise}). 
\end{cases}
\]
It is easily seen that $\phi$ is plurisubharmonic on $S^+\setminus C^+$, and is pluriharmonic outside the Levi-flat hypersurface $H_1:=\{p\in V^+ \mid \Phi_+(p)=1\}$. 
So $T$ is a semi-positive closed $(1,1)$-current on $S^+\setminus C^+$ with support $H_1$, and hence 
$T$ is naturally regarded as that on $X$. 
We notice that $\{T\}=\{C^+\} \in H^{1, 1}(S^+, \mathbb{R})$, and 
from the construction of $B_\gamma$, we have 
\[
(\{T\}.\{ B_{\gamma} \})=\int_{B_\gamma}T=1. 
\]
On the other hand, the integral
\[
(\{T\}. \{ D_{\gamma} \})=\int_{D_{\gamma}} T =\int_{D_{\gamma}} \frac{\sqrt{-1}}{2\pi}\partial\overline{\partial}(\phi|_{D_{\gamma} \cap V}) 
\]
is well-defined as the $(-2)$-curve $D_{\gamma}$ is not contained in the support $H_1$ of $T$. 
Indeed, if $D_{\gamma}$ is contained in $H_1$, then the restriction of the natural projection $p : W^+ \to C^+$ induces 
a holomorphic map $D_{\gamma} \to C^+$, which is surjective since if it were not, its image $p(D_{\gamma})$ would be a point $c$ and 
the rational curve $D_{\gamma}$ would be contained in a disk $p^{-1}(c)$. 
However, the existence of the surjective holomorphic map $D_{\gamma} \to C^+$ contradicts that the genus of $D_{\gamma}$ is distinct from that of $C^+$. 

Since the integral is non-negative, we have $(\{T\}. \{ D_{\gamma} \}) \ge 0$ and thus $(\{T\}. \{ D_{\gamma} \})=(\{T\}. \{ B_{\gamma} \})=1$, 
which shows that $\{ D_{\gamma} \}=\{ B_{\gamma} \}$. 
\end{proof}

\begin{proof}[Proof of Proposition \ref{proposition:effdiv}]
Proposition \ref{proposition:effdiv} follows from Lemmas \ref{lemma:-2-curve01}--\ref{lemma:-2-curve03}. 
Note that $X$ is non-elliptic as there is no elliptic curve on $X$, and 
$X$ is non-Kummer as there is no $16$ disjoint $(-2)$-curves among $\{D_\gamma, D_0^+,\dots, D_7^+,D_0^-,\dots, D_7^-\}$.
\end{proof}

\begin{remark} \label{remark:effective}
In our arguments, we use the following well-known facts. 
\begin{enumerate} 
\item If $D$ is an effective divisor on a compact complex surface $X$, then the cohomology class $-\{D\}$ is never effective. 
Indeed, if $-\{D\}$ is effective with an expression $-\{D\}=\{D_0\}$ for some effective divisor $D_0$, 
then $D+D_0$ satisfies $\{D+D_0\}=0$, and thus $D+D_0$ is the zero of a holomorphic function on $X$, 
which contradicts that $X$ is compact. 
\item If $D$ is an irreducible curve on a compact complex surface $X$ with $(D.D)<0$, 
then any effective divisor $D_0$ on $X$ with $\{D_0\}=\{D\}$ satisfies $D_0=D$. 
Indeed, if $D_0 \neq D$, then $D$ is not contained in $D_0$ and hence $D$ satisfies $(D.D_0) \ge 0$, 
which contradicts that $(D.D_0)=(D.D)<0$. 
\end{enumerate}
\end{remark}

\begin{remark}
It is known that if $X$ is a projective K3 surface with $\rho(X)\geq 5$ then $X$ is an elliptic surface 
(see \cite[Chapter 11, Proposition 1.3 $(ii)$]{H}). 
The K3 surface $X$ constructed in this subsection, which is non-projective, gives an example such that $\rho(X)\geq 5$ but $X$ is non-elliptic. 
\end{remark}

\subsection{Towards the gluing construction of McMullen's examples}

In this subsection, we consider automorphisms on K3 surfaces constructed in \cite{M1,M2}. 
Let $L$ be a lattice with $\mathrm{rank}\,L=r$, which is a free $\mathbb{Z}$-module $L \cong \mathbb{Z}^r$
together with a non-degenerate $\mathbb{Z}$-valued symmetric bilinear form 
\[
( \cdot . \cdot ) : L \times L \to \mathbb{Z}. 
\] 
We denote by $O(L)$ by the set of isometries $f : L \to L$, that is, $( f(x).f(y) )=( x.y )$ 
for any $x,y \in L$. 
Note that any isometry $f \in O(L)$ admits its linear extensions to 
$L \otimes \mathbb{Q}$, $L \otimes \mathbb{R}$ 
and $L \otimes \mathbb{C}$. 
The lattice $L$ is said to be {\it even} if $( x. x ) \in 2  \mathbb{Z}$ for any $x \in L$. 
The {\it dual lattice} $L^{\vee}$ of $L$ is defined by
\[
L^{\vee}:=\{x \in L \otimes \mathbb{Q} \, | \, ( x.L ) \subset \mathbb{Z} \}. 
\]
Then we have $L \subset L^{\vee} \subset L \otimes \mathbb{Q}$, and we say that $L$ is {\it unimodular} if $L=L^{\vee}$. 
The signature of $L$, denoted by $(s_+,s_-)$, is that of the induced bilinear form on $L \otimes \mathbb{R}$. 
It is known that an even unimodular lattice exists if and only if its signature $(s_+,s_-)$ satisfies $s_+ -s_- \equiv 0~\mathrm{mod}~8$. 
Moreover if $s_{\pm}>0$ and $s_+ -s_- \equiv 0~\mathrm{mod}~8$, then an even unimodular lattice with signature $(s_+,s_-)$ is unique up to lattice isometries. 
\par
Let $X$ be a K3 surface. Then its cohomology group $H^2(X; \mathbb{Z})$ together with the cup product is a unique even unimodular lattice 
$\Pi_{3,19} \cong 3U \oplus 2 E_8(-1)$ of signature $(3,19)$, where $U$ and $E_8(-1)$ are even unimodular lattices of signatures $(1,1)$ and $(0,8)$ respectively. 
The complex structure of $X$ yields the Hodge decomposition 
\[
H^2(X;\mathbb{C})=[H^{2,0}(X) \oplus H^{0,2}(X)] \oplus H^{1,1}(X), 
\]
where $H^{i,j}(X)=\overline{H^{j,i}(X)}$, and $[H^{2,0}(X) \oplus H^{0,2}(X)]$ and $H^{1,1}(X)$ have signature $(2,0)$ and $(1,19)$ respectively. 
Moreover, since every K3 surface is K\"ahler, $X$ admits the K\"ahler cone $\mathcal{C}_X \neq \emptyset \subset H^{1,1}(X) \cap H^2(X;\mathbb{R})$, 
the classes represented by the symplectic forms of K\"ahler metrics on $X$. 
\par 
In the setting, a {\it K3 structure} on $L=\Pi_{3,19}$ consists of the following data: 
\begin{enumerate}
\item a Hodge decomposition $L \otimes \mathbb{C}=[L^{2,0} \oplus L^{0,2}] \oplus L^{1,1}$, 
where $L^{i,j}=\overline{L^{j,i}}$, and $[L^{2,0} \oplus L^{0,2}]$ and $L^{1,1}$ have signature $(2,0)$ and $(1,19)$ respectively, 
\item a K\"ahler cone $\mathcal{C} \subset L^{1,1}_{\mathbb{R}}:=L^{1,1} \cap (L \otimes \mathbb{R})$, a connected component of 
$\{ x \in L^{1,1}_{\mathbb{R}} \, | \, (x.x) >0, (x.y) \neq 0~({}^\forall y \in \Psi) \}$, where 
$\Psi:=\{ y \in L^{1,1} \cap L \, | \, ( y.y) =-2 \}$. 
\end{enumerate}
It is known that for any K3 structure on $L$, there is a unique K3 surface $X$ and an isomorphism $\iota : L \to H^2(X;\mathbb{Z})$ such that
$\iota(L^{i,j})=H^{i,j}(X)$ and $\iota (\mathcal{C})=\mathcal{C}_X$. 
Furthermore, any $f \in O(L)$ satisfying $f(L^{i,j})=L^{i,j}$ and $f(\mathcal{C})=\mathcal{C}$ is realized by a unique automorphism $F : X \to X$, 
that is, $F^* \circ \iota=\iota \circ f$. 
\par
Now let $L$ be a general even lattice. Then the {\it glue (discriminant) group} $G(L):=L^{\vee}/L$ admits 
the discriminant form $q_L : G(L) \to \mathbb{Q}/ \mathbb{Z}$, $q_L(x)=(1/2) ( x.x ) ~\mathrm{mod}~ \mathbb{Z}$. 
For two even lattices $L_1$ and $L_2$, assume that there is an isomorphism $\phi : G(L_1) \to G(L_2)$, called a {\it gluing map}, 
such that $q_{L_1}(x)+q_{L_2}(\phi(x))=0$ for any $x \in G(L_1)$. 
Then 
\[
L:=\{ x_1+x_2 \in L_1^{\vee} \oplus L_2^{\vee} \, | \, \phi(\bar{x}_1)=\bar{x}_2 \}
\]
becomes an even unimodular lattice, where $\bar{x}_i$ denotes the projection of $x_i \in L_i^{\vee}$ to $G(L_i)=L^{\vee}_i/L_i$. 
Moreover, two isometries $f_i \in O(L_i)$ satisfying $\phi \circ \bar{f}_1=\bar{f}_2 \circ \phi$ yield an isometry $f \in O(L)$, 
where $\bar{f}_i : G(L_i) \to G(L_i)$ is the induced action.  
This is a method constructing lattice automorphisms. 
\\ {\bf Kummer surface automorphism} : Let $Y=\mathbb{C}^2/\Lambda$ be a $2$-dimensional complex torus, where $\Lambda:=\oplus_{i=1}^4 \mathbb{Z} v_i$
is a lattice generated by $v_i \in \mathbb{C}^2$, and let $\iota : Y \to Y$ be the involution given by $\iota(y)=-y$. 
The fixed points of $\iota$ correspond to the $16$ double points of $Y/\iota$, explicitly given by 
\[
\{ v_t:=\frac{1}{2} \sum_{i=1}^4 t_i v_i \, | \, t=(t_1,t_2,t_3,t_4) \in (\mathbb{Z}/2\mathbb{Z})^4 \} \subset Y/\iota. 
\]
By blowing up these $16$ points of $Y/\iota$, we have the {\it Kummer surface} $X=\kappa(Y)$. 
The natural degree $2$ rational map is denoted by $\pi : Y \dashrightarrow X$. 
The integral class in $H^2(X,\mathbb{Z})$ obtained from blowing up $v_t$ is denoted by 
$E_t \in H^2(X,\mathbb{Z})$. 
Moreover, let $K \subset H^2(X;\mathbb{Z})$ be the so called {\it Kummer lattice}, that is, 
the minimal sublattice $K \subset H^2(X;\mathbb{Z})$ such that $H^2(X;\mathbb{Z})/K$ is torsion-free and 
$K$ contains the $16$ classes $E_t$, and let $L \subset H^2(X;\mathbb{Z})$ be the image of $\pi_* : H^2(Y;\mathbb{Z}) \to H^2(X;\mathbb{Z})$.  
Then we have the following facts: 
\begin{enumerate}
\item The sublattice $K$ is spanned by the basis $E_t$ and all elements $E_W$ with hyperplanes 
$W \subset (\mathbb{Z}/2\mathbb{Z})^4$ i.e. $W$ is given by the equation of the form 
$\sum a_i \cdot t_i=c$ for some $a_i,c \in \{0,1\}$ with $(a_1,\dots,a_4,c) \neq (0,\dots,0)$. 
Here, for a subset $V \subset (\mathbb{Z}/2\mathbb{Z})^4$, put 
\[
E_V:= \frac{1}{2} \sum_{t \in V} E_t. 
\]
The glue group $G(K)$ is isomorphic to $(\mathbb{Z}/2\mathbb{Z})^6$ and is generated by 
$E_{ij}:=E_{A_{ij}}$ for $i < j \in \{1,2,3,4\}$, where 
$A_{ij}=\{ t=(t_1,t_2,t_3,t_4) \in (\mathbb{Z}/2\mathbb{Z})^4 \, | \, t_k=0 ({}^\forall k \notin \{i,j\} ) \}$. 
Moreover, by noting 
\[
( E_{ij}.E_{kl} ) =
\left\{ 
\begin{array}{cl}
-2 & (\# (\{i,j\} \cap \{k,l\})=2) \\
-1 & (\# (\{i,j\} \cap \{k,l\})=1) \\
-1/2 & (\# (\{i,j\} \cap \{k,l\})=0), 
\end{array}
\right. 
\]
we have 
\[
q_K (\sum_{i<j}t_{ij} E_{ij})= \frac{1}{2} (t_{12} t_{34}+t_{13}t_{24}+t_{14}t_{23}) ~~ \mathrm{mod} ~~ \mathbb{Z} \qquad (t_{ij} \in \{0,1\}). 
\] 
\item The sublattice $L$ is spanned by $2V_{ij}:=\pi_*(v_i \vee v_j) \in H^2(X;\mathbb{Z})$ for $i < j \in \{1,2,3,4\}$. 
The glue group $G(L)$ is isomorphic to $(\mathbb{Z}/2\mathbb{Z})^6$ and is generated by 
$V_{ij}$ for $i \neq j \in \{1,2,3,4\}$. 
Moreover, by noting 
\[
( V_{ij}.V_{kl} ) =
\left\{ 
\begin{array}{cl}
\frac{1}{2} \mathrm{sgn} 
\begin{pmatrix}
1 & 2 & 3& 4 \\
i & j & k & \ell 
\end{pmatrix} & (\# (\{i,j\} \cap \{k,l\})=0) \\
0 & (\# (\{i,j\} \cap \{k,l\}) \ge 1), 
\end{array}
\right. 
\]
we have 
\[
q_L (\sum_{i<j}t_{ij} V_{ij})= \frac{1}{2} (t_{12} t_{34}+t_{13}t_{24}+t_{14}t_{23}) ~~ \mathrm{mod} ~~ \mathbb{Z} \qquad (t_{ij} \in \{0,1\}). 
\] 
\item By the gluing map $\phi : G(K) \to G(L)$ given by $\phi(E_{ij})=V_{ij}$, we obtain 
the cohomology group $H^2(X;\mathbb{Z})$. 
\end{enumerate}
Any automorphism $f : Y \to Y$ commutes with the involution $\iota$ and thus induces automorphism $f : Y/\iota \to Y/\iota$ permuting 
the double points $f : \{ v_t\} \to \{ v_t\}$. 
Therefore $f$ gives rise to an automorphism $F = \kappa(f)  : X \to X$. 
On the other hand, $f$ induces lattice automorphisms $f^* : K \to K$ given by the permutation $f^* : \{ v_t\} \to \{ v_t\}$, and 
$f^* : L \to L$ given by the action $f^* : H^2(Y;\mathbb{Z}) \cong \oplus_{i<j} \mathbb{Z}  v_i \vee v_j$. 
It is seen that two automorphisms are compatible with the gluing map and hence yield an automorphism on $H^2(X;\mathbb{Z})$, which 
is the same as the action $F^* : H^2(X;\mathbb{Z}) \to H^2(X;\mathbb{Z})$. 
\par
For any integer $a \ge 0$, McMullen \cite{M1} considered an automorphism $f \in \mathrm{Aut}(Y)$ such that the characteristic polynomial of $f^* |H^1(Y)$ is given by
$t^4+at^2+t+1$. The automorphism $f$ can be constructed for example from a lattice automorphism $f^* : \Lambda \to \Lambda$ defined by 
\[
\begin{pmatrix}
0 & 0 & -1 & 0 \\
1 & 0 & 0 & 0 \\
0 & 1 & 1 & a+1 \\
0 & 0 & -1 & -1 
\end{pmatrix}.
\]
The automorphism $f$ induces the action $f^*$ on 
$L=2 \mathbb{Z} V_{12} \oplus 2 \mathbb{Z} V_{13} \oplus 2 \mathbb{Z} V_{14} \oplus 2 \mathbb{Z} V_{23} \oplus 2 \mathbb{Z} V_{24} \oplus 2 \mathbb{Z} V_{34}$, given by
\[
\begin{pmatrix}
0 & 1 & 0 & 0 & 0 & 0 \\
0 & 0 & 0 & 1 & 0 & -(a+1) \\
0 & 0 & 0 & 0 & 0 & 1 \\
1 & 1 & a+1 & 0 & 0 & 0 \\
0 & -1 & -1 & 0 & 0 & 0 \\
0 & 0 & 0 & -1 & -1 & a 
\end{pmatrix}.
\]
Its characteristic polynomial is given by $S(t)=t^6-a t^5-t^4+(2a-1)t^3-t^2-at+1$ and 
the holomorphic $2$-form $\sigma$ corresponds to 
\[
\sigma(s):=-\Bigl( \frac{s}{s^2-1}+\frac{1}{s^2} \Bigr) V_{12}- \Bigl( \frac{s^2}{s^2-1}+\frac{1}{s} \Bigr) V_{13}+\frac{1}{s}V_{14}
+ \Bigl( a-\frac{s^3}{s^2-1} \Bigr) V_{23}+\frac{s}{s^2-1} V_{24}+V_{34},  
\]
where $s \in U(1)$ is given by $S(s)=0$ and $( \sigma(s).\sigma(\bar{s}) ) >0$. 
Then by putting $A_{\alpha,\beta}=2V_{14}$, $A_{\beta,\gamma}=2V_{13}$, $A_{\gamma,\alpha}=2V_{12}$, 
$B_{\gamma}=V_{23}+E_{23}$, $B_{\alpha}=-2V_{14}-V_{24}-E_{24}$, $B_{\beta}=2V_{14}+2 V_{13}+V_{34}+E_{34}$, 
we have, up to scale, 
\[
\int_{A_{\alpha,\beta}} \sigma=a_{\beta}- \tau \cdot a_{\alpha}=r_1(s), \quad 
\int_{A_{\beta,\gamma}} \sigma=\tau=r_2(s), \quad 
\int_{A_{\gamma,\alpha}} \sigma=1
\]
with
\[
r_1(s)=a-\frac{s^3}{s^2-1}, \qquad r_2(s)=-\frac{s}{s^2-1}. 
\]
In particular, $(a_{\alpha},a_{\beta})$ is given by
\[
a_{\alpha}=\frac{r_1(s)-r_1(\bar{s})}{r_2(\bar{s})-r_2(s)} =-1+\frac{|s^2-1|}{|s|^2+1}, \qquad 
a_{\beta}=\frac{r_1(s) \cdot r_2(\bar{s})-r_1(\bar{s}) \cdot r_2(s)}{r_2(\bar{s})-r_2(s)}= a-\frac{2 |s|^2 \mathrm{Re} (s)}{|s|^2+1}.
\]\\ {\bf Automorphism with minimum entropy} : 
In \cite{M2}, McMullen constructed an automorphism with minimum positive entropy on a non-projective K3 surface. 
The construction is expressed as follows: 
First let $E_{10} \cong \oplus_{i=1}^{10} \mathbb{Z} e_{i}$ be a lattice with symmetric bilinear form $( \cdot. \cdot )_{E_{10}}$ 
corresponding to the Dynkin diagram of type $E_{10}$, 
and let $f_1 \in O(E_{10})$ be a lattice isomorphism corresponding to a Coxeter element in $W(E_{10})$. 
Moreover consider a new bilinear form $L_1 \cong \oplus_{i=1}^{10} \mathbb{Z} e_{i}$ given by $( x. y )_{L_{1}}:=-( a x. y )_{E_{10}}$, 
where $a:=2(f_1+f_1^{-1})+3 \in \mathrm{End}(E_{10})$. 
It is easily seen that $L_1$ is an even lattice with signature $(3,7)$ and $f_1 \in O(L_1)$ is a lattice isomorphism on $L_1$. 
A little calculation shows that the bilinear form 
$( \cdot. \cdot )_{L_{1}} : \mathbb{Z}^{10} \times \mathbb{Z}^{10} \ni (x,y) \mapsto {}^t x \Pi_{L_1} y \in \mathbb{Z}$ is given by 
\[
\Pi_{L_1} = \begin{pmatrix}
-2 & -2 & -2 & 1 & 2 & 0 & 0 & 0 & 0 & 2 \\
-2 & -2 & -1 & 2 & 0 & 0 & 0 & 0 & 0 & 2 \\
-2 & -1 & -2 & 1 & 2 & 0 & 0 & 0 & 0 & 0 \\
1 & 2 & 1 & -2 & -1 & 2 & 0 & 0 & 0 & -2 \\
2 & 0 & 2 & -1 & -2 & -1 & 2 & 0 & 0 & 0 \\
0 & 0 & 0 & 2 & -1 & -2 & -1 & 2 & 0 & 0 \\
0 & 0 & 0 & 0 & 2 & -1 & -2 & -1 & 2 & 0 \\
0 & 0 & 0 & 0 & 0 & 2 & -1 & -2 & -1 & 2 \\
0 & 0 & 0 & 0 & 0 & 0 & 2 & -1 & -2 & -1 \\
2 & 2 & 0 & -2 & 0 & 0 & 0 & 2 & -1 & -2 \\
\end{pmatrix}, 
\]
and the glue group $G(L_1) \cong (\mathbb{Z}/3\mathbb{Z})^2$ is generated by $(u_1,u_2)$ with 
\[
u_1:= \frac{1}{3} (2 e_1+ e_3+e_4+e_5+e_8+e_9), \quad 
u_2:= \frac{1}{3} (e_1+2 e_2+ 2e_3+e_5+e_6+e_9+e_{10}),  
\]
on which $f_1$ acts as $f_1(u_1)=u_2, f_1(u_2)=2u_1$. 

Next let $L_2:= A_2 \oplus A_2 \cong (\mathbb{Z} e_{11} \oplus \mathbb{Z} e_{12} ) \oplus (\mathbb{Z} e_{21} \oplus \mathbb{Z} e_{22})$ be a lattice with symmetric bilinear form 
$( \cdot. \cdot )_{L_2}$ corresponding to the Dynkin diagram of type $A_2 \oplus A_2$, 
and let $f_2 \in O(L_2)$ be a lattice isomorphism given by 
\[
\begin{pmatrix}
0 & 0 & 0 & 1 \\
0 & 0 & 1 & 0 \\
1 & 0 & 0 & 0 \\
0 & 1 & 0 & 0 \\
\end{pmatrix}. 
\]
Then $L_2$ is an even lattice of signature $(0,4)$, and its glue group $G(L_2) \cong (\mathbb{Z}/3\mathbb{Z})^2$ is generated by $(v_1,v_2)$ with 
\[
v_1:=\frac{1}{3}(e_{11}+2e_{12}), \quad v_2:=\frac{1}{3}(e_{21}+e_{22}), 
\]
on which $f_2$ acts as $f_2(v_1)=v_2, f_2(v_2)=2v_1$. 
Hence by the gluing map $\phi : G(L_1) \to G(L_2)$ given by $\phi(u_i)=v_i$, we obtain the even unimodular lattice $L_0$ with signature $(3,11)$ 
and the automorphism $f_0 \in O(L_0)$ associated to $f_1$ and  $f_2$. 
Finally, $f_0$ extends trivially to $f$ on the even unimodular lattice $L=L_0 \oplus E_8(-1)$ with signature $(3,19)$. 
McMullen \cite{M2} showed that $f$ preserves some K3 structure on $L$ and hence $f$ is realized by an automorphism on a K3 surface. 
\par
The characteristic polynomial of $f_1$ is given by $S(t)=t^{10}+t^9-t^7-t^6-t^5-t^4-t^3+t+1$ and 
the holomorphic $2$-form $\sigma$ corresponds to 
\[
\sigma(s):= (1+s+s^9)e_1+(1+s^8)e_2+(s^2+s^3+s^4+s^5+s^6+s^7-s^9)e_3+ \sum_{k=4}^{10} (\sum_{j=0}^{10-k}s^j) e_k,
\]
where $s \approx -0.9433+0.3319 \sqrt{-1} \in U(1)$ is given by $S(s)=0$ and $( \sigma(s).\sigma(\bar{s}) ) >0$. 
Then by putting $A_{\alpha,\beta}=e_2+5e_3+3e_4+4e_5+2e_6+e_7+2e_8$, 
$A_{\beta,\gamma}=e_6+e_8$, $A_{\gamma,\alpha}=e_2-e_1$, 
$B_{\gamma}=e_1-2e_2$, $B_{\alpha}=-e_9$, $B_{\beta}=e_3$, 
we have, up to scale, 
\[
\int_{A_{\alpha,\beta}} \sigma=a_{\beta}- \tau \cdot a_{\alpha}=r_1(s), \quad 
\int_{A_{\beta,\gamma}} \sigma=\tau=r_2(s), \quad 
\int_{A_{\gamma,\alpha}} \sigma=1
\]
with
\[
r_1(s)=-2-3s-s^2-s^3+s^6+3s^7+s^8-s^9, \qquad r_2(s)=s^3+s^4+s^5-s^8. 
\]
In particular, $(a_{\alpha},a_{\beta})$ is given by
\[
a_{\alpha}=\frac{r_1(s)-r_1(\bar{s})}{r_2(\bar{s})-r_2(s)} \approx 0.4179, \qquad 
a_{\beta}=\frac{r_1(s) \cdot r_2(\bar{s})-r_1(\bar{s}) \cdot r_2(s)}{r_2(\bar{s})-r_2(s)} \approx 0.6784.
\]
Note that $a_{\alpha}$ and $a_{\beta}$ are algebraic. 
Hence if $(a_{\alpha},a_{\beta}) \notin \mathbb{Q}^2$, then $(a_{\alpha},a_{\beta})$ satisfies the Diophantine condition. 

\begin{question}
Is the K3 surface realizable in our construction. In other words, does the condition $( \sigma(s).\sigma(\bar{s}) ) > \Lambda$ hold? 
\end{question}


\section{A relative variant of Arnol'd's Theorem}

\subsection{{Proof of Theorem \ref{thm:rel_Arnol'd}}}\label{section_thm:rel_Arnol'd}

In this section, we prove Theorem \ref{thm:rel_Arnol'd}. 
Here we use the notations in \S \ref{section:deform_points}. 
Fix a sufficiently fine open covering $\{U_j\}$ of $C_0$ with $\#\{U_j\}<\infty$ and a coordinate $z_j$ of $U_j$ such that $z_k=z_j+A_{kj}$ holds on each $U_{jk}$ for some constant $A_{kj}\in\mathbb{C}$. 
As $C_0$ is an elliptic curve, we can take such coordinates by considering those induced by the natural coordinate of the universal cover $\mathbb{C}$. 
Fix also another open covering $\{U_j^*\}$ of $C_0$ with $\#\{U_j^*\}=\#\{U_j\}$ such that $U_j^*\Subset U_j$ for each $j$. 
In the following proof of Theorem \ref{thm:rel_Arnol'd}, we use the following: 

\begin{lemma}[{\cite[Lemma 4]{U}}]\label{lem:ueda83_lemma4}
Let $M$ be a compact complex manifold, $\mathcal{U}=\{U_j\}_{j=1}^N$ an open covering of $M$, 
and $\mathcal{U}^*=\{U_j^*\}_{j=1}^N$ be an open covering of $M$ such that $U_j^*\Subset U_j$ for each $j$. 
Then there exists a positive constant $K=K(M, \mathcal{U}, \mathcal{U}^*)$ such that, for any flat line bundle $E\in \check{H}^1(\mathcal{U}, U(1))$ over $M$ and for any $0$-cochain $\mathfrak{f}\in \check{C}^0(\mathcal{U}, \mathcal{O}_M(E))$, the inequality 
\[
d(\mathbb{I}_M, E)\cdot \|\mathfrak{f}\|\leq K\cdot \|\delta\mathfrak{f}\|
\]
holds. 
\end{lemma}

Here we denote by $\|\mathfrak{f}\|$ the value $\max_j\sup_{p\in U_j}|f_j(p)|$ for each element $\mathfrak{f}=\{(U_j, f_j)\}_j$ of $\check{C}^0(\mathcal{U}, \mathcal{O}_M(E))$, 
and by $\|\mathfrak{g}\|:=\max_{j, k}\sup_{p\in U_{jk}}|g_{jk}(p)|$ for each element $\mathfrak{g}=\{(U_{jk}, g_{jk})\}_{j, k}$ of $\check{C}^1(\mathcal{U}, \mathcal{O}_M(E))$. 

In what follows, we always assume that $T$ is a sufficiently small open ball centered at the base point $0\in T$. 
As $N_{\mathcal{C}/\mathcal{S}}\cong {\rm Pr}_1^*L$, there exists $t_{jk}\in U(1)$ for each  $j$ and $k$ such that 
$N_{\mathcal{C}/\mathcal{S}}^{-1}=[\{(U_{jk}\times T, t_{jk})\}]\in H^1(\mathcal{C}, \mathcal{O}^*_{\mathcal{C}})$. 
Take a neighborhood $\mathcal{V}_j$ of $U_j\times T$ and a defining function $w_j$ of $U_j\times T$ in $\mathcal{V}_j$. 
It is easily observed that we can choose $w_j$ such that $t_{jk}w_k=w_j+O(w_j^2)$ holds on each $\mathcal{V}_{jk}:=\mathcal{V}_j\cap\mathcal{V}_k$. 
By fixing a holomophic extension of the coordinate function $z_j$ on $U_j$ to $\mathcal{V}_j$, we first show the following lemma, which can be regarded as a relative variant of Ueda's theorem \cite[Theorem 3]{U} for elliptic curves. 

\begin{lemma}\label{lem:rel_ueda_thm}
By shrinking $T$ and $\mathcal{V}_j$'s if necessary, one can take $\{(\mathcal{V}_j, w_j)\}$ such that $t_{jk}w_k=w_j$ holds on each $\mathcal{V}_{jk}$. 
\end{lemma}

\begin{proof}
By shrinking $T$ if necessary, we can take a positive number $Q>0$ such that 
$\{(z_j, w_j, t)\in \mathcal{V}_j\mid z_j\in U_j\cap U_k^*, |w_j|\leq Q^{-1} \}\subset \mathcal{V}_{k}$ for each $j$ and $k$. 
Lemma \ref{lem:rel_ueda_thm} is shown by the same argument as in the proof of \cite[Theorem 3]{U}. 
We will construct a new defining function $u_j$ of $C_0\times T$ in $\mathcal{V}_j$ by solving a Schr\"oder type functional equation 
\begin{equation}\label{eq:func_eq}
w_j=u_j+\sum_{\nu=2}^\infty f_{j\mid \nu}(z_j, t)\cdot u_j^\nu
\end{equation}
on each $\mathcal{V}_j$, 
where the coefficient functions $\{f_{j\mid \nu}\}_{\nu=2}^\infty$ are constructed inductively just in the same manner as in \cite[\S 4.2]{U} so that the solution $u_j$ satisfies $t_{jk}u_k=u_j$ on each $\mathcal{V}_{jk}$ if exists. 
Note that the Ueda's obstruction classes automatically vanish in our configurations, since $H^1(C_0\times T, {\rm Pr}_1^*L^{-n})=0$ for each $n\geq 1$. 
Here we used the condition that $L$ is non-torsion. 
Moreover, by $H^0(C_0\times T, {\rm Pr}_1^*L^{-n})=0$, we have that each coefficient function $f_{j\mid \nu}$ is constructed uniquely as a holomorphic function on $U_j\times T$. 

Therefore, all we have to do is to show the existence of the holomorphic solution $u_j$ of the functional equation (\ref{eq:func_eq}). 
By the implicit function theorem, it is sufficient to construct a convergent majorant series $A(u_j)=u_j+\sum_{\nu=2}^\infty A_\nu\cdot u_j^\nu$ for the functional equation (\ref{eq:func_eq}). 
Such a majorant series $A(X)$ can be constructed by the same argument as in \cite[\S 4.6]{U} as the solution of the functional equation 
\[
\sum_{\nu=2}^\infty d(\mathbb{I}_{C_0}, L^{\nu-1})\cdot A_\nu X^\nu=K\cdot\frac{M\cdot A(X)^2}{1-M\cdot A(X)}, 
\]
where $K=K(C_0, \{U_j\}, \{U_j^*\})$ is the constant as in Lemma \ref{lem:ueda83_lemma4} and $M$ is a positive constant sufficiently larger than $Q$ and $\max_j\sup_{\mathcal{V}_j}|w_j|$. 
By Siegel's technique \cite{S} (see also \cite[Lemma 5]{U}), the solution $A(X)$ actually has a positive radius of convergence, which proves the lemma. 
\end{proof}

In what follows, we always take a defining function $w_j$ of $U_j\times T$ in $\mathcal{V}_j$ as in Lemma \ref{lem:rel_ueda_thm}. 
Next we will show the existence of a suitable extension of the coordinate function $z_j\colon U_j\times T\to\mathbb{C}$ to $\mathcal{V}_j$. 
For clarity, we will denote (not by $z_j$ as above, but) by $\zeta_j\colon \mathcal{V}_j\to \mathbb{C}$ the fixed extension of $z_j$ in what follows. 
We will show the following: 

\begin{lemma}\label{lem:rel_arnold_z}
By shrinking $T$ and $\mathcal{V}_j$'s if necessary, one can take a holomorphic function $\zeta_j\colon \mathcal{V}_j\to \mathbb{C}$ 
such that $\zeta_j|_{U_j\times T}=z_j$ holds on each $\mathcal{V}_j$ and $\zeta_k=\zeta_j+A_{kj}$ holds on each $\mathcal{V}_{jk}$. 
\end{lemma}

\begin{proof}
Fix a local projection $P_j\colon\mathcal{V}_j\to U_j\times T$ with $\pi|_{\mathcal{V}_j}={\rm Pr}_2\circ P_j$ for each $j$. 
We use a function $\zeta_j:=P_j^*z_j$ as an initial extension function of $z_j\colon U_j\times T\to\mathbb{C}$ on each $\mathcal{V}_j$. 
In what follows, we denote by $g(\zeta_j, t)$ the function $P_j^*g$ for a function $g\colon U_j\to \mathbb{C}$. 
Then the expansion of $\zeta_k$ by $w_j$ can be written as
\[
\zeta_k=A_{kj}+\zeta_j+f^{(1)}_{kj}(\zeta_j, t)\cdot w_j+f^{(2)}_{kj}(\zeta_j, t)\cdot w_j^2+\cdots 
\]
on each $\mathcal{V}_{jk}$. 
As in the proof of the previous lemma, we will construct a new extension $u_j$ of $z_j$ by the defining equation 
\begin{equation}\label{eq:func_eq_2}
\zeta_j=u_j+\sum_{\nu=1}^\infty F_j^{(\nu)}(\zeta_j, t)\cdot w_j^\nu
\end{equation}
on each $\mathcal{V}_j$. 

First we explain how to define $\{F_j^{(1)}\}_j$. 
By adding three equations 
\begin{eqnarray}
\zeta_k&=&A_{kj}+\zeta_j+f^{(1)}_{kj}(\zeta_j, t)\cdot w_j+O(w_j^2), \nonumber \\
\zeta_\ell&=&A_{k\ell}+\zeta_k+f^{(1)}_{\ell k}(\zeta_k, t)\cdot w_k+O(w_k^2), \nonumber 
\end{eqnarray}
and 
\[
\zeta_j=A_{\ell j}+\zeta_\ell+f^{(1)}_{j\ell}(\zeta_\ell, t)\cdot w_\ell+O(w_\ell^2) 
\]
on $\mathcal{V}_{jk\ell}$ and by considering the coefficients of $w_j$ in the both hands sides, we obtain the equality 
$f^{(1)}_{kj}(z_j, t)+f^{(1)}_{\ell k}(z_j, t)\cdot t_{jk}^{-1}+f^{(1)}_{j\ell}(z_j, t)\cdot t_{j\ell}^{-1}=0$ on $U_{jk\ell}\times T$. 
Therefore we have that $[\{(U_{jk}\times T, f^{(1)}_{kj})\}]$ defines an element of $H^1(C_0\times T, {\rm Pr}_1^*L^{-1})$, which is equal to zero as a group. 
Thus we can take a holomorphic function $F_j^{(1)}$ on each $U_j\times T$ such that 
$F_j^{(1)}-t_{jk}^{-1}\cdot F_k^{(1)}=-f_{kj}^{(1)}$ holds on each $U_{jk}\times T$. 
As it follows from $H^0(C_0\times T, {\rm Pr}_1^*L^{-n})=0$ that such functions are unique, it gives the definition of $\{F_j^{(1)}\}$. 
Note that, by using these functions $\{F_j^{(1)}\}$, it clearly holds that the solution $\{u_j\}$ of the functional equation  (\ref{eq:func_eq_2})  satisfies $u_k-u_j =A_{kj}+O(w_j^2)$ on each $\mathcal{V}_{jk}$ after fixing $\{F_j^{(\nu)}\}_{\nu\geq 2}$ in any manner. 

Next we explain how to define $F_j^{(\nu)}$ for $\nu>1$ inductively. 
Assume that $\{F_j^{(\nu)}\}$ are already determined for each $\nu\leq n$ so that the following inductive assumption is satisfied. 
\vskip1mm
{\bf(Inductive Assumption)$_n$}: The solution $\{u_j\}$ of the functional equation  (\ref{eq:func_eq_2})  satisfies $u_k-u_j =A_{kj}+O(w_j^{n+1})$ on each $\mathcal{V}_{jk}$ for any choice of $\{F_j^{(\nu)}\}_{\nu> n}$. 
\vskip1mm\noindent
Here we will construct $\{F_j^{(n+1)}\}$ such that {(Inductive Assumption)}$_{n+1}$ is satisfied. 
Let $v_j$ be the solution of 
\[
\zeta_j=v_j+\sum_{\nu=1}^n F_j^{(\nu)}(\zeta_j, t)\cdot w_j^\nu. 
\]
Then, as we have that 
\begin{eqnarray}\label{eq:2_2}
-A_{kj}+v_k+\sum_{\nu=1}^n F_k^{(\nu)}\cdot w_k^\nu &=&-A_{kj}+ \zeta_k 
=\zeta_j+\sum_{\nu=1}^{n+1}f^{(\nu)}_{kj}\cdot w_j^{\nu}+ O(w_j^{n+2}) \nonumber \\
&=& \left(v_j+\sum_{\nu=1}^n F_j^{(\nu)}\cdot w_j^\nu\right)+\sum_{\nu=1}^{n+1}f^{(\nu)}_{kj}\cdot w_j^{\nu}+ O(w_j^{n+2}), \nonumber
\end{eqnarray}
we obtain the equality 
\[
v_k+\sum_{\nu=1}^n F_k^{(\nu)}(\zeta_k, t)\cdot w_k^\nu=v_j+A_{kj}+\sum_{\nu=1}^n \left(F_j^{(\nu)}(\zeta_j, t)+f^{(\nu)}_{kj}(\zeta_k, t)\right)\cdot w_j^\nu
+f^{(n+1)}_{kj}(\zeta_k, t)\cdot w_j^{n+1}+ O(w_j^{n+2}). 
\]
The coefficient of $w_j^{\nu}$ in the expansion of the left hand side can be calculated to be equal to 
$F_k(z_k, t)\cdot t_{jk}^{-\nu}+h_{kj}^{(\nu)}(z_j, t)$, where we denote by $h_{kj}^{(\nu)}(z_j, t)$ the function $\sum_{\mu=1}^{\nu-1} H_{kj, (\nu-\mu)}^{(\mu)}(z_j, t)\cdot t_{jk}^{-\mu}$ defined by using the coefficient functions $H_{kj, \lambda}^{(\nu)}$'s of the expansion 
\[
F_k^{(\nu)}(\zeta_k, t)=
F_k^{(\nu)}(\zeta_k(\zeta_j, w_j, t), t)= F_k^{(\nu)}(z_k(z_j, 0, t), t) + \sum_{\lambda=1}^\infty H_{kj, \lambda}^{(\nu)}(\zeta_j, t)\cdot w_j^\lambda
\]
of $F_k^{(\nu)}$ by $w_j$ on $\mathcal{V}_{jk}$. 
Note that $h_{jk}^{(\nu)}(\zeta_j, t)$ is determined only from $\{F_{j}^{(\mu)}\}_{\mu<\nu}$ and does not depend on the choice of $\{F_{j}^{(\mu)}\}_{\mu\geq\nu}$. 
Therefore, we obtain from {(Inductive Assumption)}$_{n}$ the equation 
\[
v_k=v_j +A_{kj}+\left(-h_{kj}^{(n+1)}(\zeta_j, t)+f^{(n+1)}_{kj}(\zeta_j, t)\right)\cdot w_j^{n+1}+O(w_j^{n+2})
\]
on each $\mathcal{V}_{jk}$. 
By using this equation, it follows from just the same argument as in the definition of $\{F_j^{(1)}\}_j$ that 
there uniquely exists a holomorphic function $F_j^{(n+1)}$ on each $U_j\times T$ such that 
$F_j^{(n+1)}-t_{jk}^{-n-1}\cdot F_k^{(n+1)}=h_{kj}^{(n+1)}-f^{(n+1)}_{kj}$ holds on each $U_{jk}\times T$, by which we define $\{F_j^{(n+1)}\}$ 
(The assertion {(Inductive Assumption)}$_{n+1}$ is easily checked by construction). 

Finally we show the convergence of the right hand side of the equation (\ref{eq:func_eq_2}). 
We construct a convergent majorant series $A(X)=\sum_{\nu=1}^\infty A_\nu\cdot X^\nu$ for the series $\sum_{\nu=1}^\infty F_j^{(\nu)}(\zeta_j, t)\cdot X^\nu$. 
Take positive number $M$ such that $\max_j\sup_{\mathcal{V}_j}|\zeta_j|<M$. 
Assume that $\{A_\nu\}_{\nu\leq n}$ satisfies $\max_{j}\sup_{U_{j}\times T}|F^{(\nu)}_{j}|\leq A_\nu$. 
Then, from the Cauchy--Riemann equality, it holds on each $U_j\cap U_k^*$ that 
\[
|h_{kj}^{(n+1)}-f^{(n+1)}_{kj}|\leq |f^{(n+1)}_{kj}|+\sum_{\nu=1}^n |H_{kj, (n+1-\nu)}^{(\nu)}|
\leq MQ^{n+1}+\sum_{\nu=1}^n A_\nu Q^{n+1-\nu}, 
\]
of which the right hand side is equal to the coefficient of $X^{n+1}$ in the expansion of 
\[
M\sum_{\nu=1}^\infty Q^\nu X^\nu+\left(\sum_{\nu=1}^\infty A_\nu X^\nu\right)\cdot \left(\sum_{\lambda=1}^\infty Q^\lambda X^\lambda\right)
=\frac{Q\cdot (M+A(X))\cdot X}{1-QX}. 
\]
From this observation and Lemma \ref{lem:ueda83_lemma4}, it turns out that the series $A(X)$ defined by the functional equation 
\[
\sum_{n=1}^\infty d(\mathbb{I}_{C_0}, L^{n})\cdot A_nX^n=2K\cdot \frac{Q\cdot (M+A(X))\cdot X}{1-QX} 
\]
is a majorant series of the series $\sum_{\nu=1}^\infty F_j^{(\nu)}(\zeta_j, t)\cdot X^\nu$, 
where $K=K(C_0, \{U_j\}, \{U_j^*\})$ is the constant as in Lemma \ref{lem:ueda83_lemma4}. 
Thus it is sufficient to show the solution $A(X)$ has a positive radius of convergence. 

Define a new power series $B(X)=X+B_2X^2+B_3X^3+\cdots$ by $B(X):=X+X\cdot A(X)$ and 
$\widehat{B}(X)=X+\widehat{B}_2X^2+\widehat{B}_3X^3+\cdots$ by 
\[
\sum_{n=2}^\infty d(\mathbb{I}_{C_0}, L^{n-1})\cdot \widehat{B}_nX^n=2KQ\cdot \frac{(M+1)\cdot \widehat{B}(X)^2}{1-Q\widehat{B}(X)}. 
\]
By Siegel's technique \cite{S} (see also \cite[Lemma 5]{U}), it follows that $\widehat{B}(X)$ actually has a positive radius of convergence. 
As 
\[
\sum_{n=2}^\infty d(\mathbb{I}_{C_0}, L^{n-1})\cdot B_nX^n
=2KQ\cdot \frac{(M\cdot X+B(X)-X)\cdot X}{1-QX}, 
\]
we can show by the simple inductive argument that $\widehat{B}_\nu\geq B_\nu(=A_{\nu-1})$ for each $\nu\geq 2$, which proves the lemma.  
\end{proof}

\begin{proof}[Proof of Theorem \ref{thm:rel_Arnol'd}]
Take a coordinate system $(z_j, w_j, t)$ of each $\mathcal{V}_j$ such that 
$\{w_j\}$ is as in Lemma \ref{lem:rel_ueda_thm} and $\{\zeta_j\}$ is as in Lemma \ref{lem:rel_arnold_z}. 
Define a map $P\colon \bigcup_j\mathcal{V}_j\to C_0\times T$ by 
$p(\zeta_j, w_j, t):=(\zeta_j, t)\in U_j\times T$ on each $\mathcal{V}_j$, which is well-defined by Lemma \ref{lem:rel_arnold_z}. 
By regarding $w_j$'s as fiber coordinates, we can naturally regard $\bigcup_j\mathcal{V}_j$ as an open neighborhood of the zero-section of $N_{\mathcal{C}/\mathcal{S}}$, which proves the theorem. 
\end{proof}

\subsection{More generalized variant}\label{section:more_gen_rel_arnol'd_thm}

Theorem \ref{thm:rel_Arnol'd} can be shown not only in the case where $\mathcal{C}\cong C_0\times T$ and $\pi|_{\mathcal{C}}={\rm Pr}_2$ hold, but also in the case where $\pi|_{\mathcal{C}}\colon \mathcal{C}\to T$ is a proper holomorphic submersion whose fibers $C_t:=S_t\cap \mathcal{C}$ are elliptic curves: 

\begin{theorem}\label{thm:rel_Arnol'd_gen}
Let $\pi\colon \mathcal{S}\to T$ be a deformation family of complex surfaces over a ball in $\mathbb{C}^n$, 
and $\mathcal{C}\subset \mathcal{S}$ be a submanifold such that $\pi|_{\mathcal{C}}$ is a deformation family of smooth elliptic curves. 
Assume that $d(\mathbb{I}_{C_t}, N_{C_t/S_t}^n)$ does not depend on $t\in T$ for each $n$ and that the Diophantine condition $-\log d(\mathbb{I}_{C_t}, N_{C_t/S_t}^n) = O(\log n)$ holds as $n\to\infty$. 
Then, by shrinking $T$ if necessary, there exists a tubular neighborhood $\mathcal{W}$ of $\mathcal{C}$ in $\mathcal{S}$ 
which is isomorphic to a neighborhood of the zero section in $N_{\mathcal{C}/\mathcal{S}}$. 
\end{theorem}

Note that, we have to choose the invariant distance $d$ of each ${\rm Pic}^0(C_t)$ appropriately in order it to satisfy the condition that $d(\mathbb{I}_{C_t}, N_{C_t/S_t}^n)$ does not depend on $t\in T$ for each $n$. 
A typical example of the configuration is as follows. 

\begin{example}
Let $\tau(t)$ be a point in the upper half plane such that $C_t\cong \mathbb{C}/\langle1, \tau(t)\rangle$. 
By choosing $\tau(t)$'s appropriately, we may assume that $\tau$ is a holomorphic function. 
By regarding ${\rm Pic}^0(C_t)$ as $C_t$ via the isomorphism $C_0\ni p\mapsto \mathcal{O}_{C_t}(p-[0])\in {\rm Pic}^0(C_t)$, we define an invariant distance $d$ of each ${\rm Pic}^0(C_t)$ by 
\[
d([0], [\alpha+\beta\cdot \tau(t)]):=\min\{|\alpha|, |1-\alpha|\}+\min\{|\beta|, |1-\beta|\}
\]
for each $0\leq \alpha, \beta<1$, where $[z]$ is the image of $z\in\mathbb{C}$ by the covering map $\mathbb{C}\to \mathbb{C}/\langle1, \tau(t)\rangle\cong C_t$. 
Take two algebraic irrational numbers $\alpha$ and $\beta$. 
Define a divisor $\mathcal{D}$ of $\mathcal{C}$ by $\mathcal{D}\cap C_t= [\alpha+\beta\cdot\tau(t)]-[0]$. 
Then, if the normal bundle $N_{\mathcal{C}/\mathcal{S}}$ is the line bundle corresponding to $\mathcal{D}$, then $d(\mathbb{I}_{C_t}, N_{C_t/S_t}^n)$ does not depend on $t\in T$ for each $n$ and that the Diophantine condition $-\log d(\mathbb{I}_{C_t}, N_{C_t/S_t}^n) = O(\log n)$ holds as $n\to\infty$. 
\end{example}

We can prove Theorem \ref{thm:rel_Arnol'd_gen} by almost the same manner as the proof of Theorem \ref{thm:rel_Arnol'd}. 
The only difficulty is the $t$-dependence of the constant  $K$ as in Lemma \ref{lem:ueda83_lemma4}. 
In order to overcome this difficulty, we use the following: 

\begin{lemma}\label{lem:u_lemma4_improved}
Assume that each $U_j$ is a coordinate open ball. 
Then one can take a constant  $K=K(M, \mathcal{U}, \mathcal{U}^*)$ as in Lemma \ref{lem:ueda83_lemma4} such that $K$ depends only on the number $N=\#\mathcal{U}$ and the maximum of the radii of $U_j^*$'s calculated by using the Kobayashi metrics of $U_j$'s. 
\end{lemma}

Lemma \ref{lem:u_lemma4_improved} follows directly from the improved proof of Lemma \ref{lem:ueda83_lemma4} we will describe in \S \ref{section:nice_proof_of_uedas_lemma}, which the author learned from Prof. Tetsuo Ueda. 

\begin{proof}[Proof of Theorem \ref{thm:rel_Arnol'd_gen}]
Fix a sufficiently fine open covering $\{U_j\}$ of $C_0$ with $\#\{U_j\}<\infty$ and a coordinate $z_j$ of $U_j$ such that $z_k=z_j+A_{kj}$ holds on each $U_{jk}$ for some constant $A_{kj}\in\mathbb{C}$. We may assume that each $U_j$ is a coordinate open ball. 
Fix also another open covering $\{U_j^*\}$ of $C_0$ with $\#\{U_j^*\}=\#\{U_j\}$ such that $U_j^*\Subset U_j$ for each $j$. 
Then, by shrinking $T$ if necessary, we can regard $\mathcal{C}$ as a complex manifold which is obtained by patching $U_j\times T$'s (or $U_j^*\times T$'s) by using the coordinate transformations in the form of $z_k = z_j + A_{kj}(t)$, where $A_{kj}$ is a holomorphic function defined on $T$ with $A_{kj}(0)=A_{kj}$. 
It follows from Lemma \ref{lem:u_lemma4_improved} that the constant $K$ as in Lemma \ref{lem:ueda83_lemma4} can be taken as a constant which is independent of the parameter $t\in T$. 
Then we can carry out the same argument as in the previous subsection to obtain Theorem \ref{thm:rel_Arnol'd_gen}. 
\end{proof}

\subsection{An alternative proof of  Ueda's lemma with effective constant $K$}\label{section:nice_proof_of_uedas_lemma}

Here we describe a simple proof of Lemma \ref{lem:ueda83_lemma4}, which the author learned from Prof. Tetsuo Ueda. 
One of the most remarkable points in this proof is that the constant $K=K(M, \mathcal{U}, \mathcal{U}^*)$ can be described explicitly, which is needed in our purpose. 
Actually, we will construct the constant $K$ so that the inequality 
\[
K<1+2\cdot\left(\frac{2}{1-s}\right)^{N+2}
\]
holds, where $s$ is the maximum of the constants $s_j$'s in the following: 

\begin{lemma}\label{lem:ueda83_lem4_lemma_3}
Assume that each $U_j$ is a coordinate open ball. 
For each $j$, there exists a positive constant $s_j$ less than $1$ which satisfies the following assertion: 
For any holomorphic function $f\colon U_j\to \mathbb{C}$ with $\sup_{z\in U_j}|f(z)|<1$, if there exists a point $z_0\in U_j^*$ with $f(z_0)=0$, then it holds that $\sup_{z\in U_j^*}|f(z)|<s_j$. 
Moreover, we can take such $s_j$ so that it depends only on the radius of $U_j^*$ calculated by using the Kobayashi metric of $U_j$. 
\end{lemma}

\begin{proof}
Lemma follows from the Schwarz--Pick theorem-type property of the Kobayashi metric. 
\end{proof}

Set $L_1:=\frac{2s}{1-s}$ and $L_2:=\frac{1+s}{1-s}$. 
Then we have the following: 

\begin{lemma}\label{lem:ueda83_lem4_lemma_1}
For any holomorphic function $f\colon U_j \to \mathbb{C}$ with $\sup_{z\in U}|f(z)|<1$ and for any points $z_1, z_2\in U_j^*$, we have the inequalities 
$|f(z_1)-f(z_2)|\leq L_1\cdot (1-|f(z_1)|)$ and 
$1-|f(z_2)|\leq L_2\cdot (1-|f(z_1)|)$. 
\end{lemma}

\begin{proof}
Set $a:=f(z_1)$ and consider the M\"obius transformation $T(w):=\frac{w-a}{1-\overline{a}w}$. 
As $T\circ f\colon U_j\to \Delta$ maps the point $z_1\in U_j^*$ to $0$, it follows from Lemma \ref{lem:ueda83_lem4_lemma_3} that the modulus $|\zeta|$ of $\zeta:=T\circ f(z_2)$ is less than $s$ ($\Delta\subset\mathbb{C}$ is the unit disc). Therefore we have 
\begin{eqnarray}
|f(z_1)-f(z_2)|
= \left|a-\frac{\zeta+a}{1+\overline{a}\zeta}\right| 
= \frac{(1-|a|^2)|\zeta|}{|1+\overline{a}\zeta|}<\frac{(1+|a|)s}{|1+\overline{a}\zeta|}\cdot (1-|a|)<\frac{2s}{1-s}\cdot (1-|a|), \nonumber 
\end{eqnarray}
which proves the first inequality. 

The second inequality holds obviously when $a=0$ holds. 
When $a\not=0$, let us consider the constant $\alpha:=\frac{a}{|a|}$. Then, as it holds that 
$1=|\alpha-T^{-1}(\zeta)+T^{-1}(\zeta)|\leq |\alpha-T^{-1}(\zeta)|+|T^{-1}(\zeta)|$,  we have
\[
1-|f(z_2)|\leq |\alpha-T^{-1}(\zeta)|=\left|\frac{a}{|a|}-\frac{\zeta+a}{1+\overline{a}\zeta}\right| 
\leq \frac{|a|+|a|\cdot|\zeta|}{|1+\overline{a}\zeta|}\cdot (1-|a|)
\leq \frac{1+s}{1-s}\cdot (1-|a|), 
\]
which proves the second inequality. 
\end{proof}

Denote by $K_1$ the constant $L_1\cdot L_2\cdot(L_2+1)^N$, and by $K_2$ the constant $L_2\cdot(L_2+1)^N$. Then we have the following: 

\begin{lemma}\label{lem:ueda83_lem4_lemma_2}
For each $j$, points $p, p'\in U_j^*$, and any $0$-cochain $\mathfrak{f}=\{(U_j, f_j)\}_j\in \check{C}^0(\mathcal{U}, \mathcal{O}_M(E))$ with $\|\mathfrak{f}\|$=1, 
the inequalities
$|f_j(p)-f_j(p')|\leq K_1\cdot \|\delta\mathfrak{f}\|$
and 
$1-|f_j(p)|\leq K_2\cdot \|\delta\mathfrak{f}\|$ 
hold. 
\end{lemma}

\begin{proof}
Take a positive constant $\varepsilon$ (slightly) larger than $\|\delta\mathfrak{f}\|$. 
Then $|f_{j_0}(p_0)|>1-(\varepsilon - \|\delta\mathfrak{f}\|)$ holds for some $p_0\in U_{j_0}$. 
Take a chain of open sets $U_{j_1}^*, U_{j_2}^*, \dots, U_{j_m}^*$ such that $p_0\in U_{j_1}^*$ and that $U_{j_\mu}^*\cap U_{j_{\mu+1}}^*\not=\emptyset$ holds for each $1\leq \mu<m$. 
We shall show the following assertion: 
for each $p, p'\in U_{j_m}^*$, the inequality 
$|f_{j_m}(p)-f_{j_m}(p')|\leq L_1\cdot L_2\cdot(L_2+1)^{m}\cdot \varepsilon$ 
and 
$1-|f_{j_m}(p)|\leq L_2\cdot(L_2+1)^{m}\cdot \varepsilon$ hold. 
Note that, as any $U_j$ can be linked with $U_{j_0}$ by such a chain with length at most $N=\#\mathcal{U}$, Lemma \ref{lem:ueda83_lem4_lemma_2} follows from this assertion. 

The proof is by induction on $m$. 
First, we show the case of $m=1$. 
As $|f_{j_0}(p_0)|=|t_{j_1j_0}f_{j_0(p_0)}|\leq |t_{j_1j_0}f_{j_0}(p_0)-f_{j_1}(p_0)|+|f_{j_1}(p_0)|\leq \|\delta\mathfrak{f}\|+|f_{j_1}(p_0)|$ holds, it follows from Lemma \ref{lem:ueda83_lem4_lemma_1} that
\[
1-|f_{j_1}(p)|\leq L_2\cdot (1-|f_{j_1}(p_0)|)
\leq L_2\cdot (1-|f_{j_0}(p_0)|+\|\delta\mathfrak{f}\|)
<L_2\cdot ((\varepsilon - \|\delta\mathfrak{f}\|)+\|\delta\mathfrak{f}\|)=L_2\cdot \varepsilon
\]
holds for any $p\in U_{j_1}^*$. 
Thus the second inequality follows. 
By Lemma \ref{lem:ueda83_lem4_lemma_1}, 
\begin{eqnarray}
|f_{j_1}(p)-f_{j_1}(p')|&\leq& L_1\cdot (1-|f_{j_1}(p)|)\leq L_1\cdot L_2\cdot \varepsilon
<L_1\cdot L_2\cdot (L_2+1)\cdot \varepsilon
\nonumber 
\end{eqnarray}
holds for each $p, p'\in U_{j_1}^*$, from which we have the first inequality. 

Next we show the case of $m\geq 2$ by assuming the assertion for $\mu<m$. 
Fix a point $p_m\in U_{m-1}^*\cap U_m^*$ and take any $p\in U_m^*$. Then, by the inductive assumption and the inequality 
$|f_{j_{m-1}}(p_m)|\leq |t_{j_{m}j_{m-1}}f_{j_{m-1}}(p_m)-f_{j_{m}}(p_m)|+|f_{j_{m}}(p_m)|\leq \|\delta\mathfrak{f}\|+|f_{j_{m}}(p_m)|$, we have that 
\begin{eqnarray}
1-|f_{j_m}(p)|&\leq& L_2\cdot (1-|f_{j_m}(p_m)|)
\leq L_2\cdot (1-|f_{j_{m-1}}(p_m)|+\|\delta\mathfrak{f}\|)\nonumber \\
&\leq&L_2\cdot (L_2\cdot(L_2+1)^{m-1}\cdot \varepsilon+\|\delta\mathfrak{f}\|)
<L_2\cdot(L_2+1)^{m}\cdot \varepsilon, \nonumber 
\end{eqnarray}
from which the second inequality follows. 
The first inequality follows from this inequality and Lemma \ref{lem:ueda83_lem4_lemma_1}. 
\end{proof}

Set $K:=\max\{1+2K_1+2K_2, 2K_2\} (=1+2K_1+2K_2)$. 
We shall prove that this constant $K$ satisfies the property as in Lemma \ref{lem:ueda83_lemma4}. 
Here we will use the invariant distance $d$ as in \cite[\S 4.5]{U}: i.e. 
\[
d(\mathbb{I}_M, E):=\min_{\{(U_j, t_j)\}_j\in\check{C}^0(\mathcal{U}, U(1))}\max_{j, k}|t_{jk}\cdot t_k-t_j|, 
\]
where $\{t_{jk}\}\subset U(1)$ is such that $E=\{(U_{jk}, t_{jk})\}\in \check{Z}^1(\mathcal{U}, U(1))$. 
Note that $d(\mathbb{I}_M, E)\leq 2$ follows by definition for any $E$. 

We may assume that $\|\mathfrak{f}\|=1$. 
When $\|\delta\mathfrak{f}\|\geq K_2^{-1}$, we have that 
\[
d(\mathbb{I}_M, E)\cdot \|\mathfrak{f}\|\leq 2\leq 2K_2\cdot \|\delta\mathfrak{f}\|. 
\]
Therefore it is sufficient to show the Lemma by assuming that 
$\|\delta\mathfrak{f}\|< K_2^{-1}$. 
Take $t_{jk}\in U(1)$ such that $E=\{(U_{jk}, t_{jk})\}\in \check{Z}^1(\mathcal{U}, U(1))$ and fix points $q_j\in U_j^*$ and $q_{jk}\in U_j^*\cap U_k^*$. 
By the assumption and Lemma \ref{lem:ueda83_lem4_lemma_2}, we have that 
$1-|f_j(q_j)|\leq K_2\cdot \|\delta\mathfrak{f}\|< 1$. 
Therefore $f_j(q_j)\not=0$ for each $j$. 
Set $t_j^*:=\frac{f_j(q_j)}{|f_j(q_j)|}$. 
Then we have that 
\begin{eqnarray}
|t_{jk}t_k^*-t_j^*|&\leq& \left|t_{jk}\left(\frac{f_k(q_k)}{|f_k(q_k)|}-f_k(q_{k})\right)\right| 
+ \left|t_{jk}(f_k(q_k)-f_k(q_{jk}))\right| 
+ \left|t_{jk}f_k(q_{jk})-f_j(q_{jk})\right| \nonumber \\
&&+ \left|f_j(q_{jk})-f_j(q_j)\right| 
+ \left|f_j(q_{j})-\frac{f_j(q_j)}{|f_j(q_j)|}\right| \nonumber \\
&\leq&(1-|f_k(q_k)|)
+ \left|f_k(q_k)-f_k(q_{jk})\right| 
+ \|\delta\mathfrak{f}\| + \left|f_j(q_{jk})-f_j(q_j)\right| 
+ (1-|f_j(q_j)|)\nonumber 
\end{eqnarray}
holds. 
Thus Lemma follows from the definition of our invariant distance and Lemma \ref{lem:ueda83_lem4_lemma_2}. 
\qed

\appendix

\section{Miscellaneous remarks}

\subsection{The universal line bundle on $C\times {\rm Pic}^0(C)$}

Let $C$ be a smooth elliptic curve. Fix a base point $p\in C$. 
In this subsection, we identify ${\rm Pic}^0(C)$ with $C$ via the isomorphism $C\ni q\mapsto \mathcal{O}_C(q-p)\in {\rm Pic}^0(C)$. 
Denote by $D_1$ the prime divisor $\{(q, q)\in C\times C\mid q\in C\}$ and by $D_2$ the prime divisor $ \{p\}\times C$ of $C\times {\rm Pic}^0(C)=C\times C$. 
Set $\mathcal{L}:=\mathcal{O}_{C\times C}(D_1-D_2)$ and regard it as a line bundle on $C\times {\rm Pic}^0(C)$. 

\begin{proposition}\label{prop:coarse_moduli}
Let $T$ be a complex manifold and $\mathcal{N}$ be a holomorphic line bundle on $C\times T$. 
Assume that $\mathcal{N}|_{C\times \{t\}}$ is flat (i.e. $\mathcal{N}|_{C\times \{t\}}\in {\rm Pic}^0(C\times \{t\})$) for all $t\in T$. 
Then, there uniquely exists a holomorphic map $i\colon T\to {\rm Pic}^0(C)$ such that $({\rm id}_C\times i)^*\mathcal{L}=\mathcal{N}$. 
\end{proposition}

\begin{proof}
As the map $i$ needs to  map a point $t\in T$ to the point which corresponds to $\mathcal{N}|_{C\times \{t\}}$, the uniqueness is clear. 
Therefore, all we have to do is to show the existence of such a holomorphic map $i$. 
It is sufficient to construct this map $i$ by assuming $T$ is a sufficiently small open ball centered at $0\in \mathbb{C}^n$. 
In what follows, we denote by $C_t$ the submanifold $C\times\{t\}$ and by $N_t$ the line bundle $N_t:=\mathcal{N}|_{C_t}$ for each $t\in T$. 
Fix $q_0\in C$ such that $N_0=\mathcal{O}_{C_0}(q_0-p_0)$, where $p_0:=(p, 0)$. 
Consider the restriction map 
$H^0(C\times T, \mathcal{N}\otimes{\rm Pr}_1^*\mathcal{O}_C(p))\to H^0(C_0, N_0\otimes \mathcal{O}_{C_0}(p_0))=
H^0(C_0, \mathcal{O}_{C_0}(q_0))$, where ${\rm Pr}_1\colon C\times T\to C$ is the first projection. 
As it is easily observed, this map is surjective (Use, for example, Nadel's vanishing theorem to $H^1(C\times T, \mathcal{O}_{C\times T}(-C_0)\otimes \mathcal{N}\otimes{\rm Pr}_1^*\mathcal{O}_C(p))$). 
Therefore, there exists a holomorpchic section $F\colon C\times T\to \mathcal{N}\otimes{\rm Pr}_1^*\mathcal{O}_C(p)$ such that the zero divisor of $F|_{C_0}\colon C_0\to \mathcal{O}_{C_0}(q_0)$ is equal to $\{q_0\}$. 
This means that the zero divisor $D:={\rm div}(F)$ of $F$ transversally intersects $C_0$ at only the point $q_0$. 
Thus we may assume that $D$ is a prime divisor and transversally intersects $C_t$ at only one point, say $q_t\in C_t$, by shrinking $T$ if necessary. 
By the implicit function theorem, the map $t\mapsto q_t$ defines a holomorphic map $i\colon T\to C$. 
As 
it holds as divisors that $({\rm id}_C\times i)^*D_1=D$ and $({\rm id}_C\times i)^*D_2=\{p\}\times T$, the proposition follows. 
\end{proof}

We call this line bundle $\mathcal{L}$ {\it the universal line bundle} on $C\times{\rm Pic}^0(C)$. 

\subsection{The cohomology of the tangent bundle of a blow-up of $\mathbb{P}^2$ at general points}

Fix an integer $N\geq 4$ and distinct $N$ points $Z:=\{p_1, p_2, \dots, p_N\}$ in $\mathbb{P}^2$. 
Denote by $S$ the blow-up of $\mathbb{P}^2$ at $Z$. 
In this subsection, we compute the cohomology groups $H^q(S, T_S)$, where $T_S$ is the tangent bundle of $S$. 
By the simple computation, we obtain the short exact sequence $0\to \pi_*T_{S}\to T_{\mathbb{P}^2}\to j_*N_{Z/\mathbb{P}^2}\to 0$, where $\pi\colon S\to \mathbb{P}^2$ is the blow-up morphism and $j\colon Z\to \mathbb{P}^2$ is the inclusion. 
This short exact sequence induces the following long exact sequence
\begin{eqnarray}\label{eq:long_ex_S}
0&\to& H^0(\mathbb{P}^2, \pi_*T_S)\to H^0(\mathbb{P}^2, T_{\mathbb{P}^2})\to H^0(\mathbb{P}^2, j_*N_{Z/\mathbb{P}^2}) \\
&\to& H^1(\mathbb{P}^2, \pi_*T_S)\to H^1(\mathbb{P}^2, T_{\mathbb{P}^2}) \to 0\to H^2(\mathbb{P}^2, \pi_*T_S)\to H^2(\mathbb{P}^2, T_{\mathbb{P}^2}). \nonumber
\end{eqnarray}
From this exact sequence, we have the following: 

\begin{lemma}\label{lem:P^2_bup_cohomoogy}
Assume that $N\geq 4$ and $Z$ includes four points in which no three points are collinear. 
Then it holds that $H^0(S, T_S)=0, {\rm dim}\,H^1(S, T_S)=2N-8$, and $H^2(S, T_S)=0$. 
\end{lemma}

\begin{proof}
As it follows from Euler's short exact sequence that $H^0(\mathbb{P}^2, T_{\mathbb{P}^2})\cong \mathbb{C}^8$ and $H^q(\mathbb{P}^2, T_{\mathbb{P}^2})=0$ ($q>0$, note that ${\rm dim}\,H^0(\mathbb{P}^2, j_*N_{Z/\mathbb{P}^2})=2N$), one can deduce from the exact sequence (\ref{eq:long_ex_S}) that $H^2(S, T_S)=0$ and ${\rm dim}\,H^1(S, T_S)={\rm dim}\,H^0(S, T_S)+2N-8$ (Note that $H^0(\mathbb{P}^2, j_*N_{Z/\mathbb{P}^2})\cong \mathbb{C}^{2N}$. Note also that here we use the vanishing $R^q\pi_*T_S=0$ for each $q>0$ to see $H^q(\mathbb{P}^2, \pi_*T_S)\cong H^q(S, T_S)$). 
Again by the exact sequence (\ref{eq:long_ex_S}) , it is sufficient for proving $H^0(S, T_S)=0$ to show the restriction $H^0(\mathbb{P}^2, T_{\mathbb{P}^2})\to H^0(\mathbb{P}^2, j_*N_{Z/\mathbb{P}^2})$ is injective, 
which can be shown by a simple computation when $Z$ includes four points in which no three points are collinear. 
\end{proof}

\subsection{Comparison of distances on Picard varieties}\label{section:distance_Pic0}

In this subsection we will show the Lipschitz equivalence between Euclidean distance and Ueda's distance \cite[\S 4.5]{U} on the connected component $\mathcal{P}^0(C)$ of the group $H^1(C, {\rm U}(1))$ of unitary flat line bundles on a compact differentiable manifold $C$ which includes the trivial bundle $\mathbb{I}_C$ (Note that ${\rm Pic}^0(C)$ can be naturally identified with $\mathcal{P}^0(C)$ when $C$ is compact K\"ahler, see \cite[\S 1]{U} for example). 
Here we mean by Euclidean distance the induced distance $d_{\rm Euc}$ from a Euclidean metric on the vector space $H^1(C, \mathbb{R})$ via the exponential map $H^1(C, \mathbb{R})\to H^1(C, {\rm U}(1))$, whose image coincides with $\mathcal{P}^0(C)$ by universal coefficient theorem ($d_{\rm Euc}$ is {\it not} determined uniquely, however such distances are Lipschitz equivalent to each other). 
Ueda's distance is the invariant distance (in the sense of \cite[\S 4.1]{U}) which is determined by fixing a sufficiently fine finite open covering 
$\{U_j\}$ of $C$ and by letting 
\[
d_{\rm Ueda}(\mathbb{I}_C, F) := \inf\left\{\left.\max_{j, k}|1-t_{jk}|\,\right|\, [\{(U_{jk}, t_{jk})\}]=F\in \check{H}^1(\{U_j\}, {\rm U}(1))\right\}
\]
for each $\mathcal{P}^0(C)$. 

\begin{proposition}
For any compact differentiable manifold $C$ and a sufficiently fine finite open covering 
$\{U_j\}$ of $C$, $d_{\rm Ueda}$ is Lipschitz equivalent to $d_{\rm Euc}$: i.e. there exists a positive constant $K>0$ such that 
\[
\frac{1}{K}d_{\rm Euc}(E, F)\leq
d_{\rm Ueda}(E, F)\leq
Kd_{\rm Euc}(E, F)
\]
holds for any $E, F\in \mathcal{P}^0(C)$. 
\end{proposition}

\begin{proof}
As both of the distances are invariant, it is sufficient to show the existence of a constant $K>0$ by which the inequalities for $E=\mathbb{I}_C$ holds. 
Take loops $\gamma_1, \gamma_2, \dots, \gamma_r$ of $C$ such that the classes $[\gamma_1], [\gamma_2], \dots, [\gamma_r]$ generates $H^1(C, \mathbb{Z})/H^1(C, \mathbb{Z})_{\rm tor}$, where $H^1(C, \mathbb{Z})_{\rm tor}$ is the group of torsion elements of $H^1(C, \mathbb{Z})$. 
By choosing $\{U_j\}$ and $\gamma_\bullet$'s suitably, we may assume the following: for each $\lambda=1, 2, \dots, r$, there exists a positive integer $N_\lambda$ and a sequence $\{U_{j_{\lambda, \ell}}\}_{\ell=0}^{N_\lambda}$ of open sets such that $j_{\lambda,0}=j_{\lambda,N_\lambda}$ and that $\gamma_\lambda$ passes through $U_{j_{\lambda,0}}, U_{j_{\lambda,1}}, \dots, U_{j_{\lambda,N_\lambda}}$ in this order (For obtaining such $\{U_j\}$ and $\gamma_\bullet$, consider a triangulation of $C$ for example. Such $\{U_j\}, \{\gamma_\lambda\}$ can be constructed by letting each $U_j$ be the interior of the union of simplices which contains a vertex, and by taking $\gamma_\bullet$'s so that they consist of finite union of edges). Note that the image of the natural map $\mathcal{P}^0(C)\to H^1(C, {\rm U}(1))$ coincides with ${\rm Hom}_{\rm group}(\textstyle\bigoplus_{\lambda=1}^r\mathbb{Z}\cdot [\gamma_\lambda],\ {\rm U}(1))\ (\cong {\rm U}(1)^{\oplus r})$ via the isomorphism $H^1(C, {\rm U}(1))\cong {\rm Hom}_{\rm group}(H_1(X, \mathbb{Z}), {\rm U}(1))$. 
Note also that the map $\mathcal{P}^0(C)\to {\rm Hom}_{\rm group}(\textstyle\bigoplus_{\lambda=1}^r\mathbb{Z}\cdot [\gamma_\lambda],\ {\rm U}(1))\cong {\rm U}(1)^{\oplus r}$ can be interpreted as the map 
\[
\mathcal{P}^0(C) \ni F\to (\rho_F([\gamma_1]),\ \rho_F([\gamma_2]),\ \cdots \rho_F([\gamma_r]))\in {\rm U}(1)^{\oplus r}, 
\]
where $\rho_F$ is the monodromy of a flat line bundle $F$. 
As it holds from the definition of the monodromy that 
\[
\rho_F([\gamma_\lambda]) = \prod_{i=1}^{N_\lambda}t_{j_{\lambda, i-1}j_{\lambda, i}}
\]
for any $\{t_{jk}\}\subset {\rm U}(1)$ such that $[\{(U_{jk}, t_{jk})\}]=F\in \check{H}^1(\{U_j\}, {\rm U}(1))$, it holds for such $t_{jk}$'s that 
\[
\frac{1}{N_\lambda}\min_{m\in\mathbb{Z}}\left|\frac{\log \rho_F([\gamma_\lambda])}{2\pi\sqrt{-1}}-m\right|
\leq \max_{1\leq i\leq N_\lambda}\min_{m\in\mathbb{Z}}\left|\frac{\log t_{j_{\lambda, i-1}j_{\lambda, i}}}{2\pi\sqrt{-1}}-m\right|
\leq \frac{1}{4}\max_{1\leq i\leq N_\lambda}|1-t_{j_{\lambda, i-1}j_{\lambda, i}}|
\]
for each $\lambda=1, 2, \dots, r$ and $F\in \mathcal{P}^0(F)$. Therefore we have that 
\[
d_{\rm Euc}(\mathbb{I}_C, F)\leq \frac{\max_{1\leq \lambda\leq r}N_\lambda}{4}\cdot d_{\rm Ueda}(\mathbb{I}_C, F)\quad (F\in \mathcal{P}^0(C)). 
\]

For the other inequality, it is sufficient to show the existence of an open neighborhood $B$ of $\mathbb{I}_C$ in $\mathcal{P}^0(C)$ and a constant $K>0$ such that $d_{\rm Ueda}(\mathbb{I}_C, F)\leq K\cdot d_{\rm Euc}(\mathbb{I}_C, F)$ holds for any $F\in \mathcal{P}^0(C)$, since $\mathcal{P}^0(C)\setminus B$ is compact. Take a neighborhood $B$ and continuous functions $s_{jk}\colon B\to {\rm U}(1)$ as in Lemma \ref{lem:sjk_for_representing_euc_dist_on_P0_of_C}. 
Then, as 
\[
d_{\rm Ueda}(\mathbb{I}_C, F)\leq \max_{j, k}|1-s_{jk}(F)|
\]
follows from the definition of Ueda's distance, the proposition holds. 
\end{proof}

\begin{lemma}\label{lem:sjk_for_representing_euc_dist_on_P0_of_C}
There exists a neighborhood $B$ of $\mathbb{I}_C$ in $\mathcal{P}^0(C)$, 
continuous functions $s_{jk}\colon B\to {\rm U}(1)$, and a constant $K>0$ such that the following holds for any $F\in B$: \\
$(i)$ $s_{jk}(\mathbb{I}_C)=1$ for any $j, k$. \\
$(ii)$ $[\{(U_{jk}, s_{jk}(F))\}]=F\in \check{H}^1(\{U_j\}, {\rm U}(1))$. \\
$(iii)$ $\max_{j, k}|1-s_{jk}(F)| \leq K\cdot d_{\rm Euc}(\mathbb{I}_C, F)$. 
\end{lemma}

\begin{proof}
Take elements $\{(U_{jk}, a_{jk}^{(\lambda)})\} \in \check{Z}^1(\{U_j\}, \mathbb{R})$ for $\lambda=1, 2, \dots, r:={\rm dim}H^1(C, \mathbb{R})$ such that $[\{(U_{jk}, a_{jk}^{(\lambda)})\}]$'s are basis of $\check{H}^1(\{U_j\}, \mathbb{R})$. As the map $\check{H}^1(\{U_j\}, \mathbb{R}) \to \check{H}^1(\{U_j\}, {\rm U}(1))$, which is the map induced from the exponential map $\mathbb{R}\ni a\mapsto \exp(2\pi\sqrt{-1}a) \in{\rm U}(1)$, can be regarded as the universal covering of $\mathcal{P}^0(C)$, the image $B$ of the subset 
\[
\left\{\left.\left[\left\{\left(U_{jk}, \sum_{\lambda=1}^r\ell_\lambda\cdot a_{jk}^{(\lambda)}\right)\right\}\right]\in \check{H}^1(\{U_j\}, \mathbb{R})\right| -\varepsilon< \ell_1, \ell_2, \dots, \ell_r<\varepsilon\right\}
\]
of $\check{H}^1(\{U_j\}, \mathbb{R})$ ($0<\varepsilon<<1$) and the continuous map $s_{jk}\colon B\to {\rm U}(1)$ induced by 
\[
\left[\left\{\left(U_{jk}, \sum_{\lambda=1}^r\ell_\lambda\cdot a_{jk}^{(\lambda)}\right)\right\}\right]
\mapsto \exp\left(2\pi\sqrt{-1}\sum_{\lambda=1}^r\ell_\lambda\cdot a_{jk}^{(\lambda)}\right)
\]
satisfies the properties $(i)$, $(ii)$, and $(iii)$. 
\end{proof}



\end{document}